\documentclass[11pt, titlepage, reqno]{amsart}
\usepackage[T1]{fontenc}
\usepackage{lmodern}
\usepackage{amsmath,amssymb,mathtools}
\usepackage[margin=1in]{geometry}
\usepackage{microtype}
\usepackage{booktabs}
\usepackage{longtable}
\usepackage{xcolor}
\usepackage[hidelinks]{hyperref} %
\usepackage[capitalize,nameinlink,noabbrev]{cleveref}

\hypersetup{
  bookmarksdepth=2,
  bookmarksopen=true,
  bookmarksopenlevel=2
}

\hypersetup{%
	unicode=true,%
	colorlinks=true,%
	linkcolor=blue,%
	urlcolor=black,%
	citecolor=pink,%
	bookmarksopen=true,%
	bookmarksnumbered,%
	pdfpagemode=UseOutlines,%
	pdfstartview=Fit,%
	pdfpagelayout=OneColumn
  }

\newtheorem{theorem}{Theorem}[section]
\newtheorem{proposition}[theorem]{Proposition}
\newtheorem{lemma}[theorem]{Lemma}
\newtheorem{corollary}[theorem]{Corollary}
\theoremstyle{definition}
\newtheorem{definition}[theorem]{Definition}
\theoremstyle{remark}
\newtheorem{remark}{Remark}[section]
\numberwithin{equation}{section}

\allowdisplaybreaks[4]

\newcommand{\dd}{\,\mathrm{d}}
\newcommand{\R}{\mathbb R}
\newcommand{\U}{\mathcal U}
\newcommand{\N}{\mathcal N}
\newcommand{\norm}[1]{\lVert#1\rVert}
\newcommand{\av}[1]{\overline{#1}}
\DeclareMathOperator{\diver}{div}
\DeclareMathOperator{\spec}{spec}

\title[Radial equilibria and concentration thresholds]{Unstable radial equilibria as sharp concentration thresholds for the J\"ager--Luckhaus model on the unit disk}

\author[X.~Mao]{Xuan Mao}
\address[X.~Mao]{School of Mathematics, Hohai University, Nanjing 211100, Jiangsu, China}
\address[X.~Mao]{Laboratory of Mathematical Modeling and Intelligent Computing for Water Systems, Hohai University, Nanjing 211100, China}
\email[X.~Mao]{maox@hhu.edu.cn}
\urladdr[X.~Mao]{https://orcid.org/0009-0001-4960-8901}
\thanks{Xuan~Mao is funded by
  ``the Fundamental Research Funds for the Central Universities'' (No.~B250201215), Basic Research Program of Jiangsu~(No.~BK20251482) and National Natural Science Foundation of China (No.~12601368).}

\author[W.~Wang]{Wanwan Wang}
\address[W.~Wang]{School of Mathematics and Information Science, Nanchang Hangkong University, Nanchang 330063, P. R. China}
\email[W.~Wang]{wwwang@nchu.edu.cn}
\thanks{Wanwan~Wang is partially supported by the National Natural Science Foundation of China (Nos. 12301252, 12461043 and 12461041)
  and Natural Science Foundation of Jiangxi Province (No. 20242BAB20003)
  and the Doctoral Research Foundation of Nanchang Hangkong University (No. EA202207184).}

\date{\today}

\begin{document}
\begin{abstract}

We identify unstable radial equilibria as sharp concentration thresholds for the J\"ager--Luckhaus chemotaxis system on the unit disk with no-flux boundary conditions.
Concentration is compared through cumulative mass.
At fixed total mass, every radial initial profile more concentrated than an unstable equilibrium produces finite-time blow-up, whereas every nonstationary profile less concentrated than that equilibrium generates a global, uniformly bounded solution converging to the least concentrated radial equilibrium.
We also classify the radial equilibria and determine their nonlinear stability.
Let $\Lambda = j_{1,1}^{2}$, where $j_{1,1}$ is the first positive zero of the Bessel function $J_1$.
For $M\ne \Lambda\pi$,
the least concentrated equilibrium is the unique locally asymptotically stable equilibrium under mass-preserving radial perturbations.
It is spatially homogeneous for $M<\Lambda\pi$ and strictly increasing for $M>\Lambda\pi$.
Every other radial equilibrium, including the homogeneous state at $M=\Lambda\pi$, is nonlinearly unstable.
These results clarify how total population and initial spatial aggregation jointly determine collapse or bounded redistribution, and how stable peripheral enrichment can occur within the radial model.
\end{abstract}
\maketitle

\section{Introduction}

Paul Painlev\'e wrote, in translation, \cite[p.~2]{Painleve1900}:
\begin{quote}
\itshape
``It appeared that, between two truths of the real domain, the easiest
and shortest path quite often passes through the complex domain.''
\end{quote}
For nonlinear evolution equations, stationary solutions can similarly
guide the study of the dynamics.
We study the classification and stability of radial equilibria for the
parabolic--elliptic J\"ager--Luckhaus chemotaxis model \cite{Jaeger1992} on the unit disk.
At a fixed total mass, we determine which equilibria act as concentration
thresholds between finite-time blow-up and bounded relaxation for initial
data comparable with them in cumulative mass, and which equilibrium serves
as the attracting state.

Let $T_{\max}\in(0,\infty]$ denote the maximal existence time of the
classical solution. The system reads
\begin{equation}\label{eq:system}
\begin{cases}
u_t=\Delta u-\diver(u\nabla v),&(x,t)\in\Omega\times(0,T_{\max}),\\
0=\Delta v-\mu+u,&(x,t)\in\Omega\times(0,T_{\max}),\\
\partial_\nu u=\partial_\nu v=0,&(x,t)\in\partial\Omega\times(0,T_{\max}),\\
u(\cdot,0)=u_0,&x\in\Omega,
\end{cases}
\qquad
\mu\coloneq\frac{M}{\pi},\quad M\coloneq\int_\Omega u_0\dd x>0,
\end{equation}
in $\Omega\coloneq B_1(0)\subset\R^2$.
We normalize $v$ by imposing $\int_\Omega v(\cdot,t)\dd x=0$ for every $t\in[0,T_{\max})$.
Unless otherwise stated, $u_0\in C(\overline\Omega)$ is nonnegative and
radially symmetric.
The maximal classical solution preserves mass and satisfies the usual
extensibility criterion:
\begin{equation*}
\text{if }T_{\max}<\infty,\quad\text{then}\quad
\limsup_{t\uparrow T_{\max}}\norm{u(t)}_\infty=\infty.
\end{equation*}
See, e.g., \cite[Lemma~2.1]{Winkler2019}.
Set
\begin{equation*}
\Lambda\coloneq j_{1,1}^2 > 8,
\end{equation*}
where $J_\nu$ is the Bessel function of the first kind of order $\nu$,
given by the power series
\[
J_\nu(z)=\sum_{m=0}^\infty
\frac{(-1)^m}{m!\,(m+\nu)!}
\left(\frac z2\right)^{2m+\nu}\qquad\text{for }\nu\in\{0,1\},
\]
and $j_{\nu,k}$ denotes its $k$th positive zero.
The constant equilibrium of mass $M$ is $(u,v)=(\mu,0)$.

\subsection*{Radial equilibria: classification and concentration order}

Stationary chemotaxis systems have been
studied using bifurcation, variational, and concentration methods,
see, for example, \cite{Schaaf1985,Lin1988}.
For the classical model with linear sensitivity and signal degradation,
the flux relation $u=Ce^v$ reduces the normalized stationary problem to
\[
-\Delta v+v=\lambda e^v\quad\text{in }D,\qquad
\partial_\nu v=0\quad\text{on }\partial D,
\]
where $D\subset\R^2$ is a smooth bounded domain and $\lambda>0$.
The cell mass is $M=\lambda\int_D e^v\dd x$.
Senba--Suzuki \cite[Theorem~2]{Senba2000} and
Wang--Wei \cite[Theorem~1.1]{Wang2002} obtained nonconstant equilibria,
allowing nonradial profiles, under the conditions
\[
4\pi<M<(1+\lambda_1(D))|D|,\qquad
M\notin\{4\pi m:m=1,2,\ldots\},
\]
where $\lambda_1(D)$ is the first positive Neumann eigenvalue of
$-\Delta$ on $D$.
For prescribed mass on the disk, Wang--Wang--Yang
\cite[Theorem~1.1 and Remark~1.1]{Wang2019} proved that the constant
is the only radial equilibrium for $0<M\le8\pi$.
Further constructions produce point-concentrating and layered
nonconstant families, including solutions with unbounded mass
\cite{Pino2006,Pistoia2015,Pino2016,Bonheure2021}.

For the J\"ager--Luckhaus system \eqref{eq:system}, the flux relation yields the
normalized mean-field problem
\[
u=\frac{Me^v}{\int_\Omega e^v\dd x},\qquad
-\Delta v=M\left(\frac{e^v}{\int_\Omega e^v\dd x}-\frac1\pi\right)
\quad\text{in }\Omega,
\]
with $\partial_\nu v=0$ and $\int_\Omega v\dd x=0$.
Horstmann--Lucia \cite{Horstmann2011} proved that the constant is the only radial
equilibrium on a disk for $0<M\le8\pi$.
Mao--Liu--Li \cite{Mao2026a} subsequently gave an alternative
proof using the cumulative mass equation.
The classification below determines the supercritical mass ranges for
existence of nonconstant radial equilibria, the number and uniqueness
of the monotone equilibria at each prescribed mass, and the exact onset
of non-monotone equilibria. We also determine their concentration order
and the endpoint behavior of the decreasing family.

\begin{theorem}[Number and monotonicity of nonconstant radial equilibria]\label{thm:equilibria}
For every mass $M>0$, the unique constant equilibrium of \eqref{eq:system}
is $(\mu,0)$, where $\mu=M/\pi$.
The nonconstant nonnegative classical radial equilibria are classified
by their mass as follows:
\begin{enumerate}
\item If $0<M\le8\pi$, there are no nonconstant radial equilibria.
\item If $8\pi<M<\Lambda\pi$, there is exactly one nonconstant radial
equilibrium, which is strictly decreasing and is denoted by
$(u_M^{\mathrm d},v_M^{\mathrm d})$.
\item If $M=\Lambda\pi$, there are no nonconstant radial equilibria.
\item If $M>\Lambda\pi$, there is exactly one nonconstant monotone radial
equilibrium, which is strictly increasing and is denoted by
$(u_M^{\mathrm i},v_M^{\mathrm i})$.
\item Non-monotone radial equilibria exist if and only if
$M\ge M_{\mathrm{nm}}$, where
$M_{\mathrm{nm}}>\Lambda\pi$ is characterized in
\cref{thm:non-monotone-threshold}.
\end{enumerate}
Here and below, monotonicity refers to the radial density on $0<r<1$,
and constant profiles are regarded as monotone.
The superscripts $\mathrm d$ and $\mathrm i$ indicate decreasing and
increasing equilibria, respectively; the subscript $M$ denotes mass.
In particular, no non-monotone radial equilibria exist for $M\le\Lambda\pi$.
\end{theorem}

\begin{remark}
\label{rem:decreasing-endpoints}
The decreasing equilibria link the two critical masses through their
endpoint limits:
\[
u_M^{\mathrm d}\to\Lambda\quad\text{as }M\uparrow\Lambda\pi,\qquad
u_M^{\mathrm d}(x)\dd x\rightharpoonup8\pi\delta_0
\quad\text{as }M\downarrow8\pi.
\]
The convergence to the homogeneous state is smooth, and the concentration
limit is understood weakly as a limit of measures; see \cref{thm:branch}.
\end{remark}

To compare the concentration of radial equilibria, we follow the
cumulative comparison approach of J\"ager--Luckhaus \cite{Jaeger1992}
and introduce the mass function
\[
F_f(\xi)\coloneq\int_{B_{\sqrt\xi}}f(x)\dd x
=2\pi\int_0^{\sqrt\xi}rf(r)\dd r,\qquad\text{for }\xi\in[0,1],
\]
for any nonnegative radial density $f\in L^1(\Omega)$.
The following cumulative concentration comparison has been used in \cite[Definition~5.1]{Kim2012} and
\cite[Definition~1.1]{Winkler2019}.

\begin{definition}[Concentration order]
  Let $f$ and $g$ be nonnegative radial densities in $L^1(\Omega)$.
    We say that $f$ is \emph{more concentrated} than $g$ in the cumulative sense, and write $f\succeq g$, if
$F_f(\xi)\ge F_g(\xi)$ for every $\xi\in[0,1]$, and write
$f\preceq g$ for the reverse inequality.
The additional condition $f\ne g$ excludes identical profiles.
\end{definition}

\begin{theorem}[Concentration order of radial equilibria]
\label{thm:least-concentrated}
Let $M>\Lambda\pi$, and let $(u_M^{\mathrm i},v_M^{\mathrm i})$ be the unique increasing equilibrium
of \cref{thm:equilibria}. Every nonnegative classical radial
equilibrium $(u,v)$ of the same mass satisfies
\begin{equation*}
F_u(\xi)\ge F_{u_M^{\mathrm i}}(\xi)\qquad\text{for }\xi\in[0,1].
\end{equation*}
If $u\not\equiv u_M^{\mathrm i}$, the inequality is strict for every $0<\xi<1$.
Thus $(u_M^{\mathrm i},v_M^{\mathrm i})$ is the unique least element in
the concentration order. All other radial equilibria of mass $M$ are
pairwise incomparable: for any two distinct such equilibria
$(u_1,v_1)$ and $(u_2,v_2)$, both different from
$(u_M^{\mathrm i},v_M^{\mathrm i})$, the difference $F_{u_1}-F_{u_2}$
takes both positive and negative values on $(0,1)$.
\end{theorem}

\begin{remark}
For $8\pi<M<\Lambda\pi$, the strict radial decrease of $u_M^{\mathrm d}$
gives
\begin{equation*}
F_{u_M^{\mathrm d}}(\xi) > F_\mu(\xi)=M\xi\qquad\text{for }\xi\in(0,1).
\end{equation*}
Thus the decreasing equilibrium is more concentrated than the constant
equilibrium of the same mass.
Together with \cref{thm:equilibria,thm:least-concentrated}, this identifies
the least concentrated radial equilibrium at every mass $M>0$.
For the subsequent dynamical statements, set
\begin{equation}\label{eq:least-reference-equilibrium}
\begin{gathered}
(\underline u_M,\underline v_M)\coloneq
\begin{cases}
(\mu,0),&0<M\le\Lambda\pi,\\
(u_M^{\mathrm i},v_M^{\mathrm i}),&M>\Lambda\pi,
\end{cases}\\
\underline F_M(\xi)\coloneq F_{\underline u_M}(\xi)\qquad \text{for }\xi\in[0,1].
\end{gathered}
\end{equation}
\end{remark}

We next use these stationary profiles as references for blow-up criteria.

\subsection*{Finite-time blow-up: mass thresholds and concentration criteria}

Due to the pioneering work of Childress--Percus~\cite{Childress1981}, J\"ager--Luckhaus~\cite{Jaeger1992} and Nagai \cite{Nagai1995},
it is well known that
the competition between diffusion and chemotactic attraction in two
dimensions led to a threshold population for chemotactic collapse:
all radial solutions with $M<8\pi$ remain globally bounded,
 while finite-time blow-up data exist for every $M>8\pi$.
Mao--Liu--Li \cite{Mao2026a} established global boundedness and
exponential stabilization also at the critical mass $M=8\pi$.
Thus $8\pi$ separates universal radial boundedness from the possibility
of finite-time collapse.

For supercritical masses, several criteria quantify the additional
concentration needed for blow-up.
Nagai \cite{Nagai1995, Nagai2001} showed, for each
$M>8\pi$, blow-up when the initial second moment about any prescribed interior point is sufficiently small.
In our radial setting, it reads
\begin{equation*}
\mathcal I(u_0)\coloneq \int_\Omega |x|^2u_0(x)\dd x \leq \frac{M-8\pi}{2}.
\end{equation*}
Cie\'slak--Lauren\c{c}ot
\cite{Cieslak2009}
developed a nonlinear virial criterion to detect
blow-up when $M>4\pi(p+1)$ and the generalized moment
\begin{equation*}
\mathcal J_p(u_0)\coloneq \frac1p\int_0^1
\left(\int_r^1\rho u_0(\rho)\dd\rho\right)^p r\dd r,
\qquad p\in(1,\infty),
\end{equation*}
is sufficiently small.
Winkler\cite{Winkler2019} established a spatially localized alternative: for $8\pi<m_0<M$ there is
$r_0=r_0(m_0,M)\in(0,1)$ such that the solution blows up in finite time whenever
\begin{equation*}
F_{u_0}(r_0^2)>m_0.
\end{equation*}
Mao--Liu--Li \cite{Mao2025} used concentrating subsolutions to obtain
a cumulative mass criterion for finite-time blow-up at every $M>8\pi$.

The present paper adds an eigenfunction-weighted moment criterion
relative to the unstable radial equilibria in \cref{thm:equilibria}.

\begin{theorem}[Eigenfunction-weighted moment criterion for blow-up]
\label{thm:eigenfunction-blowup}
Let $(u_*,v_*)$ be a radial equilibrium of mass $M$ belonging to one
of the following classes:
\begin{enumerate}
\item the decreasing equilibrium $(u_M^{\mathrm d},v_M^{\mathrm d})$
when $8\pi<M<\Lambda\pi$;
\item the constant equilibrium $(\mu,0)$ when $M\ge\Lambda\pi$;
\item any non-monotone radial equilibrium, necessarily with
$M\ge M_{\mathrm{nm}}>\Lambda\pi$.
\end{enumerate}
Set $F_*\coloneq F_{u_*}$ and $H_*(\xi)\coloneq F_*(\xi)-M\xi$
for $\xi\in[0,1]$.
Let $\psi_*$ be the positive principal eigenfunction of the adjoint
cumulative linearization, satisfying
\begin{equation*}
4\xi\psi_*''+\left(8-\frac{H_*}{\pi}\right)\psi_*'
+\mu\psi_*=\sigma_*\psi_*\quad\text{for }\xi\in(0,1),\qquad \psi_*(1)=0,
\end{equation*}
with regularity at $\xi=0$.
Then $\psi_*>0$ on $[0,1)$ and $\psi_*'<0$ on $[0,1]$.
For a nonconstant reference equilibrium, $\sigma_*>0$.
For the constant reference equilibrium $u_*\equiv\mu$, with
$M\ge\Lambda\pi$, one may take
\[
\psi_*(\xi)\coloneq\frac{J_1(j_{1,1}\sqrt\xi)}{\sqrt\xi}
\quad\text{for }\xi\in(0,1],\qquad
\psi_*(0)=\frac{j_{1,1}}2,\qquad
\sigma_*=\mu-\Lambda\ge0.
\]
Set
\[
c_*\coloneq\left(2\pi\int_0^1\frac{\psi_*^2}{-\psi_*'}\dd\xi\right)^{-1}>0.
\]
If a nonnegative radial datum $u_0\in C(\overline\Omega)$ of mass $M$
satisfies
\begin{equation}\label{eq:eigenfunction-blowup-criterion}
\mathcal A_*(u_0)\coloneq\int_0^1\psi_*(\xi)
\bigl(F_{u_0}(\xi)-F_*(\xi)\bigr)\dd\xi>0,
\end{equation}
then the corresponding solution of \eqref{eq:system} blows up in finite
time, with
\begin{equation}\label{eq:eigenfunction-blowup-time}
T_{\max}\le
\begin{cases}
\displaystyle\frac1{\sigma_*}
\log\left(1+\frac{\sigma_*}{c_*\mathcal A_*(u_0)}\right),
&\sigma_*>0,\\[2ex]
\displaystyle\frac1{c_*\mathcal A_*(u_0)},&\sigma_*=0.
\end{cases}
\end{equation}
The case $\sigma_*=0$ occurs precisely when $M=\Lambda\pi$ and
$u_*\equiv\Lambda$.
\end{theorem}

\begin{remark}
The eigenfunction-weighted moment method used by Souplet--Winkler
\cite[Proposition~3.4(i)]{Souplet2019} and subsequently by
Wang--Winkler \cite[Theorem~1.1(ii)]{Wang2026} is extended in
\cref{thm:eigenfunction-blowup} to every radially unstable equilibrium
of \eqref{eq:system} on the unit disk.
For each reference equilibrium, we choose the moment by pairing the
cumulative-mass deviation from that equilibrium with the positive
principal eigenfunction of its adjoint cumulative linearization.
\end{remark}

\begin{remark}
Condition \eqref{eq:eigenfunction-blowup-criterion} allows
$F_{u_0}-F_*$ to change sign, so it also covers cumulative profiles
that cross the reference equilibrium.
For a concrete family of blow-up data, fix any nonnegative radial
$\varphi\in C_c(\R^2)$ with
$\operatorname{supp}\varphi\subset B_1(0)$ and
$\int_{\R^2}\varphi\dd x=1$, and set
\[
u_{0,\varepsilon}(x)\coloneq\frac{M}{\varepsilon^2}
\varphi\!\left(\frac{x}{\varepsilon}\right)\qquad
\text{for }x\in\Omega\text{ and }\varepsilon\in(0,1).
\]
These are nonnegative continuous radial data of mass $M$, supported in
$\overline{B_\varepsilon(0)}$.
For every $0<\xi<1$, $F_{u_{0,\varepsilon}}(\xi)=M$ once
$\varepsilon^2\le\xi$. Dominated convergence therefore gives
\[
\lim_{\varepsilon\downarrow0}\mathcal A_*(u_{0,\varepsilon})
=\int_0^1\psi_*(\xi)\bigl(M-F_*(\xi)\bigr)\dd\xi>0,
\]
since $F_*(\xi)<M$ for $0\le\xi<1$.
Thus \cref{thm:eigenfunction-blowup} gives finite-time blow-up for all
sufficiently small $\varepsilon>0$.
\end{remark}

\subsection*{Long-time dynamics: convergence and stability}

A complementary line of research concerns attraction to, and instability
of, spatially homogeneous equilibria.
For Keller--Segel systems with signal degradation, Winkler
\cite{Winkler2010} developed self-mapping arguments for local exponential
stability of constant equilibria, which Cao \cite{Cao2015} extended to
critical Lebesgue spaces.
Moving beyond small background densities, Jiang \cite{Jiang2020} obtained local exponential
stability in these spaces using spectral methods.
For the stability--instability dichotomy of constant equilibria in the
Cauchy problem on $\R^n$, see \cite{Cygan2021,Ahn2026}.

For the J\"ager--Luckhaus coupling, the linearized density equation at
$(\mu,0)$, on the space of radial fixed-mass perturbations, is
\[
h_t=\Delta h+\mu h,\qquad
\partial_\nu h=0,\qquad \int_\Omega h\dd x=0.
\]
Its growth rates are $\mu-j_{1,k}^2$, $k\ge1$, so the radial spectrum
singles out the mass $\Lambda\pi$.
Recently, Wang--Winkler \cite{Wang2026} identified the sharp
secondary threshold for stability of the constant equilibrium observed by Winkler~\cite{Winkler2019}: for $M<\Lambda\pi$ the constant equilibrium is
locally exponentially stable under radial fixed-mass perturbations in
$L^\infty$, whereas for $M>\Lambda\pi$ every nonconstant radial datum
$u_0\succeq\mu$ produces finite-time blow-up.
The equality case $M=\Lambda\pi$ is included in \cref{thm:eigenfunction-blowup},
which shows that it is the minimal mass for the
loss of radial stability of the constant state.

These two lines of research distinguish the mass permitting collapse
from the mass at which homogeneity loses radial stability.
In particular, when $8\pi<M<\Lambda\pi$, sufficiently concentrated
data blow up even though all sufficiently small radial fixed-mass
perturbations of the constant state decay.
The equilibrium classification above identifies the stationary profiles
that separate ordered relaxation data from ordered blow-up data.
We now describe their roles in concentration thresholds, convergence,
and radial stability.

For ordered initial data, \cref{thm:eigenfunction-blowup}, together
with cumulative comparison, yields the following concentration
threshold property, beyond spectral stability.

\begin{corollary}[Equilibria as concentration thresholds]
\label{thm:concentration-thresholds}
Let $(u_*,v_*)$ be a radial equilibrium of mass $M$ in one of the
three classes of \cref{thm:eigenfunction-blowup}.
Then $(u_*,v_*)$ is a concentration threshold in the following sense:
for radial data $u_0$ of mass $M$, $u_0\succeq u_*$ with $u_0\ne u_*$
gives finite-time blow-up, whereas $u_0\preceq u_*$ generates a global
uniformly bounded solution.
Equality $u_0=u_*$ gives the stationary solution itself.
\end{corollary}

\begin{remark}
This dichotomy of concentration thresholds, especially for $8\pi<M<\Lambda\pi$,
quantitatively clarifies the need for additional initial concentration in earlier
blow-up criteria.
Indeed, \cref{rem:decreasing-endpoints} gives
$u_M^{\mathrm d}(x)\dd x\rightharpoonup8\pi\delta_0$ as
$M\downarrow8\pi$, so the threshold for ordered blow-up itself concentrates
at the origin as the mass approaches $8\pi$ from above.
\end{remark}

\begin{theorem}[Convergence for less concentrated data]
\label{thm:ordered-convergence}
Let $u$ solve \eqref{eq:system} with mass $M$. In each of the following
cases, the solution is global and uniformly bounded and satisfies
\begin{equation}\label{eq:increasing-limit}
\norm{u(t)-\underline u_M}_\infty
+\norm{v(t)-\underline v_M}_{W^{1,\infty}}\to0
\qquad\text{as }t\to\infty.
\end{equation}
\begin{enumerate}
\item $0<M\le8\pi$, with arbitrary nonnegative radial datum.
\item $8\pi<M<\Lambda\pi$ and
$u_0\preceq u_M^{\mathrm d}$, with $u_0\ne u_M^{\mathrm d}$.
\item $M=\Lambda\pi$ and $u_0\preceq\mu$.
\item $M>\Lambda\pi$ and $u_0\preceq u_*$, with $u_0\ne u_*$, where
$u_*=\mu$ or $u_*$ is any non-monotone radial equilibrium of mass $M$.
\end{enumerate}
\end{theorem}

Every threshold equilibrium in \cref{thm:concentration-thresholds}
is nonlinearly unstable under radial fixed-mass perturbations.
For $M\ne\Lambda\pi$, the attracting equilibrium has the following
stability properties.

\begin{corollary}[Stability of the least concentrated equilibrium]
\label{cor:stable-equilibria}
For each $M>0$ with $M\ne\Lambda\pi$, the equilibrium
$(\underline u_M,\underline v_M)$ in
\eqref{eq:least-reference-equilibrium} is locally asymptotically stable
under radial fixed-mass perturbations in $C(\overline\Omega)$ and is the
only radially stable equilibrium of that mass.
It is globally asymptotically stable when $M\le8\pi$.
At $M=\Lambda\pi$, the constant is the unique radial equilibrium and is
nonlinearly unstable.
\end{corollary}

We first prove the equilibrium classification in
\cref{sec:radial-equilibria}. In \cref{sec:qualitative-dynamics},
we introduce the cumulative mass equation, comparison principle, and
compactness of bounded orbits before establishing the blow-up,
convergence, and stability results. Sharp convergence rates and their
optimality are proved in \cref{sec:rates}.

\section{Radial equilibria: classification and concentration order}
\label{sec:radial-equilibria}

\subsection{Stationary reduction and sharp mass restrictions}
\label{sec:stationary-mass-bounds}

For a nonnegative classical radial equilibrium $u_*$ of positive mass,
the radial flux identity $[r(u_*'-u_*v_*')]'=0$ and regularity at zero
give $u_*'=u_*v_*'$. Hence $u_*=Ce^{v_*}$ with $C>0$, and elliptic
regularity makes both profiles smooth.
Its cumulative profile $F\coloneq F_{u_*}\in C^\infty([0,1])$ satisfies
\begin{equation}\label{eq:stationary}
4\xi F''+\frac{F-M\xi}{\pi}F'=0\quad\text{for }\xi\in(0,1),\qquad
F(0)=0,\quad F(1)=M,\quad F'>0.
\end{equation}
Put $H(\xi)\coloneq F(\xi)-M\xi$ for $\xi\in[0,1]$. Then
\begin{equation}\label{eq:H}
4\xi H''+\mu H+\frac{HH'}{\pi}=0\quad\text{for }\xi\in(0,1),\qquad H(0)=H(1)=0.
\end{equation}
In particular,
\begin{equation}\label{eq:stationary-radial-derivative}
u_*'(r)=-\frac{u_*(r)H(r^2)}{2\pi r},\qquad\text{for }r\in(0,1),
\end{equation}
and smoothness at zero gives $H(\xi)=O(\xi)$ and bounded $H'$.

\begin{lemma}[Cumulative sign and radial monotonicity]
\label{lem:stationary-sign}
For a nonconstant radial equilibrium, every interior zero of $H$ is
simple. If $H\ge0$ on $[0,1]$, then $H>0$ on $(0,1)$ and
$u_*'<0$ on $(0,1)$; if $H\le0$, then $H<0$ and $u_*'>0$ there.
Thus every weakly monotone nonconstant equilibrium is strictly
monotone in the interior, and both endpoint zeros of $H$ are simple.
\end{lemma}

\begin{proof}
If $H(\xi_0)=H'(\xi_0)=0$ at some $\xi_0\in(0,1)$, uniqueness
for \eqref{eq:H} implies $H\equiv0$, which is the constant equilibrium.
A one-signed nonzero $H$ therefore has no interior zero, and
\eqref{eq:stationary-radial-derivative} gives the asserted strict
monotonicity. Conversely, that same equation shows that a weakly
monotone density has a one-signed $H$.
For a decreasing equilibrium, mass normalization gives
$u_*(0)>\mu>u_*(1)$. Therefore
$H'(0)=\pi(u_*(0)-\mu)>0$ and
$H'(1)=\pi(u_*(1)-\mu)<0$: both endpoint zeros of $H$ are simple.
For an increasing equilibrium, all these inequalities are reversed.
\end{proof}

\begin{lemma}[The first radial Neumann eigenvalue]
The first nonzero eigenvalue of $-\Delta_N$ on radial mean-zero functions
on $B_1$ is $\Lambda=j_{1,1}^2>8$.
Here $\Delta_N$ denotes the Neumann Laplacian.
In particular,
\begin{equation}\label{eq:poincare}
\int_\Omega|\nabla f|^2\dd x\ge
\Lambda\int_\Omega f^2\dd x
\end{equation}
for every radial mean-zero $f\in H^1(\Omega)$.
\end{lemma}

This is the standard radial Neumann spectrum, with modes
$J_0(j_{1,k}r)$ and positive eigenvalues $j_{1,k}^2$.
The derivative identities $J_0'=-J_1$ and $(zJ_1(z))'=zJ_0(z)$
are given in \cite[\S~3.2, p.~45]{Watson1944}; the positive zeros
and their simplicity are treated in
\cite[Chapter~XV, pp.~477--480]{Watson1944}.

Define the positive cumulative adjoint weight
\begin{equation}\label{eq:psi0}
\psi_0(\xi)\coloneq\frac{J_1(j_{1,1}\sqrt\xi)}{\sqrt\xi}
\quad\text{for }\xi\in(0,1],\qquad
\psi_0(0)=\frac{j_{1,1}}2.
\end{equation}
We use the identities
\begin{equation}\label{eq:psi0-ode}
4\xi\psi_0''+8\psi_0'=-\Lambda\psi_0\quad\text{for }\xi\in(0,1),\qquad
\psi_0(1)=0.
\end{equation}
Moreover, $\psi_0>0$ on $[0,1)$ and $\psi_0'<0$ on $[0,1]$.
The Bessel identities and the interior sign properties are given in
\cite[Lemma~3.1]{Wang2026} for $n=2$ and $R=1$;
the endpoint derivatives are
$\psi_0'(0)=-j_{1,1}^3/16<0$ and
$\psi_0'(1)=\tfrac12j_{1,1}J_0(j_{1,1})<0$.

\begin{lemma}[Sharp weighted Poincar\'e inequality]\label{lem:weighted-poincare}
For every $q\in H^1_0((0,1))$,
\begin{equation}\label{eq:weighted-poincare}
4\int_0^1|q'|^2\dd\xi
\ge\Lambda\int_0^1\frac{q^2}{\xi}\dd\xi.
\end{equation}
The constant $\Lambda=j_{1,1}^2$ is sharp.
\end{lemma}

\begin{proof}
Set $\Phi(\xi)\coloneq\xi\psi_0(\xi)
=\sqrt\xi J_1(j_{1,1}\sqrt\xi)$ for $\xi\in[0,1]$.
By \eqref{eq:psi0-ode}, it satisfies
\[
\Phi>0\text{ on }(0,1),\qquad
\Phi(0)=\Phi(1)=0,\qquad -4\xi\Phi''=\Lambda\Phi.
\]
For $q\in C_c^\infty((0,1))$, integration by parts gives
\[
4\int_0^1|q'|^2\dd\xi
-\Lambda\int_0^1\frac{q^2}{\xi}\dd\xi
=4\int_0^1\Phi^2
\left|\left(\frac q\Phi\right)'\right|^2\dd\xi\ge0.
\]
For $q\in H^1_0((0,1))$, the estimates
\begin{equation}\label{eq:weighted-endpoint-control}
\begin{gathered}
\frac{|q(\xi)|^2}{\xi}\le\int_0^\xi|q'(s)|^2\dd s
\quad\text{for }\xi\in(0,1],\\
\int_0^\varepsilon\frac{q^2}{\xi}\dd\xi
\le\varepsilon\int_0^\varepsilon|q'|^2\dd\xi
\quad\text{for }\varepsilon\in(0,1]
\end{gathered}
\end{equation}
show continuity of the weighted norm on $H^1_0((0,1))$, so density
extends the inequality to this space.
Finally, $\Phi\in H^1_0((0,1))$, and multiplication of its equation by
$\Phi/\xi$ gives equality in \eqref{eq:weighted-poincare}.
\end{proof}

\begin{lemma}[Weighted stationary estimates]
Let $(u_*,v_*)$ be a nonnegative classical radial equilibrium of
\eqref{eq:system} with mass $M>0$, and put
$H\coloneq F-M\xi$ and $g\coloneq H_-\coloneq\max\{-H,0\}$.
Then
\begin{equation}\label{eq:stationary-moment}
(M-\Lambda\pi)\int_0^1\psi_0H\dd\xi
=\frac12\int_0^1\psi_0'H^2\dd\xi,
\end{equation}
and
\begin{equation}\label{eq:non-monotone-mass-excess}
(M-\Lambda\pi)\int_0^1\frac{g^2}{\xi}\dd\xi
\ge\frac13\int_0^1\frac{g^3}{\xi^2}\dd\xi\ge0.
\end{equation}
\end{lemma}

\begin{proof}
We first record two integration identities that will also be used
for time-dependent cumulative profiles. For any smooth $q$ with
$q(0)=q(1)=0$, two integrations by parts yield
\begin{align}
\int_0^1 4\xi\psi_0q''\dd\xi
&=\bigl[4\xi\psi_0q'\bigr]_0^1
-\int_0^1(4\psi_0+4\xi\psi_0')q'\dd\xi\notag\\
&=\bigl[4\xi\psi_0q'
-(4\psi_0+4\xi\psi_0')q\bigr]_0^1
+\int_0^1(8\psi_0'+4\xi\psi_0'')q\dd\xi\notag\\
&=-\Lambda\int_0^1\psi_0q\dd\xi.
\label{eq:weighted-linear-pairing}
\end{align}
Indeed, at $\xi=1$ the boundary terms vanish by $\psi_0(1)=q(1)=0$.
At $\xi=0$, smoothness gives $q(\xi)=O(\xi)$ and bounded $q'$,
while $\psi_0$ and $\psi_0'$ are bounded, so both boundary terms
also tend to zero. The last equality uses \eqref{eq:psi0-ode}.
These integrations may equivalently be carried out on $[\delta,1]$
and then passed to the limit $\delta\downarrow0$.
For the nonlinear term, a further integration by parts gives
\begin{equation*}
\int_0^1\psi_0qq'\dd\xi
=\frac12\bigl[\psi_0q^2\bigr]_0^1
-\frac12\int_0^1\psi_0'q^2\dd\xi
=-\frac12\int_0^1\psi_0'q^2\dd\xi.
\end{equation*}
Multiplying \eqref{eq:H} by $\psi_0$, integrating, and applying
these two identities with $q=H$ yields
\[
0=(\mu-\Lambda)\int_0^1\psi_0H\dd\xi
-\frac1{2\pi}\int_0^1\psi_0'H^2\dd\xi,
\]
which is \eqref{eq:stationary-moment}.

To prove \eqref{eq:non-monotone-mass-excess}, note that smoothness
of $H$ and $H(0)=H(1)=0$ give
$g\in H^1_0((0,1))$ and $g(\xi)=O(\xi)$,
which ensures finiteness of all the weighted integrals below.
Testing \eqref{eq:H} by $g/\xi$ gives the exact balance
\begin{equation}\label{eq:negative-part-balance}
\mu\int_0^1\frac{g^2}{\xi}\dd\xi
=4\int_0^1|g'|^2\dd\xi
+\frac1{3\pi}\int_0^1\frac{g^3}{\xi^2}\dd\xi.
\end{equation}
Indeed, the negative-part chain rule yields, almost everywhere,
$Hg=-g^2$, $H'g'=-|g'|^2$, and $HH'g=g^2g'$, while
\[
\int_0^1\frac{g^2g'}{\xi}\dd\xi
=\frac13\int_0^1\frac{(g^3)'}{\xi}\dd\xi
=\frac13\int_0^1\frac{g^3}{\xi^2}\dd\xi.
\]
These integrations can first be performed on $[\delta,1]$.
The endpoint terms at $\delta$ vanish as $\delta\downarrow0$, because
$H'(\delta)g(\delta)=O(\delta)$ and
$g(\delta)^3/\delta=O(\delta^2)$.
Combining \eqref{eq:negative-part-balance} with
\cref{lem:weighted-poincare} gives
\eqref{eq:non-monotone-mass-excess}.
\end{proof}

\begin{corollary}[Mass threshold and radial shape]
\label{cor:radial-shape}
Let $(u_*,v_*)$ be a nonnegative classical radial equilibrium of
\eqref{eq:system} with mass $M>0$.
For every nonconstant such equilibrium,
\[
\begin{aligned}
M<\Lambda\pi
&\quad\text{if and only if}\quad
u_*'(r)<0\text{ for every }r\in(0,1),\\
M>\Lambda\pi
&\quad\text{if and only if}\quad
u_*'(r_0)>0\text{ for some }r_0\in(0,1).
\end{aligned}
\]
At $M=\Lambda\pi$, the constant $(\Lambda,0)$ is the only equilibrium.
In particular, every non-monotone radial equilibrium has $M>\Lambda\pi$.
Moreover, $(M-\Lambda\pi)H\ge0$ on $[0,1]$ forces
$(u_*,v_*)=(\mu,0)$.
\end{corollary}

\begin{proof}
For a nonconstant equilibrium, \eqref{eq:stationary-moment} and
$\psi_0'<0$ give
\[
(M-\Lambda\pi)\int_0^1\psi_0H\dd\xi<0.
\]
This excludes $M=\Lambda\pi$ and proves the last assertion.
If $M<\Lambda\pi$, \eqref{eq:non-monotone-mass-excess} forces
$H_-=0$, so \cref{lem:stationary-sign} gives $u_*'<0$ on $(0,1)$.
If $M>\Lambda\pi$, the displayed inequality gives
$\int_0^1\psi_0H\dd\xi<0$. Since $\psi_0>0$ on $(0,1)$,
we have $H(\xi_0)<0$ for some $\xi_0\in(0,1)$, and
\eqref{eq:stationary-radial-derivative} yields $u_*'(\sqrt{\xi_0})>0$.
These alternatives prove both equivalences.
\end{proof}

\begin{remark}
Taking $q(\xi)=\xi(1-\xi)$ in \eqref{eq:weighted-linear-pairing},
so that $q''=-2$, yields
\[
\Lambda\int_0^1\xi(1-\xi)\psi_0\dd\xi
=8\int_0^1\xi\psi_0\dd\xi
>8\int_0^1\xi(1-\xi)\psi_0\dd\xi.
\]
Since $\psi_0>0$ on $(0,1)$, this proves $\Lambda>8$.
\end{remark}

\subsection{The shooting family and its first extrema}

Using the stationary flux identity from
\cref{sec:stationary-mass-bounds}, set
$z(r)\coloneq\log(u_*(r)/\mu)$ for $r\in[0,1]$.
The stationary problem becomes
\begin{equation}\label{eq:z}
z''+\frac1r z'+\mu(e^z-1)=0\quad\text{for }r\in(0,1),\qquad z'(0)=z'(1)=0.
\end{equation}
Integrating this equation gives $\int_\Omega e^z\dd x=\pi$.
Conversely, every solution of \eqref{eq:z} yields the equilibrium
$u_*\coloneq\mu e^z$, $v_*\coloneq z-\av z$, where $\av z\coloneq\pi^{-1}\int_\Omega z\dd x$.

For $\beta\ne0$, consider the parameter-free initial-value problem
\begin{equation}\label{eq:shooting}
Z_\beta''+\frac1sZ_\beta'+e^{Z_\beta}-1=0\quad\text{for }s\in(0,\infty),\qquad
Z_\beta(0)=\beta,\quad Z_\beta'(0)=0.
\end{equation}
The equation at the origin is understood through
\[
sZ_\beta'(s)=-\int_0^s t(e^{Z_\beta(t)}-1)\dd t.
\]
Its solutions are smooth and depend smoothly on $\beta$ on compact
$s$-intervals.

\begin{lemma}[First extremum]\label{lem:first-extremum}
For every $\beta\ne0$, $Z_\beta'$ has a first positive zero $R_\beta$.
If $\beta>0$, then
\[
Z_\beta'<0\text{ on }(0,R_\beta),\qquad
Z_\beta(R_\beta)<0,\qquad Z_\beta''(R_\beta)>0.
\]
If $\beta<0$, all these signs are reversed.
The map $\beta\mapsto R_\beta$ is smooth on each of the intervals
$(0,\infty)$ and $(-\infty,0)$.
\end{lemma}

\begin{proof}
The energy
\[
E_\beta(s)\coloneq\frac12(Z_\beta')^2+e^{Z_\beta}-Z_\beta-1,
\qquad\text{for }s\in[0,\infty),
\]
satisfies $E_\beta'=-s^{-1}(Z_\beta')^2\le0$ for $s>0$.
Since $e^z-z-1\to\infty$ at both ends of the real line, this proves
global existence and boundedness of $Z_\beta$ and $Z_\beta'$ on $[0,\infty)$.

For $\beta>0$, $Z_\beta'(s)<0$ as long as $Z_\beta>0$.
On that interval it solves
\[
Z_\beta''+\frac1sZ_\beta'+c(s)Z_\beta=0,\qquad
c(s)\coloneq\frac{e^{Z_\beta}-1}{Z_\beta}>1.
\]
The quotient defining $c$ extends continuously with value $1$ at
$Z_\beta=0$. We give the direct Sturm comparison argument~\cite[Theorem~3.1]{Simon2005}, implying a finite first zero of $Z_\beta$.
Let $j\coloneq j_{0,1}$ be the first
positive zero of $J_0$. The function $J\coloneq J_0$ satisfies
\[
J''+\frac1sJ'+J=0,\qquad J(0)=1,\quad J'(0)=0,
\]
and $J>0$ on $[0,j)$, $J(j)=0$, and $J'(j)<0$, since this zero is
simple. Suppose that $Z_\beta$ has no zero in $(0,j)$. Then
$Z_\beta>0$ there and $Z_\beta(j)\ge0$. Multiplying the equation for
$Z_\beta$ by $sJ$ and that for $J$ by $sZ_\beta$, and subtracting, gives
the weighted Wronskian
\begin{align*}
\mathcal W(s)&\coloneq s(Z_\beta'J-Z_\beta J'),\\
\mathcal W'(s)&=-s(c(s)-1)Z_\beta J<0\quad\text{for }s\in(0,j).
\end{align*}
The integral identity at the origin yields
\[
Z_\beta'(s)=-\frac{e^\beta-1}{2}s+o(s),\qquad
J'(s)=-\frac{s}{2}+O(s^3),
\]
so $\mathcal W(s)=O(s^2)\to0$ as $s\downarrow0$. Integrating from
$\varepsilon$ to $j$ and letting $\varepsilon\downarrow0$ therefore gives
\[
\mathcal W(j)=-\int_0^j s(c(s)-1)Z_\beta(s)J(s)\dd s<0.
\]
On the other hand, $\mathcal W(j)=-jZ_\beta(j)J'(j)\ge0$, a contradiction.
Thus the first positive zero $s_\beta^0$ of $Z_\beta$ exists and satisfies
\[
0<s_\beta^0<j_{0,1}.
\]
Moreover, the integral identity gives $Z_\beta'(s_\beta^0)<0$, so
$Z_\beta$ crosses zero there.
After $Z_\beta$ becomes negative, a derivative that remained negative
would give, for some $s_0$ and $c_0>0$,
\[
(sZ_\beta')'=s(1-e^{Z_\beta})\ge c_0s,\qquad\text{for }s\in[s_0,\infty),
\]
a contradiction. Its first zero is therefore a strict negative minimum.

For $\beta<0$, $Z_\beta'>0$ while $Z_\beta<0$.
Before its first zero, $\beta\le Z_\beta<0$ and
\[
Z_\beta''+\frac1sZ_\beta'+c(s)Z_\beta=0,
\]
where
\[
\begin{aligned}
c(s)&\coloneq\frac{1-e^{Z_\beta(s)}}{-Z_\beta(s)}
=\int_0^1 e^{tZ_\beta(s)}\dd t\\
&\ge\int_0^1 e^{t\beta}\dd t
=a_\beta\coloneq \frac{1-e^\beta}{-\beta}>0.
\end{aligned}
\]
The appropriate comparison function is now
$J_\beta(s)\coloneq J_0(\sqrt{a_\beta}\,s)$, which satisfies
\[
J_\beta''+\frac1sJ_\beta'+a_\beta J_\beta=0,\qquad
J_\beta(0)=1,\quad J_\beta'(0)=0.
\]
Its first positive zero is $b_\beta\coloneq j_{0,1}/\sqrt{a_\beta}$, with
$J_\beta'(b_\beta)<0$. While $Z_\beta<0$ and $J_\beta>0$,
\[
\begin{aligned}
\mathcal W_\beta(s)&\coloneq s(Z_\beta'J_\beta-Z_\beta J_\beta'),\\
\mathcal W_\beta'(s)&=-s(c(s)-a_\beta)Z_\beta J_\beta\ge0.
\end{aligned}
\]
As above, $Z_\beta'(s)=O(s)$ and $J_\beta'(s)=O(s)$ imply
$\mathcal W_\beta(s)\to0$ at the origin. If $Z_\beta$ stayed negative on
$[0,b_\beta]$, integration would give $\mathcal W_\beta(b_\beta)\ge0$, whereas
\[
\mathcal W_\beta(b_\beta)
=-b_\beta Z_\beta(b_\beta)J_\beta'(b_\beta)<0.
\]
Hence $Z_\beta$ has a finite first positive zero satisfying
\[
0<s_\beta^0\le\frac{j_{0,1}}{\sqrt{a_\beta}}.
\]
The integral identity gives $Z_\beta'(s_\beta^0)>0$, so it crosses zero.
If the derivative then remained positive, the identity
$(sZ_\beta')'=s(1-e^{Z_\beta})$ would force it to become negative.
The first extremum is a strict positive maximum.

At either extremum $Z_\beta''=1-e^{Z_\beta}\ne0$.
Smooth parameter dependence and the implicit function theorem give
smoothness of the first extremum radius. The expansion
$Z_\beta'(s)=(1-e^\beta)s/2+O(s^3)$ controls the sign near zero,
so the continued extremum remains the first one.
\end{proof}

Set
\begin{equation}\label{eq:shooting-family}
\begin{gathered}
\mu_\beta\coloneq R_\beta^2,\qquad
u_\beta(r)\coloneq R_\beta^2e^{Z_\beta(R_\beta r)},\\
v_\beta(r)\coloneq Z_\beta(R_\beta r)-\av{Z_\beta(R_\beta|\cdot|)}
\qquad\text{for }r\in[0,1].
\end{gathered}
\end{equation}
Integration of \eqref{eq:shooting} up to $R_\beta$ yields
\begin{equation}\label{eq:shooting-mass}
\int_0^{R_\beta}se^{Z_\beta(s)}\dd s=\frac{R_\beta^2}{2},
\qquad M_\beta\coloneq\int_\Omega u_\beta\dd x=\pi R_\beta^2.
\end{equation}
Hence \eqref{eq:shooting-family} gives strictly decreasing equilibria for
$\beta>0$ and strictly increasing equilibria for $\beta<0$.
Conversely, every monotone nonconstant equilibrium arises this way:
rescale \eqref{eq:z} by $s\coloneq\sqrt\mu\,r$ and take
$\beta\coloneq z(0)$. The mean condition and monotonicity give the sign of $\beta$;
the boundary is its first extremum.

\begin{proposition}[Small amplitude and the increasing family]\label{prop:small}
As $\beta\to0$ from either side,
\[
R_\beta\to j_{1,1},\qquad u_\beta\to\Lambda
\quad\text{in }C^m(\overline\Omega)
\]
for every integer $m\ge0$.
For $\beta<0$, $M_\beta>\Lambda\pi$ and $M_\beta\to\infty$ as
$\beta\to-\infty$. Thus an increasing equilibrium exists at every
mass $M>\Lambda\pi$.
\end{proposition}

\begin{proof}
Smooth dependence at $\beta=0$ gives
\[
\frac{Z_\beta(s)}{\beta}\to J_0(s)
\quad\text{in }C^m_{\mathrm{loc}}([0,\infty))
\]
as $\beta\to0$, for every integer $m\ge0$.
The first positive zero of $J_0'=-J_1$ is simple.
Near $s=0$, the expansion
$Z_\beta'(s)/\beta=-(e^\beta-1)s/(2\beta)+O(s^3)$ is uniform
for small nonzero $\beta$. Together with convergence away from zero, this
excludes an earlier extremum. Hence $R_\beta\to j_{1,1}$.
Formula \eqref{eq:shooting-family} gives $u_\beta\to\Lambda$ in
$C^m(\overline\Omega)$ for every integer $m\ge0$.

For $\beta<0$, \cref{cor:radial-shape} gives $M_\beta>\Lambda\pi$.
Before the first zero $s_\beta^0$ of $Z_\beta$,
\[
(sZ_\beta')'=s(1-e^{Z_\beta})\le s,
\qquad Z_\beta(s)\le\beta+\frac{s^2}{4}.
\]
Therefore $s_\beta^0\ge2\sqrt{-\beta}$ and
$R_\beta>s_\beta^0\to\infty$.
Continuity of $M_\beta$ and its limits at $0$ and $-\infty$
give every mass in $(\Lambda\pi,\infty)$ by the intermediate value theorem.
\end{proof}

\subsection{Increasing equilibria and concentration minimality}

Fix a strictly increasing equilibrium, with cumulative profile $F$, and set
\[
Q(\xi)\coloneq M\xi-F(\xi)=-H(\xi),\qquad
U(\xi)\coloneq u_*(\sqrt\xi)\qquad\text{for }\xi\in[0,1].
\]
Then $Q>0$ on $(0,1)$. The stationary identities give
\begin{equation}\label{eq:increasing-stationary-identities}
Q(0)=Q(1)=0,\qquad Q'=M-\pi U,\qquad
4\xi Q''+UQ=0,\qquad \frac{U'}U=\frac{Q}{4\pi\xi}
\quad\text{for }\xi\in(0,1).
\end{equation}
All these profiles are smooth at zero, and $Q(\xi)=O(\xi)$.
The cumulative linearization at fixed mass is
\begin{equation}\label{eq:increasing-cumulative-linearization}
L_Fh\coloneq4\xi h''-\frac Q\pi h'+Uh\quad\text{for }\xi\in(0,1),\qquad h(0)=h(1)=0.
\end{equation}

\begin{lemma}[A positive strict supersolution]
\label{lem:increasing-supersolution}
The function $\zeta\coloneq U^{1/3}Q$ is positive on $(0,1)$, vanishes at
both endpoints, and satisfies
\begin{equation}\label{eq:increasing-supersolution}
L_F\zeta=-\frac{U^{1/3}Q^2}{3\pi\xi}
\left(1+\frac{Q}{6\pi}\right)<0\qquad\text{for }\xi\in(0,1).
\end{equation}
\end{lemma}

\begin{proof}
Since  $U$ is smooth and strictly positive on $[0,1]$, and $Q>0$ on
$(0,1)$ with $Q(0)=Q(1)=0$, the function $\zeta=U^{1/3}Q$ is smooth,
positive in the interior, and zero at both endpoints. It remains to
compute $L_F\zeta$.

All derivatives below are taken with respect to $\xi\in(0,1)$.
For any $p\in\R$, the identity $U'/U=Q/(4\pi\xi)$ gives
\begin{align*}
(U^p)'&=pU^p\frac{U'}U
       =\frac{pQ}{4\pi\xi}U^p,\\
(U^p)''&=\frac{pQ}{4\pi\xi}(U^p)'
       +pU^p\left(\frac{Q}{4\pi\xi}\right)'\\
&=U^p\left[
\frac{pQ'}{4\pi\xi}-\frac{pQ}{4\pi\xi^2}
+\frac{p^2Q^2}{16\pi^2\xi^2}\right].
\end{align*}
The product rule therefore yields
\[
(U^pQ)'=(U^p)'Q+U^pQ',\qquad
(U^pQ)''=(U^p)''Q+2(U^p)'Q'+U^pQ''.
\]
Substituting these expressions into
$L_Fh=4\xi h''-(Q/\pi)h'+Uh$, we obtain
\begin{align*}
L_F(U^pQ)
&=U^p(4\xi Q''+UQ)
  +\left[8\xi(U^p)'-\frac{Q}{\pi}U^p\right]Q'\\
&\quad+\left[4\xi(U^p)''-\frac{Q}{\pi}(U^p)'\right]Q\\
&=U^p\Biggl\{4\xi Q''+UQ
  +\frac{(2p-1)QQ'}{\pi}\\
&\hspace{3.5em}
  +\left[\frac{pQ'}{\pi}-\frac{pQ}{\pi\xi}
  +\frac{p(p-1)Q^2}{4\pi^2\xi}\right]Q\Biggr\}.
\end{align*}
Using $4\xi Q''+UQ=0$ and combining the two $QQ'$ terms gives
\[
L_F(U^pQ)=U^pQ\left[
\frac{(3p-1)Q'}{\pi}-\frac{pQ}{\pi\xi}
+\frac{p(p-1)Q^2}{4\pi^2\xi}\right].
\]
We now choose $p=1/3$, so that $3p-1=0$ and $p(p-1)=-2/9$.
Consequently,
\begin{align*}
L_F\zeta
&=U^{1/3}Q\left[-\frac{Q}{3\pi\xi}
                   -\frac{Q^2}{18\pi^2\xi}\right]\\
&=-\frac{U^{1/3}Q^2}{3\pi\xi}
  \left(1+\frac{Q}{6\pi}\right)<0
  \qquad\text{for }\xi\in(0,1),
\end{align*}
since $U>0$ and $Q>0$ there. This proves
\eqref{eq:increasing-supersolution}.
\end{proof}

\begin{proposition}[Strict spectral bound and radial nondegeneracy]
\label{prop:increasing-nondegeneracy}
Define
\[
m(\xi)\coloneq\frac1{4\xi U(\xi)}\quad\text{for }\xi\in(0,1],\qquad
\mathcal H\coloneq L^2((0,1),m(\xi)\dd\xi),\qquad V\coloneq H^1_0((0,1)).
\]
The Dirichlet realization of $L_F$ associated with the quadratic form on $V$ is
self-adjoint in $\mathcal H$, has compact resolvent, and satisfies
\begin{equation}\label{eq:increasing-negative-spectrum}
\spec(L_F)\subset(-\infty,-\gamma]
\end{equation}
for some $\gamma>0$. In particular, $L_Fh=0$ with $h(0)=h(1)=0$
has only the zero classical solution.
\end{proposition}

\begin{proof}
The last identity in \eqref{eq:increasing-stationary-identities} gives
\begin{equation}\label{eq:increasing-form-operator}
L_Fh=\frac1m\left(\frac{h'}U\right)'+Uh.
\end{equation}
Its negative is represented by the symmetric form
\begin{equation}\label{eq:increasing-form}
\mathfrak a(h,h)\coloneq\int_0^1\frac{|h'|^2}{U}\dd\xi
-\frac14\int_0^1\frac{h^2}{\xi}\dd\xi,\qquad \text{for }h\in V.
\end{equation}
Since $U$ is bounded above and below by positive constants, the
estimates \eqref{eq:weighted-endpoint-control} and the usual compact
embedding away from zero imply that
$V$ embeds continuously and compactly into $\mathcal H$.
Adding a sufficiently large multiple of $\norm{h}_{\mathcal H}^2$
makes the form norm equivalent to the $H^1_0$ norm. Thus the form is
closed and defines a self-adjoint realization with compact resolvent.

For the function $\zeta$ in \cref{lem:increasing-supersolution}, put
\[
c(\xi)\coloneq-\frac{L_F\zeta}{\zeta}
=\frac{Q}{3\pi\xi}\left(1+\frac{Q}{6\pi}\right)\qquad\text{for }\xi\in(0,1).
\]
It extends boundedly to $[0,1]$ and is positive on $(0,1)$.
For $h\in C_c^\infty((0,1))$, set $v\coloneq h/\zeta$. Since $\zeta>0$
in $(0,1)$, we have $v\in C_c^\infty((0,1))$, and $h'=\zeta'v+\zeta v'$
gives
\[
\frac{|h'|^2}{U}
=\frac{\zeta^2}{U}|v'|^2+\frac{(\zeta')^2}{U}v^2
+\frac{\zeta\zeta'}{U}(v^2)'.
\]
The boundary term vanishes because $v$ is compactly supported in $(0,1)$.
Integration by parts in the last term therefore yields
\begin{align*}
\int_0^1\frac{|h'|^2}{U}\dd\xi
&=\int_0^1\frac{\zeta^2}{U}|v'|^2\dd\xi
 +\int_0^1\left[\frac{(\zeta')^2}{U}
 -\left(\frac{\zeta\zeta'}{U}\right)'\right]v^2\dd\xi\\
&=\int_0^1\frac{\zeta^2}{U}|v'|^2\dd\xi
 -\int_0^1\zeta\left(\frac{\zeta'}{U}\right)'v^2\dd\xi.
\end{align*}
On the other hand, \eqref{eq:increasing-form-operator},
$mU=1/(4\xi)$, and $L_F\zeta=-c\zeta$ imply
\[
-\frac1\zeta\left(\frac{\zeta'}{U}\right)'-\frac1{4\xi}
=-\frac m\zeta L_F\zeta=cm.
\]
Substituting $v=h/\zeta$ into the preceding integral identity and
subtracting the potential term in \eqref{eq:increasing-form}, we obtain
\begin{equation*}
\mathfrak a(h,h)=\int_0^1\frac{\zeta^2}{U}
\left|\left(\frac h\zeta\right)'\right|^2\dd\xi
+\int_0^1cmh^2\dd\xi.
\end{equation*}
By density and \eqref{eq:weighted-endpoint-control},
\begin{equation}\label{eq:increasing-positive-form}
\mathfrak a(h,h)\ge\int_0^1cmh^2\dd\xi>0
\qquad\text{for }h\in V\setminus\{0\}.
\end{equation}
This strict positivity gives a spectral gap: a minimizing sequence for
$\mathfrak a(h,h)$ under $\norm{h}_{\mathcal H}=1$ is bounded in $V$, since
\[
\int_0^1\frac{|h'|^2}{U}\dd\xi
=\mathfrak a(h,h)+\int_0^1Umh^2\dd\xi
\le\mathfrak a(h,h)+\max_{[0,1]}U.
\]
Compactness in $\mathcal H$ and weak lower semicontinuity give a
minimizer, whose value $\gamma$ is positive by
\eqref{eq:increasing-positive-form}. The variational characterization
proves \eqref{eq:increasing-negative-spectrum}.
Finally, a classical solution of $L_Fh=0$ belongs to $V$, and multiplying
by $mh$ gives $\mathfrak a(h,h)=0$; hence $h=0$.
\end{proof}

\begin{proposition}[Strict monotonicity of the increasing shooting mass]
\label{prop:increasing-mass}
For the first shooting family \eqref{eq:shooting-family},
\begin{equation*}
\frac{\dd M_\beta}{\dd\beta}<0\qquad\text{for }\beta\in(-\infty,0).
\end{equation*}
Thus $\beta\mapsto M_\beta$ is a smooth bijection from $(-\infty,0)$
onto $(\Lambda\pi,\infty)$, with smooth inverse. There is exactly one
increasing radial equilibrium $(u_M^{\mathrm i},v_M^{\mathrm i})$ at each $M>\Lambda\pi$, and these
equilibria depend smoothly on $M$.
The two parametrizations satisfy
$(u_\beta,v_\beta)=(u_{M_\beta}^{\mathrm i},v_{M_\beta}^{\mathrm i})$
for $\beta<0$.
\end{proposition}

\begin{proof}
Smoothness follows from \cref{lem:first-extremum} and
\eqref{eq:shooting-family}. Suppose that
$\dot M_{\beta_0}\coloneq \left.\dd M_\beta/\dd\beta\right|_{\beta_0}=0$
at some $\beta_0<0$. Let
\[
F_\beta(\xi)\coloneq2\pi\int_0^{\sqrt\xi}ru_\beta(r)\dd r,
\qquad G(\xi)\coloneq\left.\partial_\beta F_\beta(\xi)\right|_{\beta_0},
\qquad\text{for }\xi\in[0,1].
\]
Primes below denote derivatives with respect to $\xi$. For each $\beta$,
\eqref{eq:stationary} reads
\[
4\xi F_\beta''+\frac{F_\beta-M_\beta\xi}{\pi}F_\beta'=0,
\qquad F_\beta(0)=0,\quad F_\beta(1)=M_\beta.
\]
The profiles depend smoothly on $\beta$, so parameter differentiation
commutes with differentiation in $\xi$. Keeping $\xi$ fixed and applying
the product rule at $\beta=\beta_0$, we obtain
\begin{align*}
0
&=4\xi G''
 +\frac{G-\dot M_{\beta_0}\xi}{\pi}F_{\beta_0}'
 +\frac{F_{\beta_0}-M_{\beta_0}\xi}{\pi}G'\\
&=4\xi G''+\frac{F_{\beta_0}-M_{\beta_0}\xi}{\pi}G'
 +\frac{F_{\beta_0}'}{\pi}G
 -\frac{\dot M_{\beta_0}}{\pi}\,\xi F_{\beta_0}'.
\end{align*}
Since $F_{\beta_0}'/\pi=u_{\beta_0}(\sqrt\xi)$, the first three terms
in the last line are precisely $L_{F_{\beta_0}}G$ by
\eqref{eq:increasing-cumulative-linearization}.
Moreover, differentiating the endpoint identities at the fixed endpoints
$\xi=0$ and $\xi=1$ gives
\[
G(0)=\left.\partial_\beta F_\beta(0)\right|_{\beta_0}=0,
\qquad
G(1)=\left.\partial_\beta F_\beta(1)\right|_{\beta_0}
=\dot M_{\beta_0}.
\]
Thus, using the assumption $\dot M_{\beta_0}=0$, we conclude that
\begin{equation*}
L_{F_{\beta_0}}G=\frac{\dot M_{\beta_0}}{\pi}\,\xi F_{\beta_0}'=0,
\qquad G(0)=0,\quad G(1)=\dot M_{\beta_0}=0.
\end{equation*}
\cref{prop:increasing-nondegeneracy} implies $G=0$.
On the other hand, $F_\beta'(0)=\pi u_\beta(0)=M_\beta e^\beta$, so
\[
G'(0)=e^{\beta_0}(\dot M_{\beta_0}+M_{\beta_0})
=e^{\beta_0}M_{\beta_0}>0,
\]
a contradiction. Hence $\dd M_\beta/\dd\beta$ never vanishes on
$(-\infty,0)$. Its sign is constant, and the endpoint limits in
\cref{prop:small} force that sign to be negative and give
the stated range. The inverse function theorem gives the smooth inverse.
Every increasing equilibrium belongs to this first family, and
\eqref{eq:shooting-family} uniquely determines both the density and the
normalized potential. This proves uniqueness and smooth dependence on mass.
\end{proof}

\begin{proof}[Proof of the strict minimality assertion in \cref{thm:least-concentrated}]
Write $I\coloneq F_{u_M^{\mathrm i}}$, $S\coloneq F_u$, and $D\coloneq S-I$. In this proof all coefficients
refer to $u_M^{\mathrm i}$: set $Q\coloneq M\xi-I$, $U(\xi)\coloneq u_M^{\mathrm i}(\sqrt\xi)$, and
$m\coloneq(4\xi U)^{-1}$, and use $L_I$ and $\mathfrak a$ from
\eqref{eq:increasing-cumulative-linearization} and
\eqref{eq:increasing-form}. Subtracting the stationary equations gives
\begin{equation*}
L_ID+\frac1\pi DD'=0,\qquad D(0)=D(1)=0.
\end{equation*}
Let $g\coloneq D_-=\max\{-D,0\}\in H^1_0((0,1))$. The negative-part chain rule gives
\[
g'D'=-|g'|^2,\qquad gD=-g^2,\qquad gDD'=g^2g'
\quad\text{almost everywhere}.
\]
Multiplication by $mg$, followed by integration by parts using
\eqref{eq:increasing-form-operator}, yields
\begin{equation}
  \label{eq:increasing-order-negative-part}
\mathfrak a(g,g)=\frac1{3\pi}\int_0^1m'g^3\dd\xi.
\end{equation}
Indeed, $\int mgL_ID=\mathfrak a(g,g)$ and
$\int mgDD'=\int mg^2g'=-(1/3)\int m'g^3$.
These calculations can first be made on $[\delta,1]$. At zero,
$D,g=O(\xi)$ and $D'$ is bounded, while $m=O(\xi^{-1})$ and
$m'=O(\xi^{-2})$. Thus the endpoint terms $U^{-1}D'g=O(\xi)$ and
$mg^3=O(\xi^2)$ vanish, and $m'g^3=O(\xi)$ is integrable.
The chain rule handles all interior zeros of $D$.

By \eqref{eq:increasing-stationary-identities},
\begin{equation*}
m'=-\frac m\xi\left(1+\frac{Q}{4\pi}\right)<0
\qquad\text{for }\xi\in(0,1).
\end{equation*}
If $g\not\equiv0$, the left-hand side of
\eqref{eq:increasing-order-negative-part} is strictly positive by
\eqref{eq:increasing-positive-form}, whereas its right-hand side is
strictly negative. Therefore $g=0$ and $D\ge0$.
To prove strictness, rewrite the difference equation as
\[
4\xi D''+\frac{S-M\xi}{\pi}D'+UD=0.
\]
If $D$ vanished at an interior point, its derivative would also vanish
there. Uniqueness for this linear ODE on $(0,1)$ would imply $D=0$,
and hence $u=u_M^{\mathrm i}$. Thus $u\not\equiv u_M^{\mathrm i}$ gives $D>0$ throughout $(0,1)$.
This proves strict minimality and identifies $(u_M^{\mathrm i},v_M^{\mathrm i})$
as the unique least element in the concentration order.
\end{proof}

\begin{remark}[A geometric consequence]
Let $M>\Lambda\pi$, and let $(u,v)$ be any radial equilibrium of mass $M$
distinct from $(u_M^{\mathrm i},v_M^{\mathrm i})$.
By \eqref{eq:stationary-radial-derivative} and \cref{thm:least-concentrated},
\[
\frac{\dd}{\dd r}\log\frac{u(r)}{u_M^{\mathrm i}(r)}
=-\frac{F_u(r^2)-F_{u_M^{\mathrm i}}(r^2)}{2\pi r}<0\qquad\text{for }r\in(0,1).
\]
Thus the density ratio is strictly decreasing. Equality of the masses
implies a unique crossing $r_*\in(0,1)$, with $u(r_*)=u_M^{\mathrm i}(r_*)$ and
\[
u(r)>u_M^{\mathrm i}(r)\quad\text{for }r\in[0,r_*),\qquad
u(r)<u_M^{\mathrm i}(r)\quad\text{for }r\in(r_*,1].
\]
This holds even when $u$ has several radial extrema.
\end{remark}

\begin{proposition}[Equilibria less concentrated than the constant]
\label{prop:finite}
At each mass $M>\Lambda\pi$, the radial equilibria with cumulative profile
$F\le M\xi$ are exactly the constant equilibrium and $(u_M^{\mathrm i},v_M^{\mathrm i})$.
\end{proposition}

\begin{proof}
If $H=F-M\xi\le0$ is nonzero, \cref{lem:stationary-sign} shows
that the density is strictly increasing. \cref{prop:increasing-mass}
identifies it with $u_M^{\mathrm i}$. Conversely, both stated equilibria have
$F\le M\xi$.
\end{proof}

\subsection{The decreasing family: concentration, uniqueness, and nondegeneracy}

\begin{theorem}[Decreasing equilibria at prescribed mass]
\label{thm:decreasing-uniqueness}
For each $8\pi<M<\Lambda\pi$, there is exactly one strictly decreasing
radial equilibrium $(u_M^{\mathrm d},v_M^{\mathrm d})$ of
\eqref{eq:system}, with $\int_\Omega v_M^{\mathrm d}\dd x=0$.
There is no such equilibrium at any other positive mass.
These equilibria form a smooth branch parametrized by $M$.
Every equilibrium on this branch is nondegenerate under radial
fixed-mass perturbations: its cumulative linearization has trivial
kernel. This linearization also has a positive eigenvalue.
\end{theorem}

The mass monotonicity proved below does not assume nondegeneracy.
The latter will follow from the shooting variation.

\subsubsection{The singular endpoint of the decreasing family}

\begin{proposition}[The concentrated endpoint]\label{prop:bubble}
For $\beta>0$,
\[
8\pi<M_\beta<\Lambda\pi,\qquad
R_\beta\to\sqrt8,\quad M_\beta\to8\pi
\quad\text{as }\beta\to\infty.
\]
Moreover, $u_\beta\dd x\rightharpoonup8\pi\delta_0$.
\end{proposition}

\begin{proof}
For $\beta>0$, the profile $u_\beta$ is a strictly decreasing,
nonconstant radial equilibrium. Hence \cref{cor:radial-shape} gives
$M_\beta<\Lambda\pi$, while low-mass stationary uniqueness
\cite[Theorem~2.2]{Horstmann2011} excludes $M_\beta\le8\pi$.
Since $M_\beta=\pi R_\beta^2$ by \eqref{eq:shooting-mass}, we obtain
\begin{equation}\label{eq:R-bounds}
\sqrt8<R_\beta<j_{1,1}.
\end{equation}
We next determine how the mass is distributed as $\beta\to\infty$.
The value $Z_\beta(0)=\beta$ suggests the core scale $e^{-\beta/2}$.
Accordingly, define
\[
Y_\beta(y)\coloneq Z_\beta(e^{-\beta/2}y)-\beta\qquad\text{for }y\in[0,\infty).
\]
Writing $s\coloneq e^{-\beta/2}y$, the chain rule gives
\[
Y_\beta'(y)=e^{-\beta/2}Z_\beta'(s),\qquad
Y_\beta''(y)=e^{-\beta}Z_\beta''(s),\qquad
e^{Z_\beta(s)}=e^\beta e^{Y_\beta(y)}.
\]
Multiplying \eqref{eq:shooting} by $e^{-\beta}$ therefore yields
\[
Y_\beta''+\frac1yY_\beta'+e^{Y_\beta}-e^{-\beta}=0\quad\text{for }y\in(0,\infty),\qquad
Y_\beta(0)=Y_\beta'(0)=0.
\]
To handle the apparent singularity at $y=0$, integrating
$(yY_\beta')'=y(e^{-\beta}-e^{Y_\beta})$ twice yields
\begin{align*}
Y_\beta'(y)
&=\frac{e^{-\beta}y}{2}
 -\frac1y\int_0^y te^{Y_\beta(t)}\dd t,\\
Y_\beta(y)
&=\frac{e^{-\beta}y^2}{4}
 -\int_0^y t\log\frac yt\,e^{Y_\beta(t)}\dd t.
\end{align*}
For each fixed $A>0$, the second identity and
$\int_0^y t\log(y/t)\dd t=y^2/4$ give the uniform bounds
\[
-\frac{e^{A^2/4}y^2}{4}\le Y_\beta(y)
\le\frac{e^{-\beta}y^2}{4}\le\frac{A^2}{4}
\qquad\text{for }y\in[0,A].
\]
The exponential is uniformly Lipschitz on the resulting bounded range,
with Lipschitz constant $L_A\coloneq e^{A^2/4}$.
The limiting initial-value problem has the explicit global solution
\[
Y_*(y)\coloneq-2\log\left(1+\frac{y^2}{8}\right),\qquad
Y_*'(y)=-\frac{4y}{8+y^2},\qquad
e^{Y_*(y)}=\frac{64}{(8+y^2)^2}\qquad\text{for }y\in[0,\infty).
\]
Direct substitution verifies
$Y_*''+y^{-1}Y_*'+e^{Y_*}=0$ and $Y_*(0)=Y_*'(0)=0$.
Set $D_\beta\coloneq Y_\beta-Y_*$ on $[0,A]$. Subtracting the
integral identities for $Y_\beta$ and $Y_*$, and using
$t\log(y/t)\le y/e\le A/e$ for $0<t\le y\le A$, gives
\begin{align*}
|D_\beta(y)|
&\le\frac{e^{-\beta}A^2}{4}
 +L_A\int_0^y t\log\frac yt\,|D_\beta(t)|\dd t\\
&\le\frac{e^{-\beta}A^2}{4}
 +\frac{AL_A}{e}\int_0^y|D_\beta(t)|\dd t.
\end{align*}
Gronwall's inequality \cite[Lemma~2.7]{Teschl2012} therefore yields
\[
\|D_\beta\|_{C([0,A])}
\le\frac{A^2}{4}\exp\left(\frac{L_AA^2}{e}\right)e^{-\beta}\to0.
\]
This proves uniform convergence on $[0,A]$, including at the origin.
Together with the derivative identities below, it gives
\begin{equation}\label{eq:Liouville}
Y_\beta(y)\to
Y_*(y)=-2\log\left(1+\frac{y^2}{8}\right)
\quad\text{in }C^2_{\mathrm{loc}}([0,\infty)).
\end{equation}
In particular, convergence of the second derivatives also holds at zero:
the first integral identity gives
\[
\frac{Y_\beta'(y)}y
=\frac{e^{-\beta}}2-\int_0^1te^{Y_\beta(yt)}\dd t,
\]
whose right-hand side extends continuously to $y=0$; combining this
with the differential equation gives convergence of $Y_\beta''$ there.
The mass of the limiting core is
\begin{equation*}
\begin{aligned}
\int_0^S ye^{Y_*(y)}\dd y
&=\int_0^S\frac{64y}{(8+y^2)^2}\dd y\\
&=4-\frac{32}{8+S^2}
=\frac{4S^2}{8+S^2}\eqcolon q_S,
\qquad q_S\to4.
\end{aligned}
\end{equation*}

Local convergence on the core scale alone does not determine the total
mass, so we also estimate the mass outside the core. Define
\[
Q_\beta(s)\coloneq\int_0^s te^{Z_\beta(t)}\dd t,
\qquad\text{for }s\in[0,j_{1,1}].
\]
This is well defined on the whole interval because $Z_\beta$ is global
by \cref{lem:first-extremum}. Integrating
$(sZ_\beta')'=s(1-e^{Z_\beta})$ from zero to $s$ gives
\begin{equation}\label{eq:Q}
sZ_\beta'(s)=\frac{s^2}{2}-Q_\beta(s).
\end{equation}
Fix $S>\sqrt8$, put $\rho_\beta\coloneq e^{-\beta/2}S$, and write
$a_\beta\coloneq Q_\beta(\rho_\beta)$. The change of variables
$t=e^{-\beta/2}y$ and \eqref{eq:Liouville} imply
\begin{align*}
a_\beta
&=\int_0^S ye^{Y_\beta(y)}\dd y\longrightarrow q_S>2,\\
\rho_\beta^2e^{Z_\beta(\rho_\beta)}
&=S^2e^{Y_\beta(S)}\longrightarrow S^2e^{Y_*(S)}.
\end{align*}
Thus $0<\rho_\beta<j_{1,1}$ and $a_\beta>2$ for all sufficiently
large $\beta$, with $S$ fixed.
Since $Q_\beta'=se^{Z_\beta}>0$ for $s>0$, \eqref{eq:Q} yields
\[
Z_\beta'(t)=\frac t2-\frac{Q_\beta(t)}t
\le\frac t2-\frac{a_\beta}t
\qquad\text{for }t\in[\rho_\beta,j_{1,1}].
\]
Integrating from $\rho_\beta$ to $s$ gives
\[
Z_\beta(s)\le Z_\beta(\rho_\beta)+\frac{s^2-\rho_\beta^2}{4}
-a_\beta\log\frac{s}{\rho_\beta}.
\]
Exponentiation, followed by integration, therefore gives
\begin{align*}
Q_\beta(s)-a_\beta
&=\int_{\rho_\beta}^s te^{Z_\beta(t)}\dd t\\
&\le e^{j_{1,1}^2/4}e^{Z_\beta(\rho_\beta)}
 \int_{\rho_\beta}^s t\left(\frac{\rho_\beta}{t}\right)^{a_\beta}\dd t\\
&=e^{j_{1,1}^2/4}\rho_\beta^2e^{Z_\beta(\rho_\beta)}
 \int_1^{s/\rho_\beta}\tau^{1-a_\beta}\dd\tau\\
&=\frac{e^{j_{1,1}^2/4}\rho_\beta^2e^{Z_\beta(\rho_\beta)}}{a_\beta-2}
 \left[1-\left(\frac{\rho_\beta}{s}\right)^{a_\beta-2}\right].
\end{align*}
In particular, uniformly for $\rho_\beta\le s\le j_{1,1}$,
\begin{equation}\label{eq:tail-bound}
0\le Q_\beta(s)-Q_\beta(\rho_\beta)
\le
\frac{e^{j_{1,1}^2/4}\,\rho_\beta^2e^{Z_\beta(\rho_\beta)}}
{Q_\beta(\rho_\beta)-2}.
\end{equation}

We now make the order of the two limits explicit. Fix
$0<\delta<j_{1,1}$. For each fixed $S>\sqrt8$, we have
$\rho_\beta<\delta$ for all sufficiently large $\beta$, so
\eqref{eq:tail-bound} implies
\[
\sup_{s\in[\delta,j_{1,1}]}|Q_\beta(s)-4|
\le |a_\beta-4|
+\frac{e^{j_{1,1}^2/4}\rho_\beta^2e^{Z_\beta(\rho_\beta)}}{a_\beta-2}.
\]
First letting $\beta\to\infty$ with $S$ fixed gives
\[
\limsup_{\beta\to\infty}
\sup_{s\in[\delta,j_{1,1}]}|Q_\beta(s)-4|
\le 4-q_S+
\frac{e^{j_{1,1}^2/4}S^2e^{Y_*(S)}}{q_S-2}.
\]
Next let $S\to\infty$. The right-hand side tends to zero because
$q_S\to4$ and
\[
S^2e^{Y_*(S)}=\frac{64S^2}{(8+S^2)^2}=O(S^{-2}).
\]
We have proved $Q_\beta\to4$ uniformly on every fixed
$[\delta,j_{1,1}]$ with $\delta>0$.

By \eqref{eq:R-bounds}, the moving endpoint $R_\beta$ stays in the
fixed interval $[\sqrt8,j_{1,1}]$, where this convergence is uniform.
Moreover, $Z_\beta'(R_\beta)=0$, so \eqref{eq:Q} gives
\[
\frac{R_\beta^2}{2}=Q_\beta(R_\beta)\to4,
\qquad R_\beta\to\sqrt8,\qquad
M_\beta=\pi R_\beta^2\to8\pi.
\]

Finally, for each fixed $0<r\le1$, \eqref{eq:shooting-family} and the
change of variables $s=R_\beta t$ give
\begin{align*}
\int_{B_r}u_\beta\dd x
&=2\pi\int_0^r tR_\beta^2e^{Z_\beta(R_\beta t)}\dd t\\
&=2\pi Q_\beta(R_\beta r)\to8\pi.
\end{align*}
Here uniform convergence applies because
$\sqrt8\,r<R_\beta r<j_{1,1}$. Thus, for every fixed $0<r<1$,
\[
\int_{\Omega\setminus B_r}u_\beta\dd x
=M_\beta-\int_{B_r}u_\beta\dd x\longrightarrow0.
\]
To verify weak convergence, take $\varphi\in C(\overline\Omega)$.
Splitting the integral over $B_r$ and its complement, we obtain
\begin{align*}
\left|\int_\Omega\varphi u_\beta\dd x-8\pi\varphi(0)\right|
&\le M_\beta\sup_{x\in\overline B_r}|\varphi(x)-\varphi(0)|\\
&\quad+2\norm{\varphi}_\infty
 \int_{\Omega\setminus B_r}u_\beta\dd x
 +|\varphi(0)|\,|M_\beta-8\pi|.
\end{align*}
Letting first $\beta\to\infty$ and then $r\downarrow0$, continuity
of $\varphi$ at zero makes the right-hand side tend to zero.
This proves $u_\beta\dd x\rightharpoonup8\pi\delta_0$ on
$\overline\Omega$.
\end{proof}

\subsubsection{Mass monotonicity in the original shooting variables}

We use the original shooting radius $s$ in \eqref{eq:shooting} and the
cumulative integral $Q_\beta$ already defined in the proof of
\cref{prop:bubble}. Fix $\beta>0$ and write
$Z\coloneq Z_\beta$, $Q\coloneq Q_\beta$, and $R\coloneq R_\beta$.
By \cref{lem:first-extremum}, $Z'<0$ on $(0,R)$ and $Z(R)<0$.
In particular, $Z$ has exactly one zero $s_0\in(0,R)$.
In the following shooting-variable argument, primes denote differentiation
with respect to $s$. The functions below are
defined for $s\in[0,R]$, with their differential equations understood
on $(0,R)$.

Differentiate $Q_\beta$ at fixed $s$ and set
\[
V(s)\coloneq\partial_\beta Q_\beta(s)=-s\partial_\beta Z_\beta'(s),
\]
where the second equality follows from \eqref{eq:Q}.
Since $Q_\beta'=se^{Z_\beta}$, we have
\[
V'=se^Z\partial_\beta Z_\beta,\qquad
V''=\left(\frac1s+Z'\right)V'-e^ZV.
\]
Thus, for the operator
\begin{equation*}
\mathcal L_{\mathrm d}h
\coloneq h''-\left(\frac1s+Z'\right)h'+e^Zh,
\end{equation*}
the shooting variation satisfies
\begin{equation}\label{eq:dec-variation}
\mathcal L_{\mathrm d}V=0,\qquad
V(s)=\frac{e^\beta}{2}s^2+O(s^4)\quad\text{as }s\downarrow0.
\end{equation}
The self-adjoint form
\begin{equation*}
\mathcal L_{\mathrm d}h
=se^Z\left(\frac{h'}{se^Z}\right)'+e^Zh
\end{equation*}
gives the Green identity
\begin{equation}\label{eq:dec-green}
\int_a^b\frac{f\mathcal L_{\mathrm d}g-g\mathcal L_{\mathrm d}f}{se^Z}\dd s
=\left[\frac{fg'-gf'}{se^Z}\right]_a^b,\qquad\text{for }0<a<b\le R.
\end{equation}
Below, every function used at zero is $O(s^2)$ with derivative $O(s)$.
Hence the boundary term at zero vanishes, and the integrals are
convergent; applications with $a=0$ follow by taking $a\downarrow0$.

\begin{lemma}[Three comparison functions]\label{lem:dec-comparison}
Define
\[
D\coloneq-sZ',\qquad T\coloneq s^2(e^Z-1),\qquad \Psi\coloneq e^{Z/2}D.
\]
Then $D,\Psi>0$ on $(0,R)$, while $T>0$ on $(0,s_0)$ and
$T(s_0)=0$. Moreover,
\begin{equation}\label{eq:dec-comparisons}
\begin{aligned}
\mathcal L_{\mathrm d}T&=-2D<0,\\
\mathcal L_{\mathrm d}\Psi
&=\frac{\Psi}{2s^2}\int_0^s tD(t)\dd t>0,\\
\mathcal L_{\mathrm d}D&=(e^Z-1)D,
\end{aligned}
\end{equation}
and
\begin{equation*}
\Psi(0)=\Psi(R)=0,\qquad
\Psi'(R)=Re^{Z(R)/2}(e^{Z(R)}-1)<0.
\end{equation*}
\end{lemma}

\begin{proof}
Equation~\eqref{eq:shooting} gives
\[
D'=s(e^Z-1),\qquad T=sD'.
\]
These identities and direct differentiation give the first and third
identities in \eqref{eq:dec-comparisons}. For the middle one, the
product rule yields
\[
\frac{\mathcal L_{\mathrm d}(e^{Z/2}D)}{e^{Z/2}D}
=-\frac{e^Z-1}{2}-\frac{Z'}s-\frac{(Z')^2}{4}.
\]
On the other hand, \eqref{eq:shooting} implies
\[
\left[-s^2(e^Z-1)-2sZ'-\frac{s^2}{2}(Z')^2\right]'
=-s^2Z'=sD.
\]
The expression in brackets vanishes at zero. Integrating and dividing
by $2s^2$ proves the middle identity and its strict positivity.
The sign and endpoint assertions follow from $Z'<0$ on $(0,R)$,
$Z'(R)=0$, and $Z(R)<0$.
\end{proof}

\begin{proposition}[The shooting variation and mass monotonicity]
\label{prop:dec-monotonicity}
The variation $V=\partial_\beta Q_\beta$ has exactly one zero $b$ in
$(0,R)$, and
\begin{equation*}
s_0<b<R,\qquad V(R)<0.
\end{equation*}
Consequently,
\begin{equation*}
\frac{\dd R_\beta}{\dd\beta}<0,\qquad
\frac{\dd M_\beta}{\dd\beta}<0\qquad\text{for }\beta\in(0,\infty).
\end{equation*}
The map $\beta\mapsto M_\beta$ is a smooth, strictly decreasing
bijection from $(0,\infty)$ onto $(8\pi,\Lambda\pi)$, with smooth inverse.
\end{proposition}

\begin{proof}
By \eqref{eq:dec-variation}, $V>0$ near zero, and every positive zero
is simple by ODE uniqueness.

If the first positive zero $b$ satisfied $b\le s_0$, then
$V>0$ on $(0,b)$, $V'(b)<0$, and $T(b)\ge0$.
Using $f=V$ and $g=T$ in \eqref{eq:dec-green} would give
\[
0>-2\int_0^b\frac{VD}{se^Z}\dd s
=-\frac{T(b)V'(b)}{be^{Z(b)}}\ge0.
\]
Thus $V$ has no zero in $(0,s_0]$.

If $V$ had no zero in $(0,R)$, then $V>0$ there and $V(R)\ge0$.
Using $f=V$ and $g=\Psi$ instead would give
\[
0<\int_0^R\frac{V\mathcal L_{\mathrm d}\Psi}{se^Z}\dd s
=\frac{V(R)\Psi'(R)}{Re^{Z(R)}}\le0.
\]
Hence its first zero $b$ exists and satisfies $s_0<b<R$.

To exclude another zero, let $d\in(b,R]$ be the first one after $b$
and put $B\coloneq-V$. Then $B>0$ on $(b,d)$, $B'(b)>0$, and $B'(d)<0$.
Since $Z<0$ on $(b,R)$, we have
$\mathcal L_{\mathrm d}D=(e^Z-1)D<0$ there. Thus
\eqref{eq:dec-green} with $f=B$ and $g=D$ gives
\[
0>\int_b^d\frac{B\mathcal L_{\mathrm d}D}{se^Z}\dd s
=-\frac{D(d)B'(d)}{de^{Z(d)}}
+\frac{D(b)B'(b)}{be^{Z(b)}}>0.
\]
The contradiction also excludes $d=R$, so $V(R)<0$.

Finally, differentiate $Q_\beta(R_\beta)=R_\beta^2/2$.
Since $Q_\beta'(R_\beta)=R_\beta e^{Z_\beta(R_\beta)}$, this yields
\[
V(R_\beta)
+R_\beta e^{Z_\beta(R_\beta)}\frac{\dd R_\beta}{\dd\beta}
=R_\beta\frac{\dd R_\beta}{\dd\beta}.
\]
Therefore
\[
\frac{\dd R_\beta}{\dd\beta}
=\frac{V(R_\beta)}
{R_\beta(1-e^{Z_\beta(R_\beta)})}<0,\qquad
\frac{\dd M_\beta}{\dd\beta}
=\frac{2\pi V(R_\beta)}{1-e^{Z_\beta(R_\beta)}}<0.
\]
Smoothness follows from \cref{lem:first-extremum}. The endpoint limits
in \cref{prop:small,prop:bubble} give the stated range and bijectivity,
and the inverse function theorem gives the smooth inverse.
\end{proof}

\subsubsection{Nondegeneracy and the sign of the principal eigenvalue}

\begin{proposition}[Radial nondegeneracy and a positive eigenvalue]
\label{prop:dec-spectral}
At a decreasing equilibrium with parameter $\beta>0$, write
$M\coloneq M_\beta$, $F\coloneq F_{u_\beta}$, and
$U(\xi)\coloneq u_\beta(\sqrt\xi)$ for $\xi\in[0,1]$, and put
$m(\xi)\coloneq(4\xi U(\xi))^{-1}$ for $\xi\in(0,1]$.
The radial fixed-mass cumulative linearization is
\[
\mathcal Ah
\coloneq4\xi h''+\frac{F-M\xi}{\pi}h'+Uh,\qquad\text{for }\xi\in(0,1).
\]
Its regular Dirichlet problem
\begin{equation}\label{eq:dec-kernel}
\begin{gathered}
\mathcal Ah=0\quad\text{for }\xi\in(0,1),\qquad h(0)=h(1)=0,\\
\text{with }h\in C^1([0,1])\cap C^2((0,1])
\end{gathered}
\end{equation}
has only the zero solution. Its self-adjoint Dirichlet realization in
$\mathcal H\coloneq L^2((0,1),m\dd\xi)$ has compact resolvent, does not have zero
in its spectrum, and has a positive eigenvalue.
\end{proposition}

\begin{proof}
The stationary identity $U'/U=-(F-M\xi)/(4\pi\xi)$ gives
\[
\mathcal Ah=\frac1m\left(\frac{h'}U\right)'+Uh.
\]
Consequently, every regular solution of $\mathcal Ah=0$ with $h(0)=0$
satisfies
\[
\frac{h'(\xi)}{U(\xi)}
=\frac{h'(0)}{U(0)}-\frac14\int_0^\xi\frac{h(\eta)}{\eta}\dd\eta.
\]
This Volterra identity determines $h$ uniquely from $h'(0)$.

To identify a nonzero regular solution, use the original shooting
radius $s\coloneq R_\beta\sqrt\xi$. Equations~\eqref{eq:shooting-family}
and \eqref{eq:Q} give
\[
F(\xi)-M\xi=-2\pi sZ_\beta'(s),\qquad
U(\xi)=R_\beta^2e^{Z_\beta(s)}.
\]
The chain rule therefore yields
\begin{equation}\label{eq:dec-operator-correspondence}
\mathcal A\bigl[\widetilde h(R_\beta\sqrt\xi)\bigr]
=R_\beta^2(\mathcal L_{\mathrm d}\widetilde h)(R_\beta\sqrt\xi).
\end{equation}
In particular, $h_0(\xi)\coloneq V(R_\beta\sqrt\xi)$ is a regular solution of
$\mathcal Ah_0=0$, and
\[
h_0'(0)=\frac{R_\beta^2e^\beta}{2}>0,\qquad
h_0(1)=V(R_\beta)<0
\]
by \cref{prop:dec-monotonicity}.
Every regular solution with $h(0)=0$ is a scalar multiple of $h_0$,
so imposing $h(1)=0$ forces $h=0$.

The form of $-\mathcal A$ on $H^1_0((0,1))$ is
\[
\mathfrak a_{\mathrm d}(h,h)
\coloneq\int_0^1\frac{|h'|^2}{U}\dd\xi
-\frac14\int_0^1\frac{h^2}{\xi}\dd\xi.
\]
Since $U$ is smooth and bounded above and below by positive constants,
the weighted estimates \eqref{eq:weighted-endpoint-control} and the
compactness argument in the proof of
\cref{prop:increasing-nondegeneracy} apply here as well.
The negative term is bounded in absolute value by
$\max_{[0,1]}U\,\norm h_{\mathcal H}^2$.
Thus the form is closed and bounded below, and its self-adjoint
realization has compact resolvent.

To exclude a zero eigenvalue, let $h\in H^1_0((0,1))$ be a weak kernel
element. The first estimate in \eqref{eq:weighted-endpoint-control}
gives $h/\xi\in L^1((0,1))$. In the distributional equation
\[
\left(\frac{h'}U\right)'=-\frac{h}{4\xi},
\]
the right-hand side is integrable at zero. Thus $h'/U$, and hence
$h'$, extends continuously to zero. ODE regularity away from zero
then gives the regularity in \eqref{eq:dec-kernel}.
The regular kernel assertion already proved forces $h=0$.
Compactness of the resolvent now excludes zero from the spectrum.

Finally, set $\psi(\xi)\coloneq\Psi(R_\beta\sqrt\xi)$, using
\cref{lem:dec-comparison}. This is a smooth Dirichlet test function,
positive on $(0,1)$. By \eqref{eq:dec-operator-correspondence},
\[
\mathcal A\psi
=R_\beta^2(\mathcal L_{\mathrm d}\Psi)(R_\beta\sqrt\xi)>0
\qquad\text{for }\xi\in(0,1).
\]
Hence $\langle\mathcal A\psi,\psi\rangle_{\mathcal H}>0$, and the
variational characterization gives a positive eigenvalue.

For any eigenfunction with eigenvalue $\lambda$, the distributional
identity is
\[
\left(\frac{h'}U\right)'
=\left(\frac{\lambda}{U}-1\right)\frac{h}{4\xi}.
\]
The same argument gives regularity at zero, and bootstrapping gives a
smooth regular eigenfunction. The associated density mode is
$f(r)\coloneq h'(r^2)/\pi$. The endpoint conditions give zero mass, while the
eigenvalue equation at $\xi=1$ gives $h''(1)=0$, hence the Neumann
condition for $f$. The potential is normalized to have zero mean,
so its additive constant is not a kernel direction.
\end{proof}

\begin{proof}[Proof of \cref{thm:decreasing-uniqueness}]
Every strictly decreasing radial equilibrium belongs to the first
family \eqref{eq:shooting-family} with $\beta>0$, as proved in the
shooting construction. By \cref{prop:dec-monotonicity}, each mass
in $(8\pi,\Lambda\pi)$ determines exactly one such $\beta$, and
no decreasing equilibrium has a mass outside this interval.
The formulas \eqref{eq:shooting-family} and the smooth inverse of
$\beta\mapsto M_\beta$ give smooth dependence on $M$, in particular
in $C^2(\overline{B_1})$.
The nondegeneracy and positive-eigenvalue assertions follow from
\cref{prop:dec-spectral}.
\end{proof}

\begin{theorem}[The endpoints of the decreasing family]\label{thm:branch}
For $8\pi<M<\Lambda\pi$, let $(u_M^{\mathrm d},v_M^{\mathrm d})$ be the
unique strictly decreasing radial equilibrium of \cref{thm:equilibria}.
These equilibria form a smooth family parametrized by their mass $M$,
and every strictly decreasing radial equilibrium belongs to this family.
The endpoint limits are
\begin{align}
u_M^{\mathrm d}&\to\Lambda
&&\text{as }M\uparrow\Lambda\pi,\label{eq:small-end}\\
u_M^{\mathrm d}(x)\dd x&\rightharpoonup8\pi\delta_0
&&\text{as }M\downarrow8\pi.\label{eq:large-end}
\end{align}
The density convergence in \eqref{eq:small-end} holds in
$C^m(\overline\Omega)$ for every integer $m\ge0$;
the convergence in \eqref{eq:large-end} is weak convergence
of finite measures on $\overline\Omega$.
\end{theorem}

\begin{proof}[Proof of \cref{thm:branch}]
The shooting construction parametrizes every strictly decreasing
radial equilibrium by \eqref{eq:shooting-family} with $\beta>0$.
By \cref{prop:dec-monotonicity}, the mass map $\beta\mapsto M_\beta$
is a smooth strictly decreasing bijection from $(0,\infty)$ onto
$(8\pi,\Lambda\pi)$, with smooth inverse $\beta=\beta(M)$.
Thus
\[
(u_M^{\mathrm d},v_M^{\mathrm d})
=(u_{\beta(M)},v_{\beta(M)})
\]
is a smooth family parametrized by mass and contains all strictly
decreasing radial equilibria.
As $M\uparrow\Lambda\pi$, we have $\beta(M)\downarrow0$, whereas
as $M\downarrow8\pi$, we have $\beta(M)\to\infty$.
The endpoint limits now follow from \cref{prop:small,prop:bubble},
respectively.
\end{proof}

\subsection{Higher shooting families}

\begin{theorem}[Complete shooting parametrization]\label{thm:higher-equilibria}
For each $\beta\ne0$, let $Z_\beta$ solve \eqref{eq:shooting}.
Its positive critical radii form an infinite sequence of simple extrema
\[
0<R_{\beta,1}<R_{\beta,2}<\cdots\to\infty.
\]
For $k\ge1$ and $r\in[0,1]$, define
\begin{equation}\label{eq:all-shooting-family}
\begin{split}
\mu_{\beta,k}&\coloneq R_{\beta,k}^2,\qquad M_{\beta,k}\coloneq\pi R_{\beta,k}^2,\\
u_{\beta,k}(r)&\coloneq R_{\beta,k}^2e^{Z_\beta(R_{\beta,k}r)},\\
v_{\beta,k}(r)&\coloneq Z_\beta(R_{\beta,k}r)
-\av{Z_\beta(R_{\beta,k}|\cdot|)}.
\end{split}
\end{equation}
These are positive radial equilibria, and every nonconstant classical
radial equilibrium occurs for exactly one pair $(\beta,k)$.
For $k=1$, these coincide with \eqref{eq:shooting-family}:
\[
(R_{\beta,1},M_{\beta,1},u_{\beta,1},v_{\beta,1})
=(R_\beta,M_\beta,u_\beta,v_\beta).
\]
Thus the first family consists of the nonconstant monotone equilibria.
The families $k\ge2$ are exactly all non-monotone equilibria;
$u_{\beta,k}$ has precisely $k-1$ interior critical radii, all simple,
and $M_{\beta,k}>\Lambda\pi$.

For each fixed $k$, the family is real analytic on each side of
$\beta=0$ and extends real analytically through that point, with
$R_{0,k}\coloneq j_{1,k}$ and $M_{0,k}\coloneq\pi j_{1,k}^2$.
At $\beta=0$, this notation refers to the analytic extension, with
$(u_{0,k},v_{0,k})=(j_{1,k}^2,0)$, since $Z_0\equiv0$.
In particular,
\begin{equation}\label{eq:higher-small-end}
R_{\beta,k}\to j_{1,k},\qquad
M_{\beta,k}\to\pi j_{1,k}^2,\qquad
u_{\beta,k}\to j_{1,k}^2
\quad\text{as }\beta\to0.
\end{equation}
The density convergence holds in $C^m(\overline\Omega)$ for every
integer $m\ge0$.
For each $k\ge2$, a member with $\beta<0$ exists at every mass
$M>\pi j_{1,k}^2$.
At a prescribed mass $M$, all non-monotone equilibria are obtained by
solving the exact shooting conditions
\[
R_{\beta,k}=\sqrt{M/\pi},\qquad\text{for }\beta\in\R\setminus\{0\},\ k\in\{2,3,\ldots\}.
\]
For $k=1$, the mass map is a smooth strictly decreasing bijection
from $\beta<0$ onto $(\Lambda\pi,\infty)$ and from $\beta>0$ onto
$(8\pi,\Lambda\pi)$. No injectivity is asserted for the higher families $k\ge2$.
\end{theorem}

\begin{lemma}[Infinitely many simple extrema]\label{lem:all-extrema}
For every $\beta\ne0$, the positive zeros of $Z_\beta'$ form the
sequence $R_{\beta,k}$ in \cref{thm:higher-equilibria}, with
$R_{\beta,1}=R_\beta$. Moreover,
\begin{equation}\label{eq:extremum-signs}
\operatorname{sgn} Z_\beta(R_{\beta,k})
=(-1)^k\operatorname{sgn}\beta,\qquad
Z_\beta''(R_{\beta,k})=1-e^{Z_\beta(R_{\beta,k})}\ne0.
\end{equation}
Each map $\beta\mapsto R_{\beta,k}$ is real analytic on
$(0,\infty)$ and $(-\infty,0)$.
\end{lemma}

\begin{proof}
The energy estimate in \cref{lem:first-extremum} bounds
$|Z_\beta(s)|$ by a constant $K_\beta$ for all $s\ge0$. Set
\[
A_\beta(s)\coloneq\int_0^1e^{tZ_\beta(s)}\dd t
\ge e^{-K_\beta}\eqcolon c_->0.
\]
Thus $Z_\beta''+s^{-1}Z_\beta'+A_\beta(s)Z_\beta=0$.
For $s>0$, the function $w(s)\coloneq\sqrt{s}Z_\beta(s)$ satisfies
\[
w''+q(s)w=0,\qquad q(s)\coloneq A_\beta(s)+\frac1{4s^2}\ge c_-.
\]
Neither $Z_\beta$ nor $w$ can vanish together with its derivative
at a positive radius, by ODE uniqueness. To verify oscillation directly,
put $a\coloneq\sqrt{c_-}$ and choose a continuous angle on $s\ge1$ such that
\[
w=A\sin\theta,\qquad w'=aA\cos\theta,\qquad A>0.
\]
Differentiation gives
\[
\theta'=a\cos^2\theta+\frac{q(s)}a\sin^2\theta\ge a.
\]
Hence $\theta\to\infty$, and $Z_\beta$ has infinitely many
simple positive zeros tending to infinity. This is also the usual
Sturm oscillation argument; see \cite{Simon2005}.

Between two successive zeros of $Z_\beta$, its sign is constant, so
\[
(sZ_\beta')'=s(1-e^{Z_\beta})
\]
has a fixed strict sign. Consequently $sZ_\beta'$ is strictly monotone
on that interval. Rolle's theorem supplies one critical radius there,
and strict monotonicity makes it unique. There is no critical radius
before the first zero of $Z_\beta$, by its integral identity at zero.
All critical radii are therefore accounted for in this way.
Their values alternate in sign, giving \eqref{eq:extremum-signs}.
At a critical radius, $Z_\beta$ cannot also be zero, so the second
derivative in that equation is nonzero. The radii cannot accumulate
in a compact interval, and thus $R_{\beta,k}\to\infty$.

The integral formulation at zero and its regular even expansion give
real analytic dependence of $Z_\beta$ on $\beta$ near the origin;
away from zero this is the usual analytic dependence for an ODE.
The analytic implicit function theorem at each simple critical radius
then gives analyticity of $R_{\beta,k}$.
Simplicity prevents the ordered radii from exchanging their indices.
\end{proof}

\begin{proof}[Proof of \cref{thm:higher-equilibria}]
\cref{lem:all-extrema} supplies all the radii in the statement.
Integrating \eqref{eq:shooting} to any $R_{\beta,k}$ gives
\[
\int_0^{R_{\beta,k}}se^{Z_\beta(s)}\dd s
=\frac{R_{\beta,k}^2}{2}.
\]
Rescaling $s=R_{\beta,k}r$ proves that
\eqref{eq:all-shooting-family} solves \eqref{eq:system}, has the stated
mass, and satisfies both Neumann conditions. Its cumulative profile is
\begin{equation*}
F_{\beta,k}(\xi)\coloneq M_{\beta,k}\xi
-2\pi R_{\beta,k}\sqrt\xi\,
Z_\beta'(R_{\beta,k}\sqrt\xi),\qquad\text{for }\xi\in[0,1].
\end{equation*}
The density derivative has precisely the same zeros and signs as
$Z_\beta'(R_{\beta,k}r)$. Its interior critical radii are exactly
\[
\frac{R_{\beta,1}}{R_{\beta,k}},\ldots,
\frac{R_{\beta,k-1}}{R_{\beta,k}},
\]
and each is simple by \eqref{eq:extremum-signs}. Thus $k\ge2$ gives
non-monotone densities, and their masses exceed $\Lambda\pi$ by
\cref{cor:radial-shape}.

Conversely, let $(u_*,v_*)$ be any nonconstant radial equilibrium,
and set $\mu=M/\pi$ and $z\coloneq\log(u_*/\mu)$.
Then \eqref{eq:z} holds. The parameter $\beta=z(0)$ is nonzero,
since the regular initial-value problem with $\beta=0$ gives
$z\equiv0$. Uniqueness gives $z(r)=Z_\beta(\sqrt\mu\,r)$.
The boundary condition says that $\sqrt\mu$ is one of the critical
radii $R_{\beta,k}$. Both $\beta$ and its index $k$ are determined
by the equilibrium, which proves exhaustion and uniqueness of the
parametrization. \cref{prop:increasing-mass,prop:dec-monotonicity}
give the asserted smooth mass bijections on the two first-family
branches, including uniqueness at fixed mass. No injectivity is
asserted for the higher families.

For each fixed $k$, analytic dependence at $\beta=0$ gives
\[
Z_\beta=\beta J_0+O(\beta^2)\quad\text{in }C^m([0,S])
\]
for every $S>0$ and integer $m\ge0$.
The function $Z_\beta'(s)/\beta$ extends analytically to $\beta=0$
with value $-J_1(s)$. Its positive zeros $j_{1,k}$ are simple, so
the implicit function theorem gives
\[
R_{\beta,k}=j_{1,k}+O(\beta).
\]
The uniform expansion at the origin, together with convergence near
the first $k$ simple zeros of $J_1$, preserves their indexing.
Substitution into \eqref{eq:all-shooting-family} proves
\eqref{eq:higher-small-end} and the more precise normalized expansion
\begin{equation*}
\frac{u_{\beta,k}(r)}{\mu_{\beta,k}}
=1+\beta J_0(j_{1,k}r)+O(\beta^2)
\quad\text{in }C^m(\overline\Omega),
\end{equation*}
for every integer $m\ge0$.

For $\beta<0$, \cref{prop:small} gives
\[
R_{\beta,k}\ge R_{\beta,1}>2\sqrt{-\beta}\to\infty
\qquad\text{as }\beta\to-\infty.
\]
Continuity and the limit $M_{\beta,k}\to\pi j_{1,k}^2$ as
$\beta\uparrow0$ prove the stated existence range by the intermediate
value theorem. This range is sufficient; no sharper lower endpoint
for an individual family with $k\ge3$ is claimed.
\end{proof}

\subsection{The exact threshold for non-monotone equilibria}

\begin{theorem}[Exact mass threshold for non-monotone equilibria]
\label{thm:non-monotone-threshold}
The second shooting mass $\beta\mapsto M_{\beta,2}$ in
\cref{thm:higher-equilibria}, extended by $M_{0,2}=\pi j_{1,2}^2$,
is continuous on $\R$ and satisfies
\begin{equation}\label{eq:second-mass-coercivity}
M_{\beta,2}\to\infty\qquad\text{as }\beta\to\pm\infty.
\end{equation}
Its minimum
\begin{equation}\label{eq:non-monotone-threshold}
M_{\mathrm{nm}}\coloneq\min_{\beta\in\R}M_{\beta,2}
\end{equation}
satisfies the strict bounds
\begin{equation*}
\Lambda\pi<M_{\mathrm{nm}}<\pi j_{1,2}^2.
\end{equation*}
A non-monotone classical radial equilibrium of mass $M$ exists if and
only if $M\ge M_{\mathrm{nm}}$. For every such mass it can be chosen
from the second family $(u_{\beta,2},v_{\beta,2})$, with exactly one
interior radial extremum. In particular, the minimum is attained at
some $\beta_*\ne0$, so the threshold mass itself admits a nonconstant
equilibrium.
\end{theorem}

\begin{lemma}[Escape of the second critical radius]
\label{lem:second-radius-escape}
The second shooting mass satisfies \eqref{eq:second-mass-coercivity}.
\end{lemma}

\begin{proof}
For $\beta<0$, \cref{prop:small} gives
\[
R_{\beta,2}>R_\beta>2\sqrt{-\beta},\qquad
M_{\beta,2}>-4\pi\beta\to\infty
\quad\text{as }\beta\to-\infty.
\]

For $\beta\to+\infty$, reuse the decreasing family
$u_\beta(r)=R_\beta^2e^{Z_\beta(R_\beta r)}$ and its concentrated
endpoint in \cref{prop:bubble}.
Its strict decrease and weak convergence to $8\pi\delta_0$ imply
\[
\frac{3\pi}{4}u_\beta(1)
\le\int_{\Omega\setminus B_{1/2}}u_\beta\dd x\to0.
\]
Since $R_\beta^2\to8$, it follows that
\begin{equation*}
Z_\beta(R_\beta)
=\log\frac{u_\beta(1)}{R_\beta^2}\to-\infty.
\end{equation*}
On the other hand, \eqref{eq:shooting} gives
$(sZ_\beta')'=s(1-e^{Z_\beta})\le s$ for all $s>0$.
Integrating from $R_\beta$, where $Z_\beta'=0$, yields
\[
Z_\beta'(s)\le\frac{s^2-R_\beta^2}{2s}\le\frac s2,
\qquad
Z_\beta(s)\le Z_\beta(R_\beta)+\frac{s^2-R_\beta^2}{4}
\quad\text{for }s\in[R_\beta,\infty).
\]
By \eqref{eq:extremum-signs}, $Z_\beta(R_{\beta,2})>0$.
Consequently,
\[
R_{\beta,2}^2>R_\beta^2-4Z_\beta(R_\beta)\to\infty,
\]
which proves the other limit.
\end{proof}

\begin{lemma}[The second mass decreases through the constant state]
\label{lem:second-mass-slope}
The extension of the second mass curve at $\beta=0$ satisfies
\begin{equation}\label{eq:second-mass-slope}
\left.\frac{\mathrm d M_{\beta,2}}{\mathrm d\beta}\right|_{\beta=0}
=-\frac{\pi j_{1,2}^2}{2}
\frac{\displaystyle\int_0^1 rJ_0(j_{1,2}r)^3\dd r}
{\displaystyle\int_0^1 rJ_0(j_{1,2}r)^2\dd r}<0.
\end{equation}
\end{lemma}

\begin{proof}
Put $j\coloneq j_{1,2}$ and $\lambda\coloneq j^2$, and define
$\psi(r)\coloneq J_0(jr)$ and $z_\beta(r)\coloneq Z_\beta(R_{\beta,2}r)$
for $r\in[0,1]$.
The analytic extension in \cref{thm:higher-equilibria} gives
$z_\beta=\beta\psi+O(\beta^2)$ in $C^m([0,1])$ for every integer
$m\ge0$, and $\mu_{\beta,2}=\lambda+O(\beta)$.
The equation for $z_\beta$ is \eqref{eq:z} with
$\mu=\mu_{\beta,2}$, while
$-(r\psi')'=\lambda r\psi$ and $\psi'(0)=\psi'(1)=0$.
Multiplying by $r\psi$ and integrating by parts gives
\begin{equation*}
(\mu_{\beta,2}-\lambda)\int_0^1 rz_\beta\psi\dd r
=-\mu_{\beta,2}\int_0^1 r(e^{z_\beta}-1-z_\beta)\psi\dd r.
\end{equation*}
Division by $\beta^2$ and passage to the limit prove the equality in
\eqref{eq:second-mass-slope}, since $M_{\beta,2}=\pi\mu_{\beta,2}$.
For its strict sign, the Bessel integral inequality
\cite[Theorem~1.1]{Pan2022} states that
\[
\int_0^x tJ_0(t)^3\dd t>0\qquad\text{for }x\in(0,\infty).
\]
In particular,
$\int_0^1 r\psi^3\dd r=j^{-2}\int_0^j tJ_0(t)^3\dd t>0$.
\end{proof}

\begin{proof}[Proof of \cref{thm:non-monotone-threshold}]
By \cref{thm:higher-equilibria,lem:second-radius-escape}, the map
$\beta\mapsto M_{\beta,2}$ is continuous on $\R$ and tends to infinity
at both ends. It therefore attains its finite global minimum at some
$\beta_*\in\R$.
For every $\beta\ne0$, \cref{cor:radial-shape} gives
$M_{\beta,2}>\Lambda\pi$, and
$M_{0,2}=\pi j_{1,2}^2>\Lambda\pi$ as well.
Evaluation at a minimizer proves $M_{\mathrm{nm}}>\Lambda\pi$.
By \cref{lem:second-mass-slope}, small positive $\beta$ satisfy
$M_{\beta,2}<M_{0,2}=\pi j_{1,2}^2$.
This proves the strict upper bound and excludes $\beta_*=0$.

Every non-monotone equilibrium is represented by some $(\beta,k)$
with $\beta\ne0$ and $k\ge2$, by \cref{thm:higher-equilibria}.
Hence its mass obeys
\[
M=M_{\beta,k}\ge M_{\beta,2}\ge M_{\mathrm{nm}}.
\]
Conversely, take a minimizer $\beta_*\ne0$ and consider the half-line
from $\beta_*$ towards $+\infty$ if $\beta_*>0$, or towards $-\infty$
if $\beta_*<0$.
On this half-line, $M_{\beta,2}$ is continuous, starts at
$M_{\mathrm{nm}}$, and tends to infinity.
The intermediate value theorem therefore supplies, for every
$M\ge M_{\mathrm{nm}}$, a nonzero $\beta$ with $M_{\beta,2}=M$.
The equilibrium $(u_{\beta,2},v_{\beta,2})$ in
\eqref{eq:all-shooting-family} has exactly one interior radial extremum.
\end{proof}

\begin{remark}
The characterization \eqref{eq:non-monotone-threshold} requires neither
monotonicity of the second mass curve nor uniqueness of its minimizer.
\end{remark}

\begin{proof}[Proof of \cref{thm:equilibria}]
The constant equilibria are precisely $(\mu,0)$, one for each mass $M>0$.
Parts~(1) and~(3) follow from
low-mass stationary uniqueness \cite[Theorem~2.2]{Horstmann2011}
and \cref{cor:radial-shape}.
For part~(2), \cref{cor:radial-shape,thm:decreasing-uniqueness}
give the unique nonconstant equilibrium, which is strictly decreasing.
Part~(4) follows from \cref{cor:radial-shape,prop:increasing-mass},
which identify the unique nonconstant monotone equilibrium above $\Lambda\pi$.
Finally, \cref{thm:non-monotone-threshold} gives the exact existence
threshold for non-monotone equilibria, proving part~(5).
\end{proof}

\section{Concentration thresholds and qualitative dynamics}
\label{sec:qualitative-dynamics}

Throughout this section and \cref{sec:rates}, $\mathcal G$ denotes the
mean-zero Neumann inverse of $-\Delta$.

\subsection{Cumulative mass, comparison, and bounded orbits}

Write
\[
\U(\xi,t)\coloneq F_{u(\cdot,t)}(\xi)
=\int_{B_{\sqrt\xi}}u(x,t)\dd x,
\qquad \text{for }(\xi,t)\in[0,1]\times[0,T_{\max}).
\]
Set $\U_0(\xi)\coloneq F_{u_0}(\xi)$ for $\xi\in[0,1]$.
Radial integration of the elliptic equation gives
\begin{equation}\label{eq:radial-v}
\begin{aligned}
v_r(r,t)=\frac{\mu r}{2}
-\frac{\U(r^2,t)}{2\pi r},&\qquad \text{for }(r,t)\in(0,1]\times(0,T_{\max}),\\
\U_\xi(\xi,t)=\pi u(\sqrt\xi,t),&\qquad
\text{for }(\xi,t)\in[0,1]\times(0,T_{\max}).
\end{aligned}
\end{equation}
Integrating the first equation over $B_{\sqrt\xi}$ yields
\begin{equation}\label{eq:cumulative}
\begin{aligned}
\U_t=\N(\U)\coloneq
4\xi\U_{\xi\xi}+\frac{\U-M\xi}{\pi}\U_\xi,
&\qquad \text{for }(\xi,t)\in(0,1)\times(0,T_{\max}),\\
\U(0,t)=0,\quad\U(1,t)=M,
&\qquad \text{for }t\in[0,T_{\max}).
\end{aligned}
\end{equation}
For smooth solutions, $\partial_\nu u=0$ also gives
$\U_{\xi\xi}(1,t)=0$ for $t\in(0,T_{\max})$.

\begin{lemma}[Comparison and endpoint separation]\label{lem:comparison}
Cumulative ordering is preserved by \eqref{eq:cumulative}, also in comparison
with smooth stationary subsolutions or supersolutions having the same
endpoint values. If two distinct ordered classical solutions are bounded
on a positive-time interval, their difference $D$ at any later time $t_0$
satisfies
\begin{equation}\label{eq:gap}
D(\xi,t_0)\ge c(t_0)\xi(1-\xi),\qquad\text{for }\xi\in[0,1].
\end{equation}
The same assertion applies when one of the solutions is stationary.
\end{lemma}

\begin{proof}
Let $F_1,F_2$ denote the two cumulative solutions and set $D\coloneq F_1-F_2$.
In this proof, $T$ denotes the common right endpoint of the time
interval under consideration, so these functions are defined on
$[0,1]\times[0,T)$.
Subtraction gives the linear equation
\[
D_t=4\xi D''+\frac{F_1-M\xi}{\pi}D'
+\frac{F_2'}{\pi}D,
\qquad \text{for }(\xi,t)\in(0,1)\times(0,T).
\]
Both endpoint values are zero. Its comparison principle is also
\cite[Lemma~5.1]{Winkler2019}.
To handle the degeneracy at $\xi=0$, set
$V(\xi,t)\coloneq D(\xi,t)/\xi$ on $(0,1]\times(0,T)$,
with its regular extension at $\xi=0$.
Then
\[
\begin{aligned}
V_t=4\xi V''+
\left(8+\frac{F_1-M\xi}{\pi}\right)V'
+\left(\frac{F_1-M\xi}{\pi\xi}+\frac{F_2'}{\pi}\right)V,\\
\qquad \text{for }(\xi,t)\in(0,1)\times(0,T).
\end{aligned}
\]
Set $\widetilde V(r,t)\coloneq V(r^2,t)$ for
$(r,t)\in[0,1]\times(0,T)$. Under $\xi=r^2$, the principal part
$4\xi V_{\xi\xi}+8V_\xi$ becomes the radial Laplacian of
$\widetilde V$ in four dimensions. The remaining drift is regular at $r=0$,
because $F_1-M\xi=O(\xi)$.
The maximum principle proves comparison.
For a nonzero ordered difference, the strong maximum principle gives
$\widetilde V(r,t_0)>0$ for $0\le r<1$, and the boundary point lemma gives
$\widetilde V_r(1,t_0)<0$.
Compactness at this fixed time implies $V(\xi,t_0)\ge c(t_0)(1-\xi)$,
which proves \eqref{eq:gap}.
\end{proof}

\begin{lemma}[A cumulative upper bound]\label{lem:boundedness}
If, throughout the maximal existence interval,
\begin{equation}\label{eq:cumulative-bound}
0\le\U(\xi,t)\le C_0\xi,
\qquad \text{for }(\xi,t)\in[0,1]\times[0,T_{\max}),
\end{equation}
then $T_{\max}=\infty$ and $\sup_{t\ge0}\norm{u(t)}_\infty<\infty$.
\end{lemma}

\begin{proof}
\eqref{eq:cumulative-bound} gives the conclusion directly
by the Bernstein estimate in \cite[Lemma~4.1]{Winkler2019}.
Alternatively, A short proof via Neumann heat semigroup estimate is given in \cite[Lemma~3.5]{Mao2026a}.
\end{proof}

\begin{lemma}[Stationary limits of bounded solutions]\label{lem:omega}
A global uniformly bounded radial solution has a precompact positive-time
orbit in $C(\overline\Omega)$. Every element of its $\omega$-limit set is
a positive radial equilibrium. This limit set is nonempty, compact, and
connected.
\end{lemma}

\begin{proof}
Uniform elliptic estimates for $v$, followed by parabolic regularity and
bootstrapping on intervals of fixed length, give uniform smoothness for
$t\ge1$ and precompactness. The free energy
\begin{equation*}
\mathcal E(u(t))\coloneq\int_\Omega u(x,t)\log u(x,t)\dd x
-\frac12\int_\Omega|\nabla v(x,t)|^2\dd x,
\qquad \text{for }t\in[0,\infty),
\end{equation*}
is continuous on a uniformly bounded orbit and satisfies the dissipation
identity \cite[Lemma~2.4]{Mao2026a}:
\begin{equation*}
\frac{\dd}{\dd t}\mathcal E(u(t))
=-\int_\Omega u\left|\nabla(\log u-v)\right|^2\dd x,
\qquad \text{for }t\in(0,\infty).
\end{equation*}
The energy is bounded below along the bounded orbit, so it has a limit.
By continuous dependence, the limiting trajectory from any
$\omega$-limit point has constant energy. At positive times it is strictly
positive by the strong maximum principle. Therefore the dissipation vanishes,
$\log u-v$ is spatially constant, and the trajectory is stationary.
Continuity back to its initial time proves the assertion for the
$\omega$-limit point.

Precompactness gives nonemptiness and compactness of the limit set.
For each $T\ge1$, the closure in $C(\overline\Omega)$ of the trajectory
tail $\{u(t):t\ge T\}$ is compact and connected: the tail is the continuous
image of a connected interval. The $\omega$-limit set is the intersection
of these nested compact connected sets and is therefore connected.
\end{proof}

\subsection{Adjoint moments, blow-up, and nonlinear instability}

Fix any radial equilibrium $(u_*,v_*)$, and retain $F$ and $H=F-M\xi$.
Define
\[
P(\xi,t)\coloneq\U(\xi,t)-F(\xi),
\qquad \text{for }(\xi,t)\in[0,1]\times[0,T_{\max}).
\]
Equation \eqref{eq:cumulative} becomes
\begin{equation}\label{eq:perturbation}
\begin{aligned}
P_t&=L_FP+\frac1\pi PP',
&& \text{for }(\xi,t)\in(0,1)\times(0,T_{\max}),\\
L_F\phi&\coloneq4\xi\phi''+\frac H\pi\phi'+\frac{F'}{\pi}\phi,
&& \text{for }\xi\in(0,1),\\
P(0,t)&=P(1,t)=0,&&\text{for }t\in[0,T_{\max}).
\end{aligned}
\end{equation}
The adjoint with respect to $\dd\xi$ is
\begin{equation}\label{eq:adjoint}
L_F^*\psi\coloneq4\xi\psi''+\left(8-\frac H\pi\right)\psi'+\mu\psi,
\qquad \text{for }\xi\in(0,1).
\end{equation}
At $\xi=0$ we require regularity; at $\xi=1$ we impose $\psi(1)=0$.

\begin{lemma}[Positive principal growth rate]\label{lem:unstable}
If $H>0$ at some interior point, in particular if $u_*$ is decreasing
or non-monotone, then the following assertions hold.
There are $\sigma_F\in(0,\mu)$ and $\psi_F\in C^\infty([0,1])$ satisfying
\begin{equation*}
\begin{gathered}
L_F^*\psi_F=\sigma_F\psi_F,\qquad \text{for }\xi\in(0,1),\qquad \psi_F(1)=0,\\
\psi_F>0\text{ on }[0,1),\quad \psi_F'<0\text{ on }[0,1].
\end{gathered}
\end{equation*}
Moreover, $L_F$ has a positive eigenfunction $\phi_F$ for $\sigma_F$,
with
\begin{equation}\label{eq:phi-bound}
c\,\xi(1-\xi)\le\phi_F(\xi)\le C\,\xi(1-\xi),
\qquad \text{for }\xi\in[0,1].
\end{equation}
\end{lemma}

\begin{proof}
Write $\widehat v(\xi)\coloneq v_*(\sqrt\xi)$ for $\xi\in[0,1]$.
By \eqref{eq:radial-v}, $\widehat v'=-H/(4\pi\xi)$ on $(0,1)$.
With $a(\xi)\coloneq\xi e^{\widehat v(\xi)}$ for $\xi\in[0,1]$, we have
\begin{equation}\label{eq:self-adjoint}
L_F^*\psi=\frac1a(4\xi a\psi')'+\mu\psi,\qquad \text{for }\xi\in(0,1).
\end{equation}
The substitution $\xi=r^2$ identifies this with a regular radial
Dirichlet operator on the four-dimensional unit ball, with a smooth
positive weight. Its principal eigenfunction is positive, and its
largest eigenvalue has the Rayleigh characterization
\begin{equation}\label{eq:Rayleigh}
\sigma_F=\sup_{\psi\ne0,\ \psi(1)=0}
\frac{\mu\int_0^1a\psi^2\dd\xi
-4\int_0^1\xi a(\psi')^2\dd\xi}
{\int_0^1a\psi^2\dd\xi}.
\end{equation}
The admissible functions have finite radial energy and are regular at
the origin; the eigenfunction is smooth there.

The gradient term for the principal eigenfunction is nonzero, so
$\sigma_F<\mu$, for every radial equilibrium. Equation
\eqref{eq:self-adjoint} and regularity at zero give
\begin{equation*}
4\xi a(\xi)\psi_F'(\xi)
=(\sigma_F-\mu)\int_0^\xi a(s)\psi_F(s)\dd s<0,
\qquad\text{for }\xi\in(0,1].
\end{equation*}
At zero, $8\psi_F'(0)=(\sigma_F-\mu)\psi_F(0)<0$.
Thus $\psi_F'<0$ on $[0,1]$, independently of the sign of $H$.

By \eqref{eq:H}, the cumulative linearization satisfies
\begin{equation*}
L_FH=\frac1\pi HH',\qquad \text{for }\xi\in(0,1).
\end{equation*}
Let $(\alpha,\eta)$ be any connected component of $\{H>0\}$.
Its endpoints lie in $[0,1]$, with $H(\alpha)=H(\eta)=0$.
At an interior endpoint the zero is simple by \cref{lem:stationary-sign}, so
$H'(\alpha)>0$ if $\alpha>0$, and $H'(\eta)<0$ if $\eta<1$.
Green's identity on this interval gives
\begin{equation*}
\sigma_F\int_\alpha^\eta\psi_FH\dd\xi
 =-\frac1{2\pi}\int_\alpha^\eta\psi_F'H^2\dd\xi
+4\alpha\psi_F(\alpha)H'(\alpha)
-4\eta\psi_F(\eta)H'(\eta)>0.
\end{equation*}
Indeed, the first term is strictly positive and the endpoint terms
are nonnegative. At $\alpha=0$ its term vanishes by regularity;
at $\eta=1$ its term vanishes because $\psi_F(1)=0$.
Since the integral on the left is positive, $\sigma_F>0$.
This argument includes every non-monotone equilibrium because its
cumulative excess $H$ has both signs by
\eqref{eq:stationary-radial-derivative}.

Finally, $\phi_F\coloneq a\psi_F$ satisfies $L_F\phi_F=\sigma_F\phi_F$.
Its endpoint behavior is linear at both ends, which proves
\eqref{eq:phi-bound}. The eigenvalue equation also gives
$\phi_F''(1)=0$, the compatibility condition for the associated density
perturbation.
\end{proof}

\begin{proposition}[A Riccati moment and finite-time blow-up]\label{prop:above}
Let $(u_*,v_*)$ be a radial equilibrium of mass $M$. Assume that $u_*$
is strictly decreasing or non-monotone, or that $u_*\equiv\mu$ with
$M\ge\Lambda\pi$. Let $u$ be a nonnegative radial solution of
\eqref{eq:system} with the same mass.
For a nonconstant reference equilibrium, let $\psi_F$ and $\sigma_F>0$
be as in \cref{lem:unstable}; for $F=M\xi$, set
$\psi_F\coloneq\psi_0$ and $\sigma_F\coloneq\mu-\Lambda\ge0$.
Set
\[
A(t)\coloneq\int_0^1\psi_F(\xi)(\U(\xi,t)-F(\xi))\dd\xi,\qquad
\text{for }t\in[0,T_{\max}),
\]
and
\[
c_F\coloneq\left(2\pi\int_0^1\frac{\psi_F^2}{-\psi_F'}\dd\xi\right)^{-1}>0.
\]
Then
\begin{equation}\label{eq:Riccati}
\begin{aligned}
A'(t)&=\sigma_FA(t)-\frac1{2\pi}
\int_0^1\psi_F'(\xi)P(\xi,t)^2\dd\xi\\
&\ge\sigma_FA(t)+c_FA(t)^2,\qquad \text{for }t\in(0,T_{\max}).
\end{aligned}
\end{equation}
If $A(0)>0$, then
\begin{equation}\label{eq:blowup-time}
T_{\max}\le
\begin{cases}
\displaystyle\frac1{\sigma_F}
\log\left(1+\frac{\sigma_F}{c_FA(0)}\right),&\sigma_F>0,\\[2ex]
\displaystyle\frac1{c_FA(0)},&\sigma_F=0.
\end{cases}
\end{equation}
The case $\sigma_F=0$ occurs precisely for the constant equilibrium
at $M=\Lambda\pi$.
In particular, this applies to $u_0\succeq u_*$ with $u_0\ne u_*$.
No other radial equilibrium of mass $M$ can lie above $F$ in
cumulative order.
\end{proposition}

\begin{proof}
For the constant reference equilibrium with $M>\Lambda\pi$, finite-time
blow-up of every nontrivial concentration-increasing perturbation is
exactly \cite[Theorem~1.1(ii)]{Wang2026}, specialized to $n=2$ and $R=1$.
The common moment estimate below also covers cumulative differences
that change sign, gives the quantitative bound \eqref{eq:blowup-time},
and supplies the missing case $M=\Lambda\pi$.

By \cref{lem:unstable} and \eqref{eq:psi0-ode},
$L_F^*\psi_F=\sigma_F\psi_F$, $\sigma_F\ge0$, $\psi_F>0$ on $[0,1)$,
and $-\psi_F'$ is continuous and strictly positive on $[0,1]$.
Consequently,
\[
I_F\coloneq\int_0^1\frac{\psi_F^2}{-\psi_F'}\dd\xi\in(0,\infty),
\qquad c_F=\frac1{2\pi I_F}>0.
\]
All differentiations below are made for $0<t<T_{\max}$, when the
solution is smooth; the initial value of $A$ is recovered by continuity
at $t=0$. Recall $P=\U-F$ and $P(0,t)=P(1,t)=0$.
Multiplying \eqref{eq:perturbation} by $\psi_F$ gives
\[
A'(t)=\int_0^1\psi_F L_FP\dd\xi
+\frac1\pi\int_0^1\psi_FPP'\dd\xi.
\]
For the linear term, two integrations by parts and
the adjoint eigenvalue equation yield
\begin{align*}
\int_0^1\psi_F L_FP\dd\xi
&=\int_0^1P L_F^*\psi_F\dd\xi\\
&\quad+\left[4\xi\psi_FP'-(4\xi\psi_F)'P
 +\frac H\pi\psi_FP\right]_0^1\\
&=\sigma_FA(t).
\end{align*}
At $\xi=1$, the boundary terms vanish because $P(1,t)=\psi_F(1)=0$.
At $\xi=0$, we have $P=O(\xi)$, bounded $P'$, and smooth $\psi_F,H$,
so the terms vanish there as well.
For the nonlinear term, a further integration by parts gives
\[
\frac1\pi\int_0^1\psi_FPP'\dd\xi
=\frac1{2\pi}[\psi_FP^2]_0^1
-\frac1{2\pi}\int_0^1\psi_F'P^2\dd\xi
=-\frac1{2\pi}\int_0^1\psi_F'P^2\dd\xi.
\]
Combining these identities proves the equality in \eqref{eq:Riccati}.

To estimate the quadratic term, write
\[
A(t)=\int_0^1
\frac{\psi_F}{\sqrt{-\psi_F'}}
\bigl(\sqrt{-\psi_F'}\,P\bigr)\dd\xi.
\]
The Cauchy--Schwarz inequality gives
\[
A(t)^2\le I_F\int_0^1(-\psi_F')P^2\dd\xi,
\qquad \text{for }t\in(0,T_{\max}),
\]
and hence
\[
A'(t)\ge\sigma_FA(t)+\frac1{2\pi I_F}A(t)^2
=\sigma_FA(t)+c_FA(t)^2,\qquad \text{for }t\in(0,T_{\max}).
\]
This argument does not require $P$ to have a fixed sign.

Now suppose $A(0)>0$. If $\sigma_F>0$, the solution of
$a'=\sigma_Fa+c_Fa^2$, $a(0)=A(0)$, is
\[
a(t)=\frac{\sigma_FA(0)e^{\sigma_Ft}}
{\sigma_F-c_FA(0)(e^{\sigma_Ft}-1)}
\qquad\text{for }t\in[0,T_*),
\]
where $T_*$ is the first expression on the right of
\eqref{eq:blowup-time}.
At the secondary critical mass, $F=M\xi$ and $\sigma_F=0$, so the
same Cauchy--Schwarz estimate gives $A'\ge c_FA^2$ and the comparison
solution is instead
\[
a(t)=\frac{A(0)}{1-c_FA(0)t},\qquad \text{for }t\in[0,T_*),\qquad
T_*\coloneq\frac1{c_FA(0)}.
\]
By \eqref{eq:Riccati} and the ODE comparison principle,
$A(t)\ge a(t)$ for $0\le t<\min\{T_{\max},T_*\}$, and
$a(t)\to\infty$ as $t\uparrow T_*$.
On the other hand, nonnegativity and mass conservation imply
$0\le\U(\xi,t),F(\xi)\le M$, so
\[
|A(t)|\le M\int_0^1\psi_F(\xi)\dd\xi
\qquad\text{for }t\in[0,T_{\max}).
\]
If $T_{\max}>T_*$, this uniform bound contradicts the preceding lower
bound as $t\uparrow T_*$. Thus $T_{\max}\le T_*$, proving
\eqref{eq:blowup-time}.

If $u_0\succeq u_*$ and $u_0\ne u_*$, then
$P(\xi,0)=F_{u_0}(\xi)-F(\xi)\ge0$ and is not identically zero.
Indeed, equality of the cumulative profiles would imply equality of
the densities by \eqref{eq:radial-v}. Continuity and the strict
positivity of $\psi_F$ on $(0,1)$ therefore give $A(0)>0$, so the
finite-time blow-up conclusion applies.
Finally, a distinct radial equilibrium of the same mass above $F$
would give a constant moment $A>0$. Then \eqref{eq:Riccati} would imply
$0=A'\ge\sigma_FA+c_FA^2>0$, a contradiction.
\end{proof}

\begin{corollary}[Nonlinear instability of the threshold equilibria]
Every equilibrium in \cref{prop:above} is unstable among smooth
nonnegative radial data with its mass fixed.
\end{corollary}

\begin{proof}
For a nonconstant equilibrium and small $\varepsilon>0$, the profile
$F+\varepsilon\phi_F$ is increasing,
has the same endpoint values as $F$, and defines a positive density
\[
u_{0,\varepsilon}(r)\coloneq
\frac{F'(r^2)+\varepsilon\phi_F'(r^2)}{\pi},
\qquad \text{for }r\in[0,1].
\]
It satisfies the Neumann compatibility condition because
$F''(1)=\phi_F''(1)=0$.
Moreover, $u_{0,\varepsilon}\to u_*$ in $C^m(\overline\Omega)$
for every integer $m\ge0$ as $\varepsilon\downarrow0$.
For the constant equilibrium with $M\ge\Lambda\pi$, take instead
$u_0(r)\coloneq\mu+\varepsilon J_0(j_{1,1}r)$ for $r\in[0,1]$,
with $\varepsilon>0$ sufficiently small. These are smooth nonnegative radial data of mass $M$, satisfy
the Neumann compatibility condition, and obey
\[
\U_0(\xi)-M\xi
=\frac{2\pi\varepsilon}{j_{1,1}}\,\xi\psi_0(\xi)>0
\qquad\text{for }\xi\in(0,1).
\]
As $\varepsilon\downarrow0$, these data approach $\mu$ in every $C^m$ norm.
In both cases, the perturbed cumulative profile lies above $F$, so
\cref{prop:above} gives finite-time blow-up. Such trajectories necessarily
leave any fixed sufficiently small uniform-norm neighborhood of $u_*$.
\end{proof}

\begin{proof}[Proof of \cref{thm:eigenfunction-blowup}]
The eigenfunction assertions follow from \cref{lem:unstable} for
nonconstant equilibria and from \eqref{eq:psi0}--\eqref{eq:psi0-ode}
for the constant equilibrium. Apply \cref{prop:above} with
$A(0)=\mathcal A_*(u_0)$ to obtain the criterion and both bounds in
\eqref{eq:eigenfunction-blowup-time}.
\end{proof}

\begin{proof}[Proof of \cref{thm:concentration-thresholds}]
Let $F_*\coloneq F_{u_*}$ and let $\psi_*$ be the positive weight in
\cref{thm:eigenfunction-blowup}.
If $u_0\succeq u_*$ and $u_0\ne u_*$, then $F_{u_0}-F_*\ge0$
and this difference is not identically zero.
Since $\psi_*>0$ on $(0,1)$, we have $\mathcal A_*(u_0)>0$.
Thus \cref{thm:eigenfunction-blowup} gives finite-time blow-up.

If $u_0\preceq u_*$, comparison gives
$0\le\U(\xi,t)\le F_*(\xi)\le C\xi$.
\cref{lem:boundedness} proves global boundedness.
Equality gives the stationary solution by uniqueness.
\end{proof}

\subsection{Global boundedness and convergence below equilibria}

\begin{lemma}[The lower member of an ordered stationary pair]
\label{lem:ordered-stationary}
Let $M>\Lambda\pi$, and let $S$ and $F$ be distinct radial stationary
cumulative profiles of mass $M$. If $S\le F$ on $[0,1]$, then the
density associated with $S$ is strictly increasing.
\end{lemma}

\begin{proof}
Put $D\coloneq F-S\ge0$. If $D$ vanished at an interior point, it would also
satisfy $D'=0$ there; uniqueness for the difference equation would
force $D\equiv0$. Thus $D>0$ on $(0,1)$.
If the density associated with $S$ were decreasing or non-monotone,
\cref{prop:above} would exclude a distinct stationary profile above
$S$, contradicting the existence of $F$.
If $S=M\xi$ were constant, then $F-M\xi=D>0$ would make the
density associated with $F$ strictly decreasing by
\cref{lem:stationary-sign}. The mass bound in
\cref{cor:radial-shape} would then give $M<\Lambda\pi$, again a
contradiction. The only remaining possibility is a weakly increasing
nonconstant density, which is strictly increasing by
\cref{lem:stationary-sign}.
\end{proof}

\begin{proof}[Completion of the proof of \cref{thm:least-concentrated}]
It remains to prove pairwise incomparability of the equilibria other
than $(u_M^{\mathrm i},v_M^{\mathrm i})$.
If two distinct radial equilibria of mass $M>\Lambda\pi$ satisfy
$u_1\preceq u_2$, then
\cref{lem:ordered-stationary} makes $u_1$ strictly increasing.
Uniqueness in \cref{thm:equilibria} gives $u_1\equiv u_M^{\mathrm i}$.
Thus two distinct equilibria other than $(u_M^{\mathrm i},v_M^{\mathrm i})$
cannot be ordered in either direction. Their cumulative difference
therefore has both signs in $(0,1)$, since it vanishes at both endpoints.
\end{proof}

\begin{proposition}[Convergence below the constant]\label{prop:below-constant}
For every mass $M>0$, if $\U_0\le M\xi$, then the solution is global
and uniformly bounded. If $M\le\Lambda\pi$, then $u(t)\to\mu$ uniformly.
If $M>\Lambda\pi$ and $u_0\ne\mu$, the solution converges to the unique
increasing equilibrium $(u_M^{\mathrm i},v_M^{\mathrm i})$ in the sense of
\eqref{eq:increasing-limit}.
\end{proposition}

\begin{proof}
The profile $M\xi$ is stationary. Comparison gives
$0\le\U(\xi,t)\le M\xi$ for
$(\xi,t)\in[0,1]\times[0,T_{\max})$, so \cref{lem:boundedness} applies.
Every $\omega$-limit point is stationary by \cref{lem:omega}
and has cumulative profile at most $M\xi$.
For $M\le\Lambda\pi$, \cref{cor:radial-shape} makes this limit the
constant equilibrium. Precompactness then implies convergence of the
entire trajectory.

Now let $M>\Lambda\pi$. By \cref{prop:finite}, the only possible
stationary limits are $\mu$ and $u_M^{\mathrm i}$. The $\omega$-limit set
is nonempty and connected by \cref{lem:omega}, so it consists of
a single equilibrium $u_\infty$, with cumulative profile
$F_\infty\coloneq F_{u_\infty}$. Precompactness implies
$u(t)\to u_\infty$ uniformly along the entire trajectory.

To exclude $u_\infty\equiv\mu$ when $u_0\ne\mu$, define the
cumulative deficit and its positive weighted moment by
\begin{equation}\label{eq:deficit-moment}
\begin{aligned}
Q(\xi,t)&\coloneq M\xi-\U(\xi,t),
&& \text{for }(\xi,t)\in[0,1]\times[0,\infty),\\
A(t)&\coloneq\int_0^1\psi_0(\xi)Q(\xi,t)\dd\xi,
&& \text{for }t\in[0,\infty).
\end{aligned}
\end{equation}
Comparison gives $0\le Q\le M\xi$, and the strict separation in
\cref{lem:comparison} gives $A(t)>0$ for every $t>0$.
Equation~\eqref{eq:cumulative} becomes
\begin{equation}\label{eq:deficit-evolution}
\begin{aligned}
Q_t&=4\xi Q_{\xi\xi}+u(\sqrt\xi,t)Q,
&& \text{for }(\xi,t)\in(0,1)\times(0,\infty),\\
Q(0,t)&=Q(1,t)=0,&&\text{for }t\in[0,\infty).
\end{aligned}
\end{equation}
Applying \eqref{eq:weighted-linear-pairing} with $q=Q(\cdot,t)$ yields
\begin{equation}\label{eq:deficit-moment-evolution}
A'(t)=-\Lambda A(t)+
\int_0^1\psi_0(\xi)u(\sqrt\xi,t)Q(\xi,t)\dd\xi,
\qquad \text{for }t\in(0,\infty).
\end{equation}
If $u(t)\to\mu$, then $\mu>\Lambda$ gives a time $t_1>0$ such that
\[
A'(t)\ge\frac{\mu-\Lambda}{2}A(t),\qquad \text{for }t\in[t_1,\infty).
\]
This forces exponential growth of a positive moment, contradicting the
uniform bound
\[
0<A(t)\le M\int_0^1\xi\psi_0(\xi)\dd\xi,
\qquad \text{for }t\in(0,\infty).
\]
Hence $u_\infty$ is nonconstant. Since $F_\infty\le M\xi$,
\cref{prop:finite} identifies $u_\infty=u_M^{\mathrm i}$.
The strict inequalities $F_{u_M^{\mathrm i}}<M\xi$ and
$(u_M^{\mathrm i})'>0$
follow from \cref{lem:stationary-sign}, and
\eqref{eq:radial-v} gives $(v_M^{\mathrm i})'>0$.

Finally, $v(t)-v_M^{\mathrm i}$ is the zero-mean Neumann solution of
$-\Delta(v(t)-v_M^{\mathrm i})=u(t)-u_M^{\mathrm i}$.
For any $p>2$, elliptic regularity and Sobolev embedding give
\[
\norm{v(t)-v_M^{\mathrm i}}_{W^{1,\infty}}
\le C_p\norm{u(t)-u_M^{\mathrm i}}_{L^p}\to0
\qquad\text{as }t\to\infty,
\]
which proves \eqref{eq:increasing-limit}.
\end{proof}

\begin{proposition}[Convergence below the decreasing equilibrium]\label{prop:below}
Suppose $8\pi<M<\Lambda\pi$ and $F\coloneq F_{u_M^{\mathrm d}}$ is the cumulative
profile of the unique decreasing equilibrium. If $\U_0\le F$ and
$\U_0\ne F$, then the solution is global
and uniformly bounded, and $u(t)\to\mu$ uniformly.
\end{proposition}

\begin{proof}
Comparison gives $\U\le F\le C\xi$. Global boundedness follows from
\cref{lem:boundedness}.
Fix $t_0>0$. Since the initial profiles are distinct and ordered,
\cref{lem:comparison} gives
\[
F(\xi)-\U(\xi,t_0)\ge c(t_0)\xi(1-\xi),\qquad \text{for }\xi\in[0,1].
\]
Let $\phi_F$ be as in \cref{lem:unstable}.
For small $\varepsilon>0$, set
\[
G(\xi)\coloneq F(\xi)-\varepsilon\phi_F(\xi),\qquad \text{for }\xi\in[0,1].
\]
Since $H=F-M\xi>0$ has simple zeros at the endpoints, $H$ and $\phi_F$
are comparable to $\xi(1-\xi)$. Choose $\varepsilon$ so that
\[
M\xi<G<F\quad\text{for }\xi\in(0,1),\qquad G'>0,
\qquad \U(\cdot,t_0)\le G.
\]
The last condition follows from \eqref{eq:phi-bound} and the separation
estimate at $t_0$.
To compute $\N(G)$, recall from \eqref{eq:cumulative} that
$\N(V)=4\xi V''+(V-M\xi)V'/\pi$. Since $F$ is stationary,
$\N(F)=0$, and direct expansion gives
\begin{align*}
\N(G)
&=4\xi(F''-\varepsilon\phi_F'')
  +\frac{F-M\xi-\varepsilon\phi_F}{\pi}
   (F'-\varepsilon\phi_F')\\
&=\underbrace{4\xi F''+\frac{F-M\xi}{\pi}F'}_{\N(F)=0}
  -\varepsilon\left[4\xi\phi_F''
  +\frac{F-M\xi}{\pi}\phi_F'
  +\frac{F'}{\pi}\phi_F\right]
  +\frac{\varepsilon^2}{\pi}\phi_F\phi_F'\\
&=-\varepsilon L_F\phi_F
  +\frac{\varepsilon^2}{\pi}\phi_F\phi_F',
\end{align*}
where the bracket is $L_F\phi_F$ by \eqref{eq:perturbation}.
Using $L_F\phi_F=\sigma_F\phi_F$ therefore gives
\begin{equation}\label{eq:strict-super}
\N(G)=-\varepsilon\sigma_F\phi_F+
\frac{\varepsilon^2}{\pi}\phi_F\phi_F'<0,
\qquad\text{for }\xi\in(0,1),
\end{equation}
after reducing $\varepsilon$ if necessary. More explicitly, since
$\sigma_F>0$ and $\phi_F>0$ on $(0,1)$, we may also impose
$\varepsilon\norm{\phi_F'}_\infty/\pi\le\sigma_F/2$. Then
\[
\N(G)=-\varepsilon\phi_F
\left(\sigma_F-\frac{\varepsilon}{\pi}\phi_F'\right)
\le-\frac{\varepsilon\sigma_F}{2}\phi_F<0
\qquad\text{for }\xi\in(0,1).
\]

Let $B(\xi,t)$ be the solution of \eqref{eq:cumulative} with
$B(\xi,0)=G(\xi)$ for $\xi\in[0,1]$, initially defined on
$[0,1]\times[0,T_B)$, where $T_B$ is its maximal existence time.
Comparison gives $M\xi\le B(\xi,t)\le G(\xi)$ on this domain;
\cref{lem:boundedness} then gives $T_B=\infty$ and uniform boundedness.
Because $G$ is a stationary supersolution, $B(\cdot,t)\le G$;
the order-preserving semigroup then gives
$B(\cdot,t+s)\le B(\cdot,t)$ for $s,t\ge0$.
Uniform boundedness and positive-time compactness give
$B(\cdot,t)\to S$ in $C([0,1])$ as $t\to\infty$, with
$M\xi\le S\le G<F$ in the interior.
The semigroup property implies that $S$ is stationary.

If $S\ne M\xi$, \cref{lem:stationary-sign} makes $S$ the
cumulative profile of a strictly decreasing equilibrium.
But $F$ is a distinct equilibrium above $S$, contrary to
\cref{prop:above}. Therefore $S=M\xi$.

Comparison starting at $t_0$ gives
\[
\U(\xi,t_0+t)\le B(\xi,t)\to M\xi
\qquad\text{for }\xi\in[0,1]\text{ as }t\to\infty.
\]
Every stationary $\omega$-limit profile of $\U$ is thus at most $M\xi$.
\cref{cor:radial-shape}, since $M<\Lambda\pi$, makes it the constant profile.
\cref{lem:omega} and precompactness prove uniform convergence.
\end{proof}

\begin{proposition}[Convergence below any non-monotone equilibrium]
\label{prop:below-non-monotone}
Let $(u_*,v_*)$ be any non-monotone radial equilibrium of mass $M$,
with cumulative profile $F$. If $u_0\preceq u_*$ and $u_0\ne u_*$,
then the solution is global and uniformly bounded and converges to the
unique increasing equilibrium $(u_M^{\mathrm i},v_M^{\mathrm i})$
in the sense of \eqref{eq:increasing-limit}.
In particular, $(u_M^{\mathrm i})'(r)>0$ and $(v_M^{\mathrm i})'(r)>0$
for $0<r<1$.
\end{proposition}

\begin{proof}
By \cref{cor:radial-shape}, $M>\Lambda\pi$.
Comparison gives $0\le\U(\xi,t)\le F(\xi)\le C\xi$, so
\cref{lem:boundedness} proves global boundedness.
Fix $t_0>0$. Since the initial profiles are distinct and ordered,
\cref{lem:comparison} gives
\[
F(\xi)-\U(\xi,t_0)\ge c(t_0)\xi(1-\xi),\qquad \text{for }\xi\in[0,1].
\]
Take the positive eigenfunction $\phi_F$ and the positive growth rate
$\sigma_F$ supplied by \cref{lem:unstable}. For sufficiently
small $\varepsilon>0$, the profile
\[
G(\xi)\coloneq F(\xi)-\varepsilon\phi_F(\xi),\qquad \text{for }\xi\in[0,1],
\]
has $G'>0$, the same endpoint values as $F$, and satisfies
$\U(\cdot,t_0)\le G$. The last condition follows from
\eqref{eq:phi-bound}. The eigenvalue equation at $\xi=1$ gives
$\phi_F''(1)=0$, so $G$ also satisfies the Neumann compatibility
condition. The calculation \eqref{eq:strict-super} applies to
this equilibrium as well, giving $\N(G)<0$ on $(0,1)$ once
$\varepsilon\norm{\phi_F'}_\infty/\pi<\sigma_F$.

Let $B(\xi,t)$ be the cumulative solution starting at $G$, with
maximal existence time $T_B$ and domain $[0,1]\times[0,T_B)$.
Comparison with this stationary supersolution gives
\[
0\le B(\xi,t)\le G(\xi)\le C\xi,
\qquad \text{for }(\xi,t)\in[0,1]\times[0,T_B).
\]
It is therefore global and uniformly bounded. The order-preserving
semigroup and $B(\cdot,s)\le G$ imply
$B(\cdot,t+s)\le B(\cdot,t)$ for $s,t\ge0$.
Monotonicity and positive-time compactness give
$B(\cdot,t)\to S$ in $C([0,1])$ as $t\to\infty$;
the semigroup property makes $S$ stationary.
Since
\[
S\le G=F-\varepsilon\phi_F<F\quad\text{for }\xi\in(0,1),
\]
\cref{lem:ordered-stationary} shows that $S$ is the cumulative
profile of a strictly increasing equilibrium. By
\cref{prop:increasing-mass}, $S=F_{u_M^{\mathrm i}}<M\xi$ on $(0,1)$.

Comparison starting at $t_0$ yields
\[
\U(\xi,t_0+t)\le B(\xi,t)\to S(\xi)
\qquad\text{for }\xi\in[0,1]\text{ as }t\to\infty.
\]
Every stationary $\omega$-limit profile $K$ of the original solution
thus satisfies $K\le S$. \cref{thm:least-concentrated} also gives
$K\ge F_{u_M^{\mathrm i}}=S$, so every such profile equals $F_{u_M^{\mathrm i}}$.
The $\omega$-limit set is therefore the singleton $\{u_M^{\mathrm i}\}$.
Precompactness proves uniform convergence of the entire density
trajectory to $u_M^{\mathrm i}$. The zero-mean Neumann
elliptic estimate then proves convergence of the potentials in
$W^{1,\infty}$, as in \cref{prop:below-constant}.
Finally, $K-M\xi<0$ and \eqref{eq:radial-v} give the asserted
strict increase of both limiting profiles.
\end{proof}

\begin{remark}
The cumulative-order hypotheses allow non-monotone initial densities.
Data whose cumulative profiles cross the comparison equilibrium fall
outside these ordered classes.
\end{remark}

\begin{proof}[Proof of \cref{thm:ordered-convergence}]
Part~(1) is \cite[Theorem~1.1]{Mao2026a}.
Part~(2) is \cref{prop:below}.
Part~(3) follows from \cref{prop:below-constant} and the
rigidity at $M=\Lambda\pi$ in \cref{cor:radial-shape}.
For part~(4), \cref{prop:below-constant} treats the constant
threshold, and \cref{prop:below-non-monotone} treats every
non-monotone threshold. Uniqueness of the increasing equilibrium
identifies both limits with $(u_M^{\mathrm i},v_M^{\mathrm i})$.
In all four cases the zero-mean Neumann equation gives, for any $p>2$,
\[
\norm{v(t)-\underline v_M}_{W^{1,\infty}}
\le C_p\norm{u(t)-\underline u_M}_{L^p}\to0
\qquad\text{as }t\to\infty.
\]
The strict increase of the high-mass limit follows from
\cref{prop:below-constant,prop:below-non-monotone}.
\end{proof}

\subsection{Local comparison and asymptotic stability}

\begin{lemma}[Strict radial spectral stability of increasing equilibria]
\label{lem:increasing-gap}
Fix $M>\Lambda\pi$, and let $(u_M^{\mathrm i},v_M^{\mathrm i})$ be the
unique strictly increasing radial equilibrium of mass $M$.
Write $F_M^{\mathrm i}\coloneq F_{u_M^{\mathrm i}}$ and put
$H_M^{\mathrm i}\coloneq F_M^{\mathrm i}-M\xi<0$ on $(0,1)$.
The largest eigenvalue of its cumulative linearization $L_{F_M^{\mathrm i}}$
is $-\gamma_M<0$, where, with
$a_M(\xi)\coloneq\xi e^{v_M^{\mathrm i}(\sqrt\xi)}$ for $\xi\in[0,1]$,
\begin{equation}\label{eq:increasing-gap}
\gamma_M=
\inf_{\psi\ne0,\ \psi(1)=0}
\frac{4\int_0^1\xi a_M(\psi')^2\dd\xi
-\mu\int_0^1a_M\psi^2\dd\xi}
{\int_0^1a_M\psi^2\dd\xi}>0.
\end{equation}
The admissible functions have finite radial energy and are regular at zero,
as in \eqref{eq:Rayleigh}.
All radial fixed-mass eigenvalues are real and strictly negative.
\end{lemma}

\begin{proof}
The linearization formulas \eqref{eq:perturbation}--\eqref{eq:adjoint}
and the weighted self-adjoint identity \eqref{eq:self-adjoint} hold
for either sign of $H$, with $F=F_M^{\mathrm i}$ and $H=H_M^{\mathrm i}$.
Let $\sigma_M$ be the largest eigenvalue of $L_{F_M^{\mathrm i}}^*$,
and let $\psi_M$ be its positive principal eigenfunction.
The weighted operator is self-adjoint with compact resolvent, and
$L_{F_M^{\mathrm i}}(a_M\psi)=a_M L_{F_M^{\mathrm i}}^*\psi$.
Since $u_M^{\mathrm i}(\sqrt\xi)=C_M^*e^{v_M^{\mathrm i}(\sqrt\xi)}$
for some $C_M^*>0$, the cumulative weight in
\cref{prop:increasing-nondegeneracy} is
$m=(4\xi u_M^{\mathrm i}(\sqrt\xi))^{-1}=(4C_M^*a_M)^{-1}$.
Multiplication by $a_M$
therefore identifies the two weighted realizations up to a constant
factor in their norms. They have the same real spectrum, and
\cref{prop:increasing-nondegeneracy} makes it strictly negative.
Thus $\sigma_M<0$, and \eqref{eq:Rayleigh} gives
\eqref{eq:increasing-gap} with $\gamma_M=-\sigma_M$.
The principal eigenvalue is simple.
Its smooth eigenfunctions correspond to radial zero-mass density modes
by $P(\xi)=\pi\int_0^\xi f(\sqrt s)\dd s$ and
$f(r)=P'(r^2)/\pi$. The eigenvalue equation at $\xi=1$ gives
$P''(1)=0$, which is exactly the Neumann condition for $f$.
\end{proof}

\begin{lemma}[Uniform local comparison and decay estimates]
\label{lem:increasing-local-control}
Fix $M>\Lambda\pi$ and consider the unique increasing equilibrium
$(u_M^{\mathrm i},v_M^{\mathrm i})$.
Let $\gamma_M>0$ be the spectral gap in \eqref{eq:increasing-gap}.
There are $\varepsilon_M,C_M>0$, depending only on $M$,
such that every nonnegative radial datum of mass $M$ with
$\varepsilon\coloneq\norm{u_0-u_M^{\mathrm i}}_\infty\le\varepsilon_M$ generates a
global uniformly bounded solution satisfying
\begin{equation}\label{eq:increasing-local-rate}
\norm{u(t)-u_M^{\mathrm i}}_\infty
+\norm{v(t)-v_M^{\mathrm i}}_{W^{1,\infty}}
\le C_M\varepsilon e^{-\gamma_M t},\qquad\text{for }t\in[0,\infty).
\end{equation}
The datum need not be monotone or concentration ordered with $u_M^{\mathrm i}$.
\end{lemma}

\begin{proof}
Retain $F_M^{\mathrm i}$ from \cref{lem:increasing-gap}.
\cref{lem:increasing-gap} gives a positive principal eigenfunction
$p_1$ of $L_{F_M^{\mathrm i}}$ with
\[
L_{F_M^{\mathrm i}}p_1=-\gamma_M p_1,\qquad
c\xi(1-\xi)\le p_1(\xi)\le C\xi(1-\xi),\qquad \text{for }\xi\in[0,1].
\]
The endpoint bounds follow from regularity at zero and the boundary
point lemma at one, just as in \eqref{eq:phi-bound}.
Moreover, $p_1''(1)=0$ by its eigenvalue equation.

Put $P_0(\xi)\coloneq F_{u_0}(\xi)-F_M^{\mathrm i}(\xi)$ for $\xi\in[0,1]$ and
$\varepsilon\coloneq\norm{u_0-u_M^{\mathrm i}}_\infty$.
Equal mass and the uniform norm give
\[
|P_0(\xi)|\le\pi\varepsilon\min\{\xi,1-\xi\}
\le A\varepsilon p_1(\xi),\qquad \text{for }\xi\in[0,1],
\]
with $A$ depending only on the equilibrium and the normalization of $p_1$.
Choose $B\ge\norm{p_1'}_\infty/\pi$, $B>0$, and let
\[
a'=-\gamma_M a+Ba^2,\qquad a(0)=A\varepsilon.
\]
If $Ba(0)\le\gamma_M/2$, this solution is global, decreases, and
\begin{equation*}
0\le a(t)=\frac{a(0)e^{-\gamma_M t}}
{1-\frac{Ba(0)}{\gamma_M}(1-e^{-\gamma_M t})}
\le2A\varepsilon e^{-\gamma_M t},
\qquad \text{for }t\in[0,\infty).
\end{equation*}
Take $\varepsilon_M$ small enough that also
$(F_M^{\mathrm i})'\pm a(t)p_1'>0$ for all $t\ge0$.
Define the smooth time-dependent profiles
\[
F_\pm(\xi,t)\coloneq F_M^{\mathrm i}(\xi)\pm a(t)p_1(\xi),
\qquad \text{for }(\xi,t)\in[0,1]\times[0,\infty).
\]
They have endpoint values $0$ and $M$, and satisfy
\[
\begin{gathered}
\partial_tF_+-\N(F_+)=a^2p_1(B-p_1'/\pi)\ge0,\\
\partial_tF_--\N(F_-)=-a^2p_1(B+p_1'/\pi)\le0,
\qquad \text{for }(\xi,t)\in(0,1)\times(0,\infty).
\end{gathered}
\]
The comparison argument of \cref{lem:comparison} also applies to
these time-dependent barriers and gives
\begin{equation}\label{eq:increasing-local-cumulative}
|\U(\xi,t)-F_M^{\mathrm i}(\xi)|\le a(t)p_1(\xi)
\qquad\text{for }(\xi,t)\in[0,1]\times[0,T_{\max}).
\end{equation}
In particular, $\U(\xi,t)\le C_0\xi$ with $C_0$ uniform in this
neighborhood. \cref{lem:boundedness} and the uniform bound on
$\norm{u_0}_\infty$ give global existence and a uniform density bound.
The same comparison holds for continuous initial data by approximation
with smooth radial data of the same mass.

To recover a density estimate, write
$u_*\coloneq u_M^{\mathrm i}$, $v_*\coloneq v_M^{\mathrm i}$, and set
$q(x,t)\coloneq u(x,t)-u_*(x)$ and
$z(x,t)\coloneq v(x,t)-v_*(x)$ for
$(x,t)\in\overline\Omega\times[0,\infty)$, and
$P(\xi,t)\coloneq\U(\xi,t)-F_M^{\mathrm i}(\xi)$ for
$(\xi,t)\in[0,1]\times[0,\infty)$.
The radial Poisson equation and
\eqref{eq:increasing-local-cumulative} give
\[
\begin{gathered}
z_r(r,t)=-\frac{P(r^2,t)}{2\pi r},\qquad
|z_r(r,t)|\le Ca(t)r(1-r^2),
\quad \text{for }(r,t)\in(0,1]\times(0,\infty),\\
\norm{\nabla z(t)}_\infty\le C\varepsilon e^{-\gamma_M t},
\qquad \text{for }t\in[0,\infty).
\end{gathered}
\]
In particular, the potential energy
$E(t)\coloneq\frac12\norm{\nabla z(t)}_2^2$, defined for $t\in[0,\infty)$,
satisfies $E(t)\le C\varepsilon^2e^{-2\gamma_M t}$ on the same interval.
Since $-\Delta z=q$, testing
\[
q_t=\Delta q-\diver(q\nabla v_*+u\nabla z),
\qquad \text{for }(x,t)\in\Omega\times(0,\infty),
\]
by $z$ and using the Neumann boundary conditions yields
\begin{align*}
E'(t)
&=-\norm{q(t)}_2^2
  +\int_\Omega q\nabla v_*\cdot\nabla z\dd x
  +\int_\Omega u|\nabla z|^2\dd x\\
&\le-\frac12\norm{q(t)}_2^2+CE(t),\qquad \text{for }t\in(0,\infty).
\end{align*}
Thus, for $t\ge1$,
\[
\int_{t-1}^t\norm{q(s)}_2^2\dd s
\le2E(t-1)+C\int_{t-1}^tE(s)\dd s
\le C\varepsilon^2e^{-2\gamma_M t},\qquad \text{for }t\in[1,\infty).
\]

On the original two-dimensional disk, $q$ also satisfies
\begin{equation}\label{eq:increasing-local-smoothing}
\begin{aligned}
q_t&=\Delta q-\nabla v\cdot\nabla q+(u+u_*-\mu)q+f,
&& \text{for }(x,t)\in\Omega\times(0,\infty),\\
f(x,t)&\coloneq-\nabla u_*(x)\cdot\nabla z(x,t),
&& \text{for }(x,t)\in\overline\Omega\times[0,\infty),\\
\partial_\nu q&=0,&& \text{for }(x,t)\in\partial\Omega\times(0,\infty).
\end{aligned}
\end{equation}
The uniform density bound and elliptic regularity bound the drift and
zeroth-order coefficient independently of $u_0$ in the chosen neighborhood,
while $\norm{f(t)}_\infty\le C\varepsilon e^{-\gamma_M t}$.
By the Gaussian kernel estimate in \cite{Daners2000}, the Neumann
evolution operator $\mathcal T(t,s)$ for the homogeneous equation satisfies
\[
\norm{\mathcal T(t,s)h}_\infty\le C(t-s)^{-1/2}\norm h_2,
\qquad
\norm{\mathcal T(t,s)h}_\infty\le C\norm h_\infty,
\qquad\text{for }s\in[0,\infty),\ t\in(s,s+1].
\]
For each $t\ge1$, the preceding time-integrated estimate permits a choice
of $s\in[t-1,t-\frac12]$ with
$\norm{q(s)}_2\le C\varepsilon e^{-\gamma_M t}$.
Variation of constants therefore gives
\[
\norm{q(t)}_\infty
\le C\norm{q(s)}_2+C\int_s^t\norm{f(\tau)}_\infty\dd\tau
\le C\varepsilon e^{-\gamma_M t},\qquad\text{for }t\in[1,\infty).
\]
For $0\le t\le1$, the maximum principle applied to
\eqref{eq:increasing-local-smoothing} gives
\[
\norm{q(t)}_\infty
\le e^{Ct}\varepsilon+
\int_0^t e^{C(t-s)}\norm{f(s)}_\infty\dd s
\le C\varepsilon,\qquad \text{for }t\in[0,1].
\]
Finally, $\int_\Omega z(t)=0$ and the bound on $\nabla z(t)$ give
$\norm{z(t)}_{W^{1,\infty}}\le C\varepsilon e^{-\gamma_M t}$
for $t\in[0,\infty)$.
Increasing the uniform constant on $[0,1]$ proves
\eqref{eq:increasing-local-rate}.
\end{proof}

\begin{proof}[Proof of \cref{cor:stable-equilibria}]
The convergence theorem gives attraction from a full neighborhood of
$(\underline u_M,\underline v_M)$.
For $8\pi<M<\Lambda\pi$, take $S\coloneq F_{u_M^{\mathrm d}}$;
for $M>\Lambda\pi$, take $S(\xi)\coloneq M\xi$.
In both cases, with $I\coloneq\underline F_M$, there is $c_M>0$ such that
\[
S(\xi)-I(\xi)\ge c_M\xi(1-\xi),\qquad \text{for }\xi\in[0,1].
\]
For the decreasing threshold this follows from the simple endpoint
zeros noted after \eqref{eq:stationary-radial-derivative}.
For the increasing reference state, its cumulative deficit is
positive in the interior, and its endpoint slopes have the signs
of $\mu-u_M^{\mathrm i}(0)>0$ and $\mu-u_M^{\mathrm i}(1)<0$; its endpoint zeros are also simple.
Equal mass gives, for
$\varepsilon\coloneq\norm{u_0-\underline u_M}_\infty$,
\[
|F_{u_0}(\xi)-I(\xi)|
\le\pi\varepsilon\min\{\xi,1-\xi\}
\le2\pi\varepsilon\xi(1-\xi),\qquad \text{for }\xi\in[0,1].
\]
Hence every sufficiently small such perturbation lies strictly below
the unstable threshold $S$ and converges to
$(\underline u_M,\underline v_M)$ by
\cref{thm:ordered-convergence}.
For $M\le8\pi$, attraction holds for all radial data by part~(1).

Uniform local stability of the constant for $0<M<\Lambda\pi$,
in particular throughout $8\pi<M<\Lambda\pi$, follows from
\cite[Theorem~1.1(i)]{Wang2026}.
For $M>\Lambda\pi$, \cref{lem:increasing-local-control} supplies the
uniform estimate \eqref{eq:increasing-local-rate}, which gives local stability.
These local estimates, together with the preceding attraction, prove
local asymptotic stability without an order restriction.
Part~(1) of the convergence theorem also gives global asymptotic
stability for $M\le8\pi$.
Every other radial equilibrium belongs to one of the threshold classes
in \cref{thm:concentration-thresholds}. Arbitrarily small
concentration-increasing perturbations therefore lead to finite-time
blow-up, proving nonlinear instability and excluding any other stable
equilibrium.
At $M=\Lambda\pi$, the constant is the only radial equilibrium and is
nonlinearly unstable by the same blow-up argument, despite the one-sided
convergence in \cref{thm:ordered-convergence}(3).
\end{proof}

\appendix
\crefalias{section}{appendix}
\section{Sharp convergence rates}
\label{sec:rates}

\begin{theorem}[Convergence rates]\label{thm:rates}
Let $u$ be a radial solution of \eqref{eq:system} of mass $M$.
\begin{enumerate}
\item Suppose $M<\Lambda\pi$. Every global uniformly bounded solution
converging uniformly to the constant equilibrium satisfies
\begin{equation}\label{eq:exact-decay}
\norm{u(t)-\mu}_\infty+\norm{v(t)}_{W^{1,\infty}}
\le C e^{-(\Lambda-\mu)t},\qquad\text{for }t\in[0,\infty).
\end{equation}
In particular, this applies to the solutions converging to the constant
in \cref{thm:ordered-convergence}(1) and (2), as well as the
small perturbations in \cref{cor:stable-equilibria} with $M<\Lambda\pi$.
The exponent $\Lambda-\mu$ is optimal among small radial fixed-mass
perturbations of the constant equilibrium.
\item Suppose $M=\Lambda\pi$ and $u_0\preceq\mu$. Then
\begin{equation}\label{eq:critical-decay}
\norm{u(t)-\mu}_\infty+\norm{v(t)}_{W^{1,\infty}}
\le\frac{C}{1+t},\qquad\text{for }t\in[0,\infty).
\end{equation}
If $u_0\ne\mu$, then
\begin{equation}\label{eq:critical-profile}
\norm{t(u(t)-\mu)+\varphi_1/\kappa_c}_\infty
+\norm{t v(t)+\varphi_1/(\Lambda\kappa_c)}_{W^{1,\infty}}
\to0\qquad\text{as }t\to\infty,
\end{equation}
where $\varphi_1(r)\coloneq c_1J_0(j_{1,1}r)$ for $r\in[0,1]$, $c_1>0$ is chosen so that
$\norm{\varphi_1}_2=1$, and
$\kappa_c\coloneq\frac12\int_\Omega\varphi_1^3\dd x>0$.
Thus the rate $t^{-1}$ is sharp for every nonconstant datum in this class.
\item Suppose $M>\Lambda\pi$.
Every global uniformly bounded radial solution converging uniformly to
the unique increasing equilibrium $(u_M^{\mathrm i},v_M^{\mathrm i})$ satisfies
\begin{equation}\label{eq:increasing-rate}
\norm{u(t)-u_M^{\mathrm i}}_\infty
+\norm{v(t)-v_M^{\mathrm i}}_{W^{1,\infty}}
\le C e^{-\gamma_M t},\qquad\text{for }t\in[0,\infty),
\end{equation}
where $\gamma_M>0$ is the radial fixed-mass spectral gap of the
equilibrium, characterized by \eqref{eq:increasing-gap}.
By uniqueness of the increasing equilibrium, $\gamma_M$ depends only on $M$.
This applies to the trajectories in \cref{thm:ordered-convergence}(4)
and to the small perturbations in \cref{cor:stable-equilibria}
with $M>\Lambda\pi$.
There are $\varepsilon_M,C_M>0$
such that every nonnegative radial datum of mass $M$ with
$\varepsilon\coloneq\norm{u_0-u_M^{\mathrm i}}_\infty\le\varepsilon_M$ satisfies
\eqref{eq:increasing-local-rate}.
The exponent $\gamma_M$ is optimal among small radial fixed-mass
perturbations of $u_M^{\mathrm i}$.
\end{enumerate}
\end{theorem}

The constant $C$ may depend on the initial datum.
In part~(3), $\varepsilon_M$ and $C_M$ are the constants from
\cref{lem:increasing-local-control}; they depend only on $M$ and are
uniform on the stated neighborhood of $u_M^{\mathrm i}$.
The optimality claims concern the best exponents valid for all small
perturbations; a particular solution may decay faster.

\subsection{The constant equilibrium below the secondary threshold}

Let $q(x,t)\coloneq u(x,t)-\mu$ for
$(x,t)\in\overline\Omega\times[0,T_{\max})$.
With $\mathcal G$ as defined in
\cref{sec:qualitative-dynamics}, one has $v(t)=\mathcal Gq(t)$ for
$t\in[0,T_{\max})$ and
\begin{equation}\label{eq:q}
\begin{aligned}
q_t&=(\Delta_N+\mu)q-\diver(q\nabla\mathcal Gq),
&&\text{for }(x,t)\in\Omega\times(0,T_{\max}),\\
\int_\Omega q(t)\dd x&=0,&&\text{for }t\in[0,T_{\max}).
\end{aligned}
\end{equation}
The radial linear eigenvalues are $\mu-j_{1,k}^2$, $k\ge1$.
The largest is $-\gamma$, where
\begin{equation*}
\gamma\coloneq\Lambda-\mu>0\qquad\text{for }M\in(0,\Lambda\pi).
\end{equation*}

\begin{proposition}[Exact spectral decay at the constant]\label{prop:rate}
Fix $0<M<\Lambda\pi$ and $0<\alpha<1$.
There are $\varepsilon_0,C>0$ such that every smooth nonnegative radial
datum of mass $M$, with $\partial_\nu u_0=0$ and
$\varepsilon\coloneq\norm{u_0-\mu}_{C^{2+\alpha}}\le\varepsilon_0$,
generates a global solution satisfying
\begin{equation}\label{eq:local-exact}
\norm{u(t)-\mu}_\infty+\norm{v(t)}_{W^{1,\infty}}
\le C\varepsilon e^{-\gamma t},\qquad\text{for }t\in[0,\infty).
\end{equation}
More generally, every global uniformly bounded radial solution
converging uniformly to $\mu$ satisfies \eqref{eq:exact-decay}.
\end{proposition}

\begin{proof}
Global existence and preliminary exponential convergence for small
radial data are supplied by
\cite[Theorem~1.1(i)]{Wang2026}.
The comparison estimates in \cite[Lemmas~3.2 and 3.3]{Wang2026},
followed by positive-time smoothing, give constants $C_*,\kappa>0$
uniform on a sufficiently small $C^{2+\alpha}$ neighborhood such that
\begin{equation}\label{eq:preliminary}
\norm{q(t)}_\infty\le C_*e^{-\kappa t},\qquad\text{for }t\in[0,\infty).
\end{equation}
On $0\le t\le1$, local continuous dependence also gives
$\norm{q(t)}_\infty\le C\varepsilon$.
We use these established preliminary estimates only to obtain
integrability of $\norm{q(t)}_\infty$; the remaining argument recovers
the full spectral exponent $\gamma$.

Testing \eqref{eq:q} by $q$ gives
\begin{equation*}
\frac12\frac{\dd}{\dd t}\norm q_2^2
=-\norm{\nabla q}_2^2+\mu\norm q_2^2
+\frac12\int_\Omega q^3\dd x,\qquad \text{for }t\in(0,\infty).
\end{equation*}
Here the nonlinear term equals
$\int q\nabla q\cdot\nabla v=-\frac12\int q^2\Delta v
=\frac12\int q^3$.
By \eqref{eq:poincare},
\[
\frac{\dd}{\dd t}\norm q_2^2
\le\bigl(-2\gamma+\norm q_\infty\bigr)\norm q_2^2,
\qquad \text{for }t\in(0,\infty).
\]
Equation \eqref{eq:preliminary} makes $\norm{q(t)}_\infty$ integrable.
Gronwall's inequality yields
\begin{equation}\label{eq:q-exact-L2}
\norm{q(t)}_2\le C\varepsilon e^{-\gamma t},\qquad \text{for }t\in[0,\infty).
\end{equation}

To obtain the uniform norm, rewrite \eqref{eq:q} as
\[
\begin{aligned}
q_t&=\Delta q-\nabla v\cdot\nabla q+(\mu+q)q,
&&\text{for }(x,t)\in\Omega\times(0,\infty),\\
\partial_\nu q&=0,&&\text{for }(x,t)\in\partial\Omega\times(0,\infty).
\end{aligned}
\]
Its drift and zeroth-order coefficient are uniformly bounded.
The Neumann evolution operator $P(t,s)$ therefore satisfies
\[
\norm{P(t,s)f}_\infty\le C(t-s)^{-1/2}\norm f_2,
\qquad\text{for }s\in[0,\infty),\ t\in(s,s+1]\text{ and }f\in L^2(\Omega),
\]
by the Gaussian kernel estimate \cite{Daners2000}.
Indeed, in dimension two the $L^2$ norm in the second variable of
$C(t-s)^{-1}\exp(-|x-y|^2/[C(t-s)])$ is at most
$C(t-s)^{-1/2}$.
For $t\ge1$, apply this with $s=t-1$ and
$q(t)=P(t,t-1)q(t-1)$, using \eqref{eq:q-exact-L2}.
For $0\le t\le1$, use the local continuous-dependence estimate.
Finally, elliptic regularity gives
$\norm{v(t)}_{W^{1,\infty}}\le C\norm{q(t)}_\infty$ for $t\ge0$.
This proves \eqref{eq:local-exact}.

For a bounded solution already converging uniformly, positive-time
regularity and interpolation imply $q(t)\to0$ in
$C^{2+\alpha}(\overline\Omega)$.
Apply the preceding argument after a sufficiently large time.
Changing the constant controls the preceding finite interval and proves
\eqref{eq:exact-decay}.
\end{proof}

\begin{proposition}[Optimality of the constant-state decay exponent]\label{prop:optimal}
For every $0<M<\Lambda\pi$, the decay exponent $\gamma=\Lambda-\mu$
in \eqref{eq:exact-decay} is optimal among small radial fixed-mass
perturbations of the constant equilibrium.
\end{proposition}

\begin{proof}
Let $\varphi_1(x)=c_1J_0(j_{1,1}|x|)$ for $x\in\overline\Omega$,
with $c_1>0$ chosen so that $\norm{\varphi_1}_2=1$.
Take $u_0\coloneq\mu+\varepsilon\varphi_1$, $\varepsilon>0$ small.
The datum is positive, has mass $M$, and satisfies the Neumann condition.
For $a(t)\coloneq\int_\Omega q(t)\varphi_1\dd x$, $t\in[0,\infty)$,
equation \eqref{eq:q} gives
\[
a'=-\gamma a+R(t),\qquad
R(t)\coloneq\int_\Omega q\nabla v\cdot\nabla\varphi_1\dd x,
\qquad \text{for }t\in(0,\infty).
\]
\cref{prop:rate} gives
$|R(t)|\le C\varepsilon^2e^{-2\gamma t}$ for $t>0$.
Hence
\[
e^{\gamma t}a(t)\to
\varepsilon+\int_0^\infty e^{\gamma s}R(s)\dd s
=\varepsilon+O(\varepsilon^2)=A_\infty\ne0\qquad\text{as }t\to\infty.
\]
Since $|a(t)|\le\norm{q(t)}_\infty\norm{\varphi_1}_1$,
a faster exponential estimate is impossible for this solution.
\end{proof}

\subsection{Critical algebraic decay and the leading profile}

\begin{proposition}[Sharp critical decay by spectral decomposition]
\label{prop:critical-rate}
Let $M=\Lambda\pi$, so $\mu=\Lambda$, and suppose $u_0\preceq\mu$.
Then \eqref{eq:critical-decay} holds. If $u_0\ne\mu$, the normalized
first radial Neumann eigenfunction $\varphi_1$ and
the positive coefficient $\kappa_c$ in \cref{thm:rates} satisfy
\begin{equation}\label{eq:critical-expansion}
\norm{u(t)-\mu+\varphi_1/(\kappa_c t)}_\infty
+\norm{v(t)+\varphi_1/(\Lambda\kappa_c t)}_{W^{1,\infty}}
=O\left(\frac{\log t}{t^2}\right)\qquad\text{as }t\to\infty.
\end{equation}
In particular, the convergence is of exact order $t^{-1}$.
\end{proposition}

\begin{proof}
The constant datum gives the stationary solution, so assume $u_0\ne\mu$.
\cref{prop:below-constant} gives global boundedness and
$q=u-\Lambda\to0$ uniformly. Positive-time regularity provides uniform
$C^m(\overline\Omega)$ bounds for $t\ge1$ and every integer $m\ge0$.
At this mass the operator $L\coloneq\Delta_N+\Lambda$ on radial mean-zero
functions has a simple eigenvalue zero with eigenfunction $\varphi_1$;
all its other eigenvalues are $\Lambda-j_{1,k}^2<0$, $k\ge2$.
Decompose the solution into these eigenspaces:
\begin{equation}\label{eq:critical-splitting}
\begin{aligned}
q(t)&=a(t)\varphi_1+z(t),&
a(t)&\coloneq\int_\Omega q(t)\varphi_1\dd x,\\
\int_\Omega z(t)\varphi_1\dd x&=0,&&\text{for }t\in[0,\infty).
\end{aligned}
\end{equation}

We first control the full perturbation by its first-mode amplitude.
Retain $Q$ and $A$ from \eqref{eq:deficit-moment}, and set
$\Phi_1(\xi)\coloneq\varphi_1(\sqrt\xi)$ for $\xi\in[0,1]$.
Since $\Phi_1'=-c_1j_{1,1}\psi_0/2$, integration by parts gives
\begin{equation}\label{eq:critical-amplitude}
\begin{aligned}
a(t)=\int_0^1(-Q_\xi)\Phi_1\dd\xi
=-\frac{c_1j_{1,1}}2 A(t),\qquad
b(t)\coloneq -a(t)>0,\\
\text{for }t\in[0,\infty).
\end{aligned}
\end{equation}
The positivity follows from \cref{lem:comparison}.
At $\mu=\Lambda$, equations~\eqref{eq:deficit-evolution} and
\eqref{eq:deficit-moment-evolution} give the exact identities
\begin{equation}\label{eq:critical-moment}
A'=\int_0^1\psi_0q(\sqrt\xi,t)Q\dd\xi
=\frac1{2\pi}\int_0^1\psi_0'Q^2\dd\xi<0,
\qquad \text{for }t\in(0,\infty).
\end{equation}
Here $q(\sqrt\xi,t)=-Q_\xi(\xi,t)/\pi$ for
$(\xi,t)\in[0,1]\times(0,\infty)$.
Consequently $b$ decreases to zero and
\begin{equation*}
\begin{aligned}
\left|\frac{b'(t)}{b(t)}\right|\le\norm{q(t)}_\infty,
\qquad \text{for }t\in(0,\infty),\\
\norm{q(t)}_\infty\to0\qquad\text{as }t\to\infty.
\end{aligned}
\end{equation*}
For every $\varepsilon>0$ there is $T_\varepsilon$ such that
\begin{equation}\label{eq:critical-ratio}
b(s)\le b(t)e^{\varepsilon(t-s)},\qquad
b(t)\ge b(T_\varepsilon)e^{-\varepsilon(t-T_\varepsilon)}
\qquad\text{for }T_\varepsilon\le s\le t.
\end{equation}

To turn this amplitude into a norm bound, we use the stable spectral
subspace of \eqref{eq:q}.
Fix $p>2$ and $1/2+1/p<\theta<1$, and define
$X_\theta\coloneq D((1-\Delta_N)^\theta)\cap L^p_{\mathrm{rad},0}(\Omega)$,
where $\Delta_N$ is the Neumann Laplacian and
$L^p_{\mathrm{rad},0}(\Omega)$ denotes the radial mean-zero subspace.
Then $X_\theta\hookrightarrow C^1(\overline\Omega)$.
Let $P_1f\coloneq(\int_\Omega f\varphi_1\dd x)\varphi_1$ for
$f\in L^p_{\mathrm{rad},0}(\Omega)$, and set
$\mathcal B(q)\coloneq-\diver(q\nabla\mathcal Gq)$ for $q\in X_\theta$.
Projection of \eqref{eq:q} gives
\begin{equation}\label{eq:critical-modal-equations}
a'=\int_\Omega q\nabla\mathcal Gq\cdot\nabla\varphi_1\dd x,
\qquad z_t=Lz+(I-P_1)\mathcal B(q),\qquad \text{for }t\in(0,\infty).
\end{equation}
Choose $0<\eta<j_{1,2}^2-\Lambda$ and $0<\varepsilon<\eta/2$.
On the stable subspace, the analytic semigroup satisfies
\begin{align*}
\norm{e^{rL}(I-P_1)f}_{X_\theta}
&\le C(1+r^{-\theta})e^{-\eta r}\norm f_p,\\
\norm{e^{rL}(I-P_1)g}_{X_\theta}
&\le Ce^{-\eta r}\norm g_{X_\theta},\qquad\text{for }r\in(0,\infty),
\end{align*}
for radial mean-zero $f\in L^p(\Omega)$ and $g\in X_\theta$.
Since $\mathcal B(q)=q^2-\nabla q\cdot\nabla\mathcal Gq$ and
$\norm{\nabla\mathcal Gq}_\infty\le C\norm q_\infty$, we have
\[
\norm{\mathcal B(q)}_p
\le C\norm q_\infty\norm q_{X_\theta},\qquad \text{for }q\in X_\theta.
\]
For any $\delta>0$, uniform convergence of $q$ to zero allows us to
choose $T\ge\max\{1,T_\varepsilon\}$ such that
$\norm{q(t)}_\infty\le\delta$ for $t\ge T$.
Writing $Z(t)\coloneq\norm{z(t)}_{X_\theta}$ for $t\in(0,\infty)$ and using
$q=-b\varphi_1+z$, we obtain
\[
\norm{\mathcal B(q(t))}_p\le C\delta\bigl(b(t)+Z(t)\bigr),
\qquad\text{for }t\in[T,\infty).
\]
Variation of constants in \eqref{eq:critical-modal-equations} gives
\[
Z(t)\le Ce^{-\eta(t-T)}Z(T)
+C\delta\int_T^t(1+(t-s)^{-\theta})e^{-\eta(t-s)}
\bigl(b(s)+Z(s)\bigr)\dd s,\qquad\text{for }t\in[T,\infty).
\]
For each finite $R>T$, the quantity
\[
H_R\coloneq\sup_{T\le t\le R}\frac{Z(t)}{b(t)},\qquad \text{for }R\in(T,\infty),
\]
is finite by positive-time regularity and positivity of $b$.
Divide the preceding estimate by $b(t)$ and use
\eqref{eq:critical-ratio}. Taking the supremum on $[T,R]$ yields
\[
H_R\le C\frac{Z(T)}{b(T)}+C\delta J(1+H_R),\qquad
\text{for }R\in(T,\infty),
\]
where
\[
J\coloneq\int_0^\infty(1+r^{-\theta})e^{-(\eta-\varepsilon)r}\dd r<\infty.
\]
Here $J$ is finite because $\theta<1$ and $\varepsilon<\eta$.
Choose $\delta$ so small that $C\delta J\le1/2$, and then choose $T$
as above. Absorbing the $H_R$ term gives a bound independent of $R$.
Consequently $Z(t)\le Cb(t)$ for $t\ge T$, and hence
\begin{equation}\label{eq:critical-norm-amplitude}
\norm{q(t)}_{C^1}\le C\norm{q(t)}_{X_\theta}\le Cb(t),
\qquad\text{for }t\in[T,\infty).
\end{equation}
This argument uses only uniform convergence of $q$, positivity of $b$,
and \eqref{eq:critical-ratio}; no decay rate for the stable component
has been assumed.

We can now sharpen the estimate for that component. By
\eqref{eq:critical-norm-amplitude},
$\norm{\mathcal B(q(t))}_p\le C\norm{q(t)}_{C^1}^2\le Cb(t)^2$ for $t\ge T$.
Variation of constants gives
\begin{align}
\norm{z(t)}_{X_\theta}
&\le C e^{-\eta(t-T)}\norm{z(T)}_{X_\theta}
+C\int_T^t(1+(t-s)^{-\theta})e^{-\eta(t-s)}b(s)^2\dd s\notag\\
&\le C e^{-\eta(t-T)}
+C b(t)^2\int_0^\infty(1+r^{-\theta})e^{-(\eta-2\varepsilon)r}\dd r
=O(b(t)^2).\label{eq:critical-stable-remainder}
\end{align}
These estimates hold for $t\ge T$, with the remainder understood as
$t\to\infty$.
The integral is finite because $\theta<1$.
The second inequality in \eqref{eq:critical-ratio} also absorbs the
exponential transient into $O(b(t)^2)$.

We now compute the leading quadratic coefficient, including its sign.
Since $\mathcal G\varphi_1=\varphi_1/\Lambda$, the first equation in
\eqref{eq:critical-modal-equations} and
\eqref{eq:critical-stable-remainder} imply
\begin{equation}\label{eq:critical-scalar}
\begin{aligned}
a'=\kappa_c a^2+O(b^3),\qquad
b'=-\kappa_c b^2+O(b^3),\qquad\text{as }t\to\infty,\\
\kappa_c=\frac1\Lambda\int_\Omega
\varphi_1|\nabla\varphi_1|^2\dd x
=\frac12\int_\Omega\varphi_1^3\dd x.
\end{aligned}
\end{equation}
The last equality is the Neumann integration-by-parts identity for
$-\Delta\varphi_1=\Lambda\varphi_1$.
Its strict positivity follows directly from the cumulative weight.
Indeed, put
\[
h_1(\xi)\coloneq\int_{B_{\sqrt\xi}}\varphi_1\dd x
=\frac{2\pi c_1}{j_{1,1}}\sqrt\xi J_1(j_{1,1}\sqrt\xi),
\qquad \text{for }\xi\in[0,1].
\]
Then $h_1(\xi)>0$ for $\xi\in(0,1)$.
Comparing the quadratic terms in \eqref{eq:critical-moment} when
$q=a\varphi_1$ and $Q=-a h_1$, and using
\eqref{eq:critical-amplitude}, gives
\begin{equation*}
\kappa_c=-\frac{c_1j_{1,1}}{4\pi}
\int_0^1\psi_0'(\xi)h_1(\xi)^2\dd\xi>0.
\end{equation*}
Thus no separate positivity result for the cubic Bessel integral is needed.

Since $b>0$ and $b\to0$, \eqref{eq:critical-scalar} gives
$(1/b)'=\kappa_c+O(b)$ as $t\to\infty$ and initially yields
$b(t)\asymp t^{-1}$ as $t\to\infty$.
Integrating once more gives
\[
\frac1{b(t)}=\kappa_c t+O(\log t),\qquad
b(t)=\frac1{\kappa_c t}+O\left(\frac{\log t}{t^2}\right),\qquad
\norm{z(t)}_{X_\theta}=O(t^{-2})\qquad\text{as }t\to\infty.
\]
Because $X_\theta\hookrightarrow C^1$,
\eqref{eq:critical-splitting} gives the asserted expansion for $u$.
Applying $\mathcal G$, together with its elliptic estimate into
$W^{1,\infty}$ and $\mathcal G\varphi_1=\varphi_1/\Lambda$, gives the
expansion for $v$.
Changing the constant on the preceding finite time interval proves
\eqref{eq:critical-decay}. The nonzero leading profile proves sharpness.
All reductions above are direct projections and semigroup estimates along
the given trajectory.
\end{proof}

\subsection{Exponential decay and optimality at the increasing equilibrium}

\begin{proposition}[Exact exponential rate at the increasing limit]
\label{prop:increasing-rate}
Fix $M>\Lambda\pi$. Every global bounded radial solution of mass $M$
converging uniformly to $(u_M^{\mathrm i},v_M^{\mathrm i})$ satisfies
\eqref{eq:increasing-rate}, with the positive spectral gap
\eqref{eq:increasing-gap} of this equilibrium.
In particular, this applies when $M>\Lambda\pi$, $u_0\preceq\mu$, and
$u_0\ne\mu$, and to every distinct datum less concentrated than a
non-monotone equilibrium of the same mass.
The exponent is optimal among small radial fixed-mass perturbations of
the equilibrium, including perturbations in the class $u_0\preceq\mu$.
\end{proposition}

\begin{proof}
Retain $F_M^{\mathrm i}$ and $H_M^{\mathrm i}$ from \cref{lem:increasing-gap}.
By uniform convergence, choose $T>0$ so that $u(T)$ lies in the
neighborhood of \cref{lem:increasing-local-control}.
Applying that lemma with initial time $T$ gives
\eqref{eq:increasing-rate}; increasing the constant covers $[0,T]$.
For optimality, write $u_*\coloneq u_M^{\mathrm i}$ and
$L\coloneq L_{F_M^{\mathrm i}}$. By
\cref{prop:increasing-nondegeneracy,lem:increasing-gap}, $L$ is
self-adjoint in
\[
\mathcal H=L^2((0,1),m(\xi)\dd\xi),\qquad
m(\xi)=\frac1{4\xi u_*(\sqrt\xi)},\quad\text{for }\xi\in(0,1],
\]
and has a positive principal eigenfunction $p_1$ with
$Lp_1=-\gamma_Mp_1$. Normalize $\norm{p_1}_{\mathcal H}=1$ and take
\[
\begin{aligned}
\U_0(\xi)&\coloneq F_M^{\mathrm i}(\xi)+\varepsilon p_1(\xi),
&&\text{for }\xi\in[0,1],\\
u_0(r)&\coloneq u_*(r)+\frac{\varepsilon}{\pi}p_1'(r^2),
&&\text{for }r\in[0,1].
\end{aligned}
\]
For sufficiently small $\varepsilon>0$, these are smooth positive data
of mass $M$ approaching $u_*$ in every $C^k$ norm.
The eigenvalue equation gives $p_1''(1)=0$, hence the Neumann condition.
Moreover, $-H_M^{\mathrm i}$ and $p_1$ are comparable to $\xi(1-\xi)$,
as in the proof of \cref{lem:increasing-local-control}, so
$\U_0\le M\xi$ for small $\varepsilon$.

For the corresponding solution, set $q=u-u_*$ and
$P=\U-F_M^{\mathrm i}$. By \cref{lem:comparison}, $P\ge0$, while
\cref{lem:increasing-local-control} gives
$\norm{q(t)}_\infty\le C\varepsilon e^{-\gamma_Mt}$ for $t\ge0$.
Equation~\eqref{eq:perturbation} and $P'=\pi q(\sqrt\xi,t)$ yield
\begin{equation*}
P_t=LP+Pq(\sqrt\xi,t),\qquad P(0,t)=P(1,t)=0.
\end{equation*}
Thus $\alpha(t)\coloneq\langle P(t),p_1\rangle_{\mathcal H}$ satisfies
\[
\alpha'=-\gamma_M\alpha+
\langle Pq,p_1\rangle_{\mathcal H}
\ge-(\gamma_M+\norm{q(t)}_\infty)\alpha,\qquad
\alpha(0)=\varepsilon,
\]
by self-adjointness and positivity of $P,p_1$. Integration gives
\[
e^{\gamma_Mt}\alpha(t)
\ge\varepsilon\exp\!\left(-\int_0^t\norm{q(s)}_\infty\dd s\right)
\ge\varepsilon e^{-C\varepsilon/\gamma_M},\qquad\text{for }t\in[0,\infty).
\]
Finally, equal mass gives
$|P(\xi,t)|\le\pi\norm{q(t)}_\infty\min\{\xi,1-\xi\}$.
Since $m(\xi)$ is comparable to $1/\xi$, Cauchy--Schwarz implies
$\alpha(t)\le\norm{P(t)}_{\mathcal H}\le C\norm{q(t)}_\infty$.
Consequently,
\begin{equation*}
\liminf_{t\to\infty}e^{\gamma_Mt}
\norm{u(t)-u_M^{\mathrm i}}_\infty
\ge c\varepsilon e^{-C\varepsilon/\gamma_M}>0.
\end{equation*}
This excludes every faster exponent $\gamma_M+\delta$, $\delta>0$,
even with a constant depending on the initial datum.
Since these perturbations also satisfy $\U_0\le M\xi$, optimality
holds within the stated ordered class as well.
\end{proof}

\begin{proof}[Proof of \cref{thm:rates}]
\cref{prop:rate} gives \eqref{eq:exact-decay}, and
\cref{prop:optimal} proves optimality of the constant-state
exponent. At $M=\Lambda\pi$, \cref{prop:critical-rate} gives
\eqref{eq:critical-decay} and the sharper expansion
\eqref{eq:critical-expansion}, which implies \eqref{eq:critical-profile}.
\cref{lem:increasing-local-control} gives the uniform local
estimate \eqref{eq:increasing-local-rate}.
\cref{prop:increasing-rate} proves \eqref{eq:increasing-rate}
for every bounded trajectory converging to an increasing equilibrium
and establishes optimality of its spectral exponent.
\end{proof}

\providecommand{\bysame}{\leavevmode\hbox to3em{\hrulefill}\thinspace}
\providecommand{\MR}{\relax\ifhmode\unskip\space\fi MR }
\providecommand{\MRhref}[2]{%
  \href{http://www.ams.org/mathscinet-getitem?mr=#1}{#2}
}
\providecommand{\href}[2]{#2}


\begin{thebibliography}{10}

\bibitem{Ahn2026}
Jaewook Ahn, Jae-Myoung Kim, and Junha Kim, \emph{Stability and instability for
  the fully parabolic {K}eller-{S}egel system around constant equilibrium},
  \href{https://arxiv.org/abs/2607.10384}{arXiv:2607.10384 [math.AP]}, 2026.

\bibitem{Bonheure2021}
Denis Bonheure, Jean-Baptiste Casteras, and Carlos Rom\'{a}n, \emph{Unbounded
  mass radial solutions for the {K}eller-{S}egel equation in the disk}, Calc.
  Var. Partial Differential Equations \textbf{60} (2021), no.~5, Paper No. 198,
  30. \MR{4299895}

\bibitem{Cao2015}
Xinru Cao, \emph{Global bounded solutions of the higher-dimensional
  {K}eller-{S}egel system under smallness conditions in optimal spaces},
  Discrete Contin. Dyn. Syst. \textbf{35} (2015), no.~5, 1891--1904.
  \MR{3294230}

\bibitem{Childress1981}
S.~Childress and J.~K. Percus, \emph{Nonlinear aspects of chemotaxis}, Math.
  Biosci. \textbf{56} (1981), no.~3-4, 217--237. \MR{632161}

\bibitem{Cieslak2009}
Tomasz Cie\'{s}lak and Philippe Lauren\c{c}ot, \emph{Finite time blow-up for
  radially symmetric solutions to a critical quasilinear
  {S}moluchowski-{P}oisson system}, C. R. Math. Acad. Sci. Paris \textbf{347}
  (2009), no.~5-6, 237--242. \MR{2537529}

\bibitem{Cygan2021}
Szymon Cygan, Grzegorz Karch, Krzysztof Krawczyk, and Hiroshi Wakui,
  \emph{Stability of constant steady states of a chemotaxis model}, J. Evol.
  Equ. \textbf{21} (2021), no.~4, 4873--4896. \MR{4350589}

\bibitem{Daners2000}
Daniel Daners, \emph{Heat kernel estimates for operators with boundary
  conditions}, Math. Nachr. \textbf{217} (2000), 13--41. \MR{1780769}

\bibitem{Pino2016}
Manuel del Pino, Angela Pistoia, and Giusi Vaira, \emph{Large mass boundary
  condensation patterns in the stationary {K}eller-{S}egel system}, J.
  Differential Equations \textbf{261} (2016), no.~6, 3414--3462. \MR{3527634}

\bibitem{Pino2006}
Manuel del Pino and Juncheng Wei, \emph{Collapsing steady states of the
  {K}eller-{S}egel system}, Nonlinearity \textbf{19} (2006), no.~3, 661--684.
  \MR{2209293}

\bibitem{Horstmann2011}
Dirk Horstmann and Marcello Lucia, \emph{Uniqueness and symmetry of equilibria
  in a chemotaxis model}, J. Reine Angew. Math. \textbf{654} (2011), 83--124.
  \MR{2795753}

\bibitem{Jaeger1992}
W.~J\"{a}ger and S.~Luckhaus, \emph{On explosions of solutions to a system of
  partial differential equations modelling chemotaxis}, Trans. Amer. Math. Soc.
  \textbf{329} (1992), no.~2, 819--824. \MR{1046835}

\bibitem{Jiang2020}
Jie Jiang, \emph{Global stability of {K}eller-{S}egel systems in critical
  {L}ebesgue spaces}, Discrete Contin. Dyn. Syst. \textbf{40} (2020), no.~1,
  609--634 (English). \MR{4026974}

\bibitem{Kim2012}
Inwon Kim and Yao Yao, \emph{The {P}atlak-{K}eller-{S}egel model and its
  variations: properties of solutions via maximum principle}, SIAM Journal on
  Mathematical Analysis \textbf{44} (2012), no.~2, 568--602. \MR{2914242}

\bibitem{Lin1988}
C.-S. Lin, W.-M. Ni, and I.~Takagi, \emph{Large amplitude stationary solutions
  to a chemotaxis system}, J. Differential Equations \textbf{72} (1988), no.~1,
  1--27. \MR{929196}

\bibitem{Mao2025}
Xuan Mao, Meng Liu, and Yuxiang Li, \emph{Critical mass for finite-time
  chemotactic collapse in the critical dimension via comparison}, Nonlinearity
  \textbf{38} (2025), no.~9, Paper No. 095026. \MR{4961762}

\bibitem{Mao2026a}
\bysame, \emph{A note on the {$8\pi$} problem of {J}\"ager-{L}uckhaus system},
  Proc. Amer. Math. Soc. \textbf{154} (2026), no.~12, 5213--5231. \MR{5136227}

\bibitem{Nagai1995}
Toshitaka Nagai, \emph{Blow-up of radially symmetric solutions to a chemotaxis
  system}, Adv. Math. Sci. Appl. \textbf{5} (1995), no.~2, 581--601.
  \MR{1361006}

\bibitem{Nagai2001}
\bysame, \emph{Blowup of nonradial solutions to parabolic-elliptic systems
  modeling chemotaxis in two-dimensional domains}, J. Inequal. Appl. \textbf{6}
  (2001), no.~1, 37--55. \MR{1887324}

\bibitem{Painleve1900}
Paul Painlev{\'e}, \emph{Analyse des travaux scientifiques jusqu'en 1900},
  Albert Blanchard, Paris, 1967 (French), Reprint of the 1900 text.

\bibitem{Pan2022}
Hongjing Pan and Shiyi Ye, \emph{An inequality of {B}essel functions and
  applications to transcritical bifurcation problems of nonlinear elliptic
  equations}, J. Differential Equations \textbf{341} (2022), 657--674.
  \MR{4493141}

\bibitem{Pistoia2015}
Angela Pistoia and Giusi Vaira, \emph{Steady states with unbounded mass of the
  {K}eller-{S}egel system}, Proc. Roy. Soc. Edinburgh Sect. A \textbf{145}
  (2015), no.~1, 203--222. \MR{3304582}

\bibitem{Schaaf1985}
Renate Schaaf, \emph{Stationary solutions of chemotaxis systems}, Trans. Amer.
  Math. Soc. \textbf{292} (1985), no.~2, 531--556. \MR{808736}

\bibitem{Senba2000}
Takasi Senba and Takashi Suzuki, \emph{Some structures of the solution set for
  a stationary system of chemotaxis}, Adv. Math. Sci. Appl. \textbf{10} (2000),
  no.~1, 191--224. \MR{1769174}

\bibitem{Simon2005}
Barry Simon, \emph{Sturm oscillation and comparison theorems},
  Sturm-{L}iouville theory, Birkh\"auser, Basel, 2005, pp.~29--43. \MR{2145076}

\bibitem{Souplet2019}
Philippe Souplet and Michael Winkler, \emph{Blow-up profiles for the
  parabolic-elliptic {K}eller-{S}egel system in dimensions {$n\geq 3$}}, Comm.
  Math. Phys. \textbf{367} (2019), no.~2, 665--681. \MR{3936129}

\bibitem{Teschl2012}
Gerald Teschl, \emph{Ordinary differential equations and dynamical systems},
  Graduate Studies in Mathematics, vol. 140, American Mathematical Society,
  Providence, RI, 2012. \MR{2961944}

\bibitem{Wang2002}
G.~Wang and J.~Wei, \emph{Steady state solutions of a reaction-diffusion system
  modeling chemotaxis}, Math. Nachr. \textbf{233/234} (2002), 221--236
  (English). \MR{1879873}

\bibitem{Wang2019}
Jun Wang, Zhian Wang, and Wen Yang, \emph{Uniqueness and convergence on
  equilibria of the {K}eller-{S}egel system with subcritical mass}, Comm.
  Partial Differential Equations \textbf{44} (2019), no.~7, 545--572.
  \MR{3949126}

\bibitem{Wang2026}
Yulan Wang and Michael Winkler, \emph{Identifying secondary critical mass
  levels in two- and higher-dimensional {K}eller-{S}egel systems},
  \href{https://doi.org/10.48550/arXiv.2610.13559}{arXiv:2610.13559 [math.AP]},
  2026.

\bibitem{Watson1944}
G.~N. Watson, \emph{A {T}reatise on the {T}heory of {B}essel {F}unctions},
  Cambridge University Press, Cambridge; The Macmillan Company, New York, 1944.
  \MR{10746}

\bibitem{Winkler2010}
Michael Winkler, \emph{Aggregation vs. global diffusive behavior in the
  higher-dimensional {K}eller-{S}egel model}, J. Differential Equations
  \textbf{248} (2010), no.~12, 2889--2905. \MR{2644137}

\bibitem{Winkler2019}
\bysame, \emph{How unstable is spatial homogeneity in {K}eller-{S}egel systems?
  {A} new critical mass phenomenon in two- and higher-dimensional
  parabolic-elliptic cases}, Math. Ann. \textbf{373} (2019), no.~3-4,
  1237--1282. \MR{3953127}

\end{thebibliography}
\end{document}